\documentclass[amssymb,12pt]{amsart}
\usepackage{latexsym}
\newtheorem{propo}{Proposition}[section]

\newtheorem{lemma}[propo]{Lemma}
\newtheorem{corol}[propo]{Corollary}

\newtheorem{theor}[propo]{Theorem}

\newtheorem{remar}[propo]{Remark}

\newcommand{\Ker}{\operatorname{Ker}}
\newcommand{\Aut}{{\mathrm {Aut}}}
\newcommand{\Out}{{\mathrm {Out}}}
\newcommand{\Mult}{{\mathrm {Mult}}}

\newcommand{\Irr}{{\mathrm {Irr}}}
\newcommand{\IBR}{{\mathrm {IBr}}}
\newcommand{\IBRL}{{\mathrm {IBr}}_{\ell}}

\newcommand{\Ind}{{\mathrm {Ind}}}
\newcommand{\diag}{{\mathrm {diag}}}
\newcommand{\soc}{{\mathrm {soc}}}
\newcommand{\End}{{\mathrm {End}}}

\newcommand{\Hom}{{\mathrm {Hom}}}

\newcommand{\Mat}{{\mathrm {Mat}}}

\newcommand{\Sym}{{\mathrm {Sym}}}
\newcommand{\Char}{{\mathrm {char}}}

\newcommand{\Ext}{{\mathrm {Ext}}}
\newcommand{\CC}{{\mathbb C}}
\newcommand{\RR}{{\mathbb R}}
\newcommand{\QQ}{{\mathbb Q}}
\newcommand{\ZZ}{{\mathbb Z}}

\newcommand{\SSS}{{\sf S}}
\newcommand{\SK}{\Sym^{k}}
\newcommand{\SM}{\Sym^{m}}
\newcommand{\WK}{\wedge^{k}}
\newcommand{\WM}{\wedge^{m}}
\newcommand{\WB}{\wedge^{2}}
\newcommand{\WD}{\wedge^{3}}
\newcommand{\WE}{\wedge^{4}}
\newcommand{\WF}{\wedge^{5}}
\newcommand{\SB}{\Sym^{2}}
\newcommand{\SD}{\Sym^{3}}
\newcommand{\SE}{\Sym^{4}}
\newcommand{\SF}{\Sym^{5}}
\newcommand{\AAA}{{\sf A}}
\newcommand{\FF}{{\mathbb F}}
\newcommand{\FQ}{\mathbb{F}_{q}}
\newcommand{\FQB}{\mathbb{F}_{q}^{\times}}

\newcommand{\FB}{\FF^{\times}}

\newcommand{\kn}{\kappa_{n}}

\newcommand{\GC}{\mathcal{G}}
\newcommand{\BQ}{\bar{\QQ}}
\newcommand{\GCC}{\mathcal{G}_{\CC}}

\newcommand{\HCQ}{\mathcal{H}_{\BQ}}
\newcommand{\GN}{G_{{\bf n}}}
\newcommand{\GNC}{\GN^{\circ}}
\newcommand{\HC}{\mathcal{H}}
\newcommand{\LC}{\mathcal{L}}
\newcommand{\SC}{\mathcal{S}}
\newcommand{\TC}{\mathcal{T}}
\newcommand{\UC}{\mathcal{U}}

\newcommand{\DC}{\mathcal{D}}

\newcommand{\eps}{\epsilon}
\newcommand{\lam}{\lambda}
\newcommand{\Lam}{\Lambda}

\newcommand{\gam}{\gamma}
\newcommand{\om}{\varpi}
\newcommand{\etab}{\bar{\eta}}
\newcommand{\xib}{\bar{\xi}}
\newcommand{\Om}{\Omega}
\newcommand{\OV}{{\Omega(V)}}
\newcommand{\OSV}{{\Omega^{*}(V)}}
\newcommand{\VL}{{V_{\lambda}}}
\newcommand{\VLB}{{V_{\lambda^{-1}}}}
\newcommand{\dvl}{d_{\lambda}(V)}
\newcommand{\dwl}{d_{\lambda^{-1}}(V)}
\newcommand{\bl}{\bar{L}}
\newcommand{\bm}{\bar{M}}
\newcommand{\la}{\langle}
\newcommand{\ra}{\rangle}
\newcommand{\EL}{E_{\lambda}}
\newcommand{\ELB}{E_{\lambda^{-1}}}
\newcommand{\ELC}{E^{*}_{\lambda}}
\newcommand{\QL}{Q_{\lambda}}
\newcommand{\SN}{\SSS_{n}}
\newcommand{\AN}{\AAA_{n}}
\newcommand{\HS}{\hat{\SSS}}
\newcommand{\TS}{\tilde{\SSS}}
\newcommand{\HA}{\hat{\AAA}}

\newcommand{\TSN}{\tilde{\SSS}_{n}}
\newcommand{\HAN}{\HA_{n}}
\newcommand{\DA}{D^{1}}
\newcommand{\DB}{D^{2}}
\newcommand{\DCN}{D^{3}_{n}}
\newcommand{\VC}{V_{\CC}}
\newcommand{\VR}{V_{R}}
\newcommand{\VQ}{V_{\BQ}}
\newcommand{\sta}{(\star)}
\newcommand{\cs}{({\bf S})}
\newcommand{\diam}{~^{(\diamondsuit)}}
\newcommand{\heart}{~^{(\heartsuit)}}
\newcommand{\spade}{~^{(\spadesuit)}}
\newcommand{\dl}{{\mathfrak d}}
\newcommand{\ml}{{\mathfrak m}}
\newcommand{\ts}{{\mathfrak t}}

\newcommand{\ta}{\hspace{0.5mm}^{2}\hspace*{-0.2mm}}
\newcommand{\tb}{\hspace{0.5mm}^{3}\hspace*{-0.2mm}}
\def\skipa{\vspace{-1.5mm} & \vspace{-1.5mm} & \vspace{-1.5mm}\\}
\renewcommand{\mod}{\bmod \,}

\begin{document}
\title
{Symmetric Powers and a Problem of Koll\'ar and Larsen}
\author{Robert M. Guralnick}
\address{Department of Mathematics, University of Southern California,
Los Angeles, CA 90089-1113, USA}
\email{guralnic@math.usc.edu}
\author{Pham Huu Tiep}
\address{Department of Mathematics, University of Florida, Gainesville,
FL 32611, USA}
\address{{\it Since Aug. 2008}: 
Department of Mathematics, University of Arizona, Tucson, AZ 85721, USA}
\email{tiep@math.arizona.edu}



\thanks{Part of this paper was written while the authors were participating
in the Workshop on Lie Groups, Representations and Discrete Mathematics at the
Institute for Advanced Study (Princeton). It is a pleasure to thank the Institute for its generous 
hospitality and support.} 

\thanks{The authors would like to thank J\'anos Koll\'ar for suggesting this problem to them and 
for insightful comments on the paper, Nolan Wallach for discussions about the 
decompositions for various tensor powers of the natural module for classical groups, 
J\"urgen M\"uller for proving some results about modular representations of
$\ta E_{6}(2)$, and Thomas Breuer, Gunter Malle, Frank L\"ubeck, and Alexandre Turull for their help 
with various computer calculations.}

\thanks{The authors gratefully acknowledge the support of the NSF (grants
DMS-0653873 and DMS-0600967), and of the NSA (grant H98230-04-0066).}

\maketitle

\section{Introduction}
Let $\FF$ be an algebraically closed field of characteristic $\ell \geq 0$, $V = \FF^{d}$, 
and let $\GC = GL(V)$, or $GO(V)$, resp. $Sp(V)$ (the full isometry group of a non-degenerate symmetric,
resp. alternating, bilinear form on $V$). In various applications, including in the 
classification of maximal subgroups of classical groups and in algebraic geometry, it is important 
to know which closed subgroups $G$ of $\GC$ can act irreducibly on $\SK(V)$ for some $k \geq 2$. The list 
of such subgroups $G$, under the assumption that $G$ is connected and positive dimensional, has been 
determined by Dynkin \cite{Dyn} in characteristic $0$ and by Seitz \cite{Se1} 
and Testerman \cite{Tes} in positive characteristic. 
A conjecture of Koll\'ar and Larsen \cite{KL} asserts that if $k$ is not too small, say $k \geq 4$, 
the complete list of such subgroups $G$ remains essentially the same when $G$ is assumed to be closed.
This conjecture has interesting implications, in particular on the holonomy group of a stable vector 
bundle on a smooth projective variety, cf. the very recent work of Balaji and Koll\'ar \cite{BK}. 
The main result of the paper proves this conjecture in the affirmative.

\begin{theor}\label{main}
{\sl Let $\FF$ be an algebraically closed field of characteristic $\ell \geq 0$ and $V = \FF^{d}$ with 
$d > 4$. Assume that a Zariski closed subgroup $G$ of $\GC := GL(V)$ acts irreducibly on $\SK(V)$ for 
some $k \geq 4$. Then either $\ell = 0$ or $\ell > k$. Moreover, one of the following holds.

{\rm (i)} $\HC \lhd G \leq N_{\GC}(\HC)$ with $\HC \in \{SL(V),Sp(V)\}$.

{\rm (ii)} $\ell > 0$, $L \lhd G \leq N_{\GC}(L)$, where $L$ is a quotient of $SL_{d}(q)$, $SU_{d}(q)$, or 
$Sp_{d}(q)$ for some power $q = \ell^{a}$.

{\rm (iii)} $k = 4,5$. Furthermore, $L \lhd G \leq N_{\GC}(L)$ with $(d,L) = (6,2J_{2})$, 
$(12,2G_{2}(4))$, $(12,6Suz)$.

{\rm (iv)} $k = 4,5$, $\ell = 5, 7$, $8900000 > d \geq 196882$, and $M \lhd G \leq N_{\GC}(M)$, 
where $M$ is the Monster sporadic finite simple group.\\
Conversely, the cases listed in {\rm (i) -- (iii)} give rise to examples.}
\end{theor}     

Observe that there are infinite series of examples of finite subgroups of $GL(V)$, not satisfying 
conclusions (ii)--(iv) of Theorem \ref{main} and such that $\SB(V)$ and $\SD(V)$ are irreducible over $G$,
cf. \cite{MT2}. 
Another curious example is that the subgroup 
$G_{2}(\CC)$ of $\GC = SO_{7}(\CC)$ is irreducible on all $\GC$-composition factors of $\SK(V)$ for 
all $k$, cf. \cite{Se1}. 

\medskip
The small dimensional case is handled by the following:

\begin{theor}\label{lowdim}
{\sl Let $\FF$ be an algebraically closed field of characteristic $\ell \geq 0$ and $V = \FF^{d}$ with 
$d \leq 4$. Assume that a Zariski closed subgroup $G$ of $\GC := GL(V)$ acts irreducibly on $\SK(V)$ for 
some $k \geq 4$. Then either $\ell = 0$ or $\ell > k$ or $d \leq 2$. Furthermore, one of the following holds.

{\rm (i)} $\HC \lhd G \leq N_{\GC}(\HC)$ with $\HC \in \{SL(V),Sp(V)\}$.

{\rm (ii)} $\ell > 0$, $L \lhd G \leq N_{\GC}(L)$, where $L$ is a quotient of $SL_{d}(q)$, $SU_{d}(q)$, or 
$Sp_{d}(q)$ for some power $q = \ell^{a}$.

{\rm (iii)} $\ell \neq 2$, $d = 2$, $G = Z(G) * SL_{2}(5)$. Furthermore, $k = 4,5$ if $\ell = 0$ or 
$\ell > 5$, $k = 4$ if $\ell = 5$, and $k = 5$ if $\ell = 3$. 

{\rm (iv)} $\ell = 5$, $d = 3$, $G = Z(G) * 3\AAA_{7}$, and $k = 4$.\\
Conversely, all the above cases give rise to examples.}
\end{theor}      

As shown in \cite{BK}, Theorems \ref{main} and \ref{lowdim} imply 

\begin{corol}\label{holonomy} {\rm \cite[Cor. 6]{BK}}
{\sl Let $E$ be a stable vector bundle on a complex smooth projective variety $X$ of rank different 
from $2$, $6$, $12$. Then the following are equivalent:

{\rm (i)} $\SK(E)$ is stable for some $k \geq 4$.

{\rm (ii)} $\SK(E)$ is stable for every $k \geq 4$.

{\rm (iii)} The commutator subgroup of the holonomy group is either $SL(E_{x})$ or $Sp(E_{x})$.
\hfill $\Box$ }
\end{corol}

The exceptions in rank $2$, $6$, and $12$ are related to the possibilities described in Theorem 
\ref{main}(iii) and Theorem \ref{lowdim}(iii).
 
\medskip
The main ideas of our proofs can be outlined as follows. Suppose a subgroup $G$ of $GL(V)$ satisfies 
the hypotheses of Theorem \ref{main}, resp. Theorem \ref{lowdim}. First, arguments along the lines of 
Aschbacher's Theorem \cite{A} reduce the problem to the cases where $G$ normalizes either a certain 
$p$-group for a prime $p$ dividing $d$, a simple algebraic group, or a finite (quasi)simple group $S$, cf. 
Proposition \ref{red}. The first case can be handled quickly using character-theoretic methods, see
Theorem \ref{extra}. In the second case, as well as in the third case with $S$ a finite group of Lie type 
defined in the same characteristic $\ell$ as of $\FF$, various tools from the (modular) representation 
theory of algebraic groups (cf. Theorem \ref{defi}) allow us to reduce to the case of connected reductive 
algebraic groups and then apply the classic results of Dynkin \cite{Dyn} and Seitz \cite{Se1}. The main 
obstacles arise in the third case and moreover when $S$ is not a finite group of Lie type in the same 
characteristic as of $\FF$. Unlike the situations considered previously in \cite{GT2} and \cite{MMT}, the 
irreducibility of $\SK(V)$ does not yield (nontrivial) upper bounds on $\dim(V)$ -- such a bound was the
crucial step in the mentioned papers. The key idea here is to show that $G$ possesses a large enough 
subgroup $C$ such that the restriction of $\SK(V)$ to $C$ contains a small enough submodule. Even though 
this argument does not yield an upper bound on $\dim(V)$, it does lead to a strong constraint on $G$ and 
some of its natural subgroups which ultimately yields a contradiction, cf. for instance Proposition 
\ref{key1}. The case when $S$ is a sporadic finite simple group also presents considerable difficulties since 
for some of them (say the Monster) there is only very scarce information about their modular representation
theory (and this is usually available only when the Sylow $\ell$-subgroups of $S$ are cyclic). As usual,
low dimensional representations such as Weil representations of finite classical groups and basic spin
representations of (double covers of) symmetric and alternating groups also require special treatment as 
well. In certain situations when $\ell$ is large enough, results of Serre \cite{S} and McNinch \cite{McN} 
allow one to reduce to the complex case.           
 
In this paper we also obtain various results concerning the reducibility of 
exterior powers $\wedge^{k}(V)$ as well. But, as the example of $\SN$ acting on the heart of the natural 
permutation module shows, the irreducibility of $\wedge^{k}(V)$ is not enough to tell apart $\GC$ from its 
finite closed subgroups, cf. also \cite{MMT}. In fact, as shown in Proposition \ref{red-a} and Theorem
\ref{extra}, a Zariski closed subgroup of $GL(V)$ with $\dim(V) \geq 6$ can be 
irreducible on $\WK(V)$ for some $k \geq 3$ only when either $G$ is almost quasi-simple
(i.e. $\soc(G/Z(G))$ is a simple, algebraic or finite, group) or $G$ stabilizes a 
decomposition of $V$ into $1$-spaces. We intend to fully investigate this question in a sequel of the 
paper. Here we will prove the following theorem

\begin{theor}\label{sym-alt}
{\sl Let $\FF$ be an algebraically closed field of characteristic $0$, $V = \FF^{d}$ with 
$d > 4$, and let $G$ be a Zariski closed subgroup of $GO(V)$. Assume $G$ does not contain $SO(V)$.
Then $G$ is either reducible on $\SE(V)/\SB(V)$, or on $\wedge^{4}(V)$ for $d > 7$, or on 
$\wedge^{2}(V)$ for $d \leq 7$.}
\end{theor}
  
(In the situation of this theorem, one can identify 
$\Sym^{k}(V)$ with the space of homogeneous polynomials of degree $k$ in $d$ variables. The 
$GO(V)$-invariant quadratic form on $V$ yields a $GO(V)$-invariant quadratic polylomial $Q$,
and the multiplication by $Q$ yields an embedding of $\SB(V)$ into $\SE(V)$.)
     
Theorem \ref{sym-alt} in particular yields another proof of Larsen's conjecture proved in 
\cite{GT2}. (Indeed, the proof in \cite{GT2} uses the irreducibility of $G$ on every 
$GO(V)$-composition factor of $V^{\otimes 4}$ to derive the containment $G \ge SO(V)$, whereas the new
proof, see Corollary \ref{larsen} and its proof below, uses $G$-irreducibility only on a few 
specific composition factors.)  
Larsen's conjecture has already been used by Katz, to study the monodromy group attached 
to a Lefschetz pencil of smooth hypersurface sections of a projective smooth variety $X$ over a finite 
field $k$ \cite{Ka1}, and to determine the geometric monodromy group attached to a family of character sums
over finite fiels \cite{Ka2}. It also has implications on the holonomy group of a stable vector bundle on 
a complex smooth projective variety, cf. \cite{BK}. 

\begin{corol}\label{larsen} {\rm (Larsen's conjecture)} 
{\sl Let $\FF$ be an algebraically closed field of characteristic $0$, $V = \FF^{d}$ and let 
$\GC = GL(V)$, $GO(V)$, or $Sp(V)$. If $d \leq 4$, assume in addition that $\GC \neq GO(V)$.
Let $G$ be a Zariski closed subgroup of $\GC$ such that $G^{\circ}$ is reductive and $G$ does not 
contain $[\GC,\GC]$. Then one of the following holds.

{\rm (i)} $\dim(\End_{G}(V^{\otimes 4})) > \dim(\End_{\GC}(V^{\otimes 4}))$. 

{\rm (ii)} $d = 6$, $\GC = Sp(V)$, and $G = 2J_{2}$.

{\rm (iii)} $d = 2$, $\GC = GL(V)$, and $G = Z(G)*SL_{2}(5)$.}
\end{corol}     

Notice that we do not consider orthogonal groups in dimensions $\leq 4$ in Theorem \ref{sym-alt} and Corollary 
\ref{larsen} because $SO_{d}(\FF)$ is not simple when $d = 1,2,4$ and isomorphic to $PSL_{2}(\FF)$ when
$d = 3$. We also obtain the following variant of Corollary \ref{sym2}(i) which holds in almost every 
characteristic:

\begin{corol}\label{sym3}
{\sl Let $G$ be a closed subgroup of $GL(V)$, with $V = \FF^{d}$ and $d \geq 3$ if $\Char(\FF) > 0$. 
Assume that the $G$-module $\SK(V)$ is irreducible for some $k \geq 4$. If $k \leq 5$ and 
$M \lhd G \leq N_{GL(V)}(M)$, $M$ being the Monster, assume furthermore that either $\Char(\FF) \neq 5,7$
or $\dim(V) \geq 8900000$. Then for every $m$, $1 \leq m \leq k$, the $G$-module $\SM(V)$ is also 
irreducible.}
\end{corol}

Corollary \ref{sym3} cannot hold for $d = 2$ when $\Char(\FF) > 0$, cf. Remark \ref{d=2} below. 

\medskip
Throughout the paper, we use the convention that $\ell > N$ means that either $\ell = 0$ or $\ell > N$.
The notation for simple groups is as in \cite{Atlas}; in particular, $M$ is the Monster, $6Suz$ is the 
sixth cover of the Suzuki group, and $2J_{2}$ is the double cover of the second Janko group. $\SN$, resp. 
$\AN$ is the symmetric, resp. alternating, group on $n$ symbols. $G * H$ denotes a central product of finite 
groups $G$ and $H$, and $G^{(\infty)}$ is the last term of the derived series of $G$. We will assume that 
$\ell > k$ whenever we address the irreducibility of $\SK(V)$ with $\dim(V) \geq 3$, cf. Lemma 
\ref{step1}(i). If $G$ is a closed subgroup of $GL(V)$ then $G^{\circ}$ denotes the connected component
of $G$, and the irreducible $G$-module with highest weight $\om$ is denoted by $L(\om)$. If $G$ is a 
finite group and $\chi$ a class function of $G$ then $\hat{\chi}$ denotes the restriction of $\chi$ to 
$\ell'$-elements in $G$;
furthermore, $\IBRL(G)$ denotes the set of all irreducible $\ell$-modular Brauer characters of 
$G$. A $G$-module (over field of characteristic $\neq 2$) is said to be of type $+$, resp.
$-$, if it supports a nondegenerate $G$-invariant symmetric, resp. alternating form; the same for 
(irreducible) ordinary or Brauer characters of $G$.
If $N \lhd G$, then a simple $N$-module $V$ is $G$-invariant if $V^{g} \simeq V$ for all $g \in G$.  
If $G$ is any finite group, then $\ml(G)$ denotes the largest degree of complex irreducible
characters of $G$; clearly, $\ml(G) \leq \sqrt{|G/Z(G)|}$. If $G$ is a finite group of Lie 
type in characteristic $p$, then $\dl(G)$ denotes the smallest degree of nontrivial projective 
representations of $G$ in characteristic other than $p$. We will freely use the 
Landazuri-Seitz-Zalesskii lower bounds on $\dl(G)$ and their latest improvements as recorded in 
\cite{T2}, and the upper bound for $\ml(G)$ as given in \cite{Se2}.

\section{Preliminaries}
Recall that $\FF$ is an algebraically closed field of characteristic $\ell \geq 0$.
Let $V = \FF^{d}$ and let $\GC = GL(V)$, $Sp(V)$, or $GO(V)$ throughout this section. 
We consider $V$ as the irreducible $\GC$-module with highest weight $\om_{1}$.

\subsection{Basic reductions.}
To get some basic reductions for our problem, one might apply the fundamental result of 
Aschbacher \cite{A}. But in our case one can give a direct argument (which in fact goes 
along the lines of the proof of Aschbacher's Theorem, and which also gives us some further
information that will be needed later). The first step is to reduce to the case where the subgroup
$G \leq GL(V)$ satisfies the following hypothesis:
$$\cs~:~\begin{array}{l}\mbox{The }G\mbox{-module }V \mbox{ is irreducible, primitive,}
        \\ \mbox{tensor indecomposable, and not tensor induced.}\end{array} \hspace{4cm}$$ 
(Recall that the $G$-representation $\Phi$ of $G$ on $V$ is {\it tensor induced}, if there is a 
decomposition $V = V_{1} \otimes \ldots \otimes V_{m}$ into $m > 1$ vector spaces $V_{i}$ of equal 
dimension, such that $\Phi(G) \leq (\otimes^{m}_{i=1}GL(V_{i}))\cdot \SSS_{m}$, with $\SSS_{m}$ 
naturally permuting the spaces $V_{i}$.)
  
\begin{lemma}\label{step1}
{\sl Assume $G \leq GL(V)$, $\dim(V) \geq 2$, and $\SK(V)$ is irreducible over $G$ for some $k \geq 2$. 

{\rm (i)} If $\dim(V) \geq 3$ then $\Char(\FF) > k$.

{\rm (ii)} $G$ satisfies the hypothesis $\cs$.}
\end{lemma}

\begin{proof}
(i) Assume that $d := \dim(V) \geq 3$ but $0 < \ell = \Char(\FF) \leq k$. We will show that 
$\SK(V)$ is reducible over $SL(V)$. Let $0 \leq a \leq \ell-1$ and $b, i > 0$ be any integers. Then 
notice that $(a+b\ell+i)/i \geq (a+i)(b+i)/i^{2}$ and in fact the inequality is strict if 
$i \geq 2$. Taking $i = 1, \ldots, d-1$, we see that 
\begin{equation}\label{abl}
  \dim(\Sym^{a+b\ell}(V)) > \dim(\Sym^{a}(V)) \cdot \dim(\Sym^{b}(V)).
\end{equation} 
Write $k = a_{0} + a_{1}\ell + \ldots + a_{s}\ell^{s}$ for some integers 
$0 \leq a_{0}, \ldots ,a_{s} \leq \ell-1$ and $a_{s} > 0$. Using induction on $s \geq 1$ with 
(\ref{abl}) as induction base, one can show that 
\begin{equation}\label{abls}
  \dim(\Sym^{k}(V)) > \dim(\Sym^{a_{0}}(V)) \cdot \dim(\Sym^{a_{1}}(V)) \cdot \ldots 
  \cdot \dim(\Sym^{a_{s}}(V)).
\end{equation}
Now if the $SL(V)$-module $V$ has highest weight $\om_{1}$, then $\SK(V)$ has highest weight 
$k\om_{1}$, whence it has a quotient isomorphic to 
$\Sym^{a_{0}}(V) \otimes (\Sym^{a_{1}}(V))^{(\ell)} \otimes \ldots \otimes (\Sym^{a_{s}}(V))^{(\ell^{s})}$ 
by Steinberg's tensor 
product theorem. Hence (\ref{abls}) implies that $\SK(V)$ is reducible. 

(ii) Assume that $G$ is reducible on $V$ and $A \neq 0$ is a 
proper $G$-submodule of $V$. Then $\SK(A) \neq 0$ is a proper $G$-submodule of $\SK(V)$, a contradiction.
Next assume that $G$ is imprimitive on $V$. Then $V = \oplus^{n}_{i=1}V_{i}$ 
with $G$ permuting the subspaces $V_{i}$'s transitively. It follows that 
$\oplus^{n}_{i=1}\SK(V_{i})$ is a proper $G$-submodule of $\SK(V)$, a contradiction.

Now assume that the $G$-module $V$ is tensor decomposable. Then $V = A  \otimes B$ as a 
$G$-module, with $\dim A, \dim B > 1$. In particular $d \geq 4$ and so in view of (i) we may assume 
that $\ell > k$. Under this assumption on $\ell$, $\SK(V)$ is just the 
fixed point subspace for $S_{k}$ with $S_{k}$ naturally permuting the $k$ factors of 
$V^{\otimes k}$. Since $\SK(A) \otimes \SK(B)$ is fixed by $S_{k}$ pointwise, it is a proper 
$G$-submodule of $\SK(V)$, again a contradiction.

Finally, assume that the $G$-module $V$ is tensor induced. In this case, 
$V = \otimes^{n}_{i=1}V_{i}$ with $G$ permuting the subspaces $V_{i}$ 
transitively. Again $d \geq 4$ and so we may assume $\ell > k$. Hence, $\SK(V)$ contains the proper 
$G$-submodule $\otimes^{n}_{i=1}\SK(V_{i})$, a contradiction.
\end{proof}

\begin{remar}\label{d=2}
{\em {\sl Notice that Lemma \ref{step1}(i) and Corollary \ref{sym3} fail if $\dim(V) \leq 2$.} In fact, 
if $\dim(V) = 2$ then for any integers $j \geq 1$ and $1 \leq b < \ell = \Char(\FF)$, 
$$\Sym^{(b+1)\ell^{j}-1}(V) = \Sym^{\ell-1}(V) \otimes \Sym^{\ell-1}(V)^{(\ell}) \otimes \ldots 
  \otimes \Sym^{\ell-1}(V)^{(\ell^{j-1})} \otimes \Sym^{b}(V)^{(\ell^{j})}$$
is irreducible over $SL(V)$ (as well as over any $SL_{2}(\ell^{n})$ with $n > j$). In particular,
$\Sym^{\ell^{j}}(V)$ is reducible but $\Sym^{2\ell^{j}-1}(V)$ is irreducible over $SL(V)$. Another example is
$G = SL_{2}(5)$ in $SL(V)$ with $\ell = 3$ and $d = 2$: here $\SF(V)$ is irreducible but $\SE(V)$ is 
reducible.}
\end{remar}

\begin{remar}\label{schur}
{\em (i) {\sl Notice that $\SK(V^{*}) \simeq (\SK(V))^{*}$ if $\Char(\FF) > k$ and 
$\WK(V^{*}) \simeq (\WK(V))^{*}$ if $1 \leq k \leq \dim(V)-1$, as modules over $GL(V)$.} Indeed, all these
modules are irreducible over $SL(V)$, and we can see the isomorphism by comparing their highest weights.

(ii) {\sl Assume $V$ and $W$ are $\FF G$-spaces and $\Char(\FF) > k$. Then  
$$\SK(V \oplus W) \simeq \bigoplus_{i+j = k}\Sym^{i}(V) \otimes \Sym^{j}(W),~~~
  \WK(V \oplus W) \simeq \bigoplus_{i+j = k}\wedge^{i}(V) \otimes \wedge^{j}(W),$$ 
and 
$$\SK(V \oplus W) \simeq \bigoplus (\SSS_{\lam}V \otimes \SSS_{\lam}W),~~~
  \WK(V \oplus W) \simeq \bigoplus (\SSS_{\lam}V \otimes \SSS_{\lam'}W),$$
where the first sum runs over all partitions $\lam$ of $k$ with at most $\dim(V)$ or $\dim(W)$ rows, and 
the second sum runs over all partitions $\lam$ of $k$ with at most $\dim(V)$ rows and at most $\dim(W)$ 
columns}, cf. \cite[p. 80]{FH}. Here $\lam'$ is the partition conjugate to $\lam$ and $\SSS_{\lam}$ is the 
Schur functor corresponding to the partition $\lam$. Indeed, the proof of these formulas in 
\cite[p. 521]{FH} uses only the semisimplicity of the group algebra $\FF S_{k}$.}
\end{remar}
 
\begin{lemma}\label{step1a}
{\sl Assume $d \geq 3$, $\Char(\FF) > 3$, $G \leq GL(V)$ and $\WB(V)$ is irreducible over $G$. 
Then one of the following holds.

{\rm (i)} $G$ satisfies the hypothesis $\cs$.

{\rm (ii)} $V = \oplus^{d}_{i=1}V_{i}$ is a sum of $1$-spaces and $G$ acts $2$-homogeneously on
$\{V_{1}, \ldots ,V_{d}\}$.

{\rm (iii)} $V = V_{1} \otimes V_{2}$ with $G$ permuting $V_{1}$ and $V_{2}$ transitively, and $\SB(V_{i})$ 
and $\WB(V_{i})$ are irreducible over $G_{1} := Stab_{G}(V_{1}) \cap Stab_{G}(V_{2})$ for $i = 1,2$. 
Furthermore, $G$ is reducible on $\WD(V)$ if $d \geq 5$.}
\end{lemma}

\begin{proof}
Assume that $G$ is reducible on $V$ and $A \neq 0$ is a proper $G$-submodule of $V$. Replacing $V$ by 
$V^{*}$ if necessary, we may assume that $\dim(A) \geq 2$. Then $\WB(A) \neq 0$ is a proper $G$-submodule 
of $\WB(V)$, a contradiction. Next assume that $G$ is imprimitive on $V$. Then $V = \oplus^{n}_{i=1}V_{i}$ 
with $G$ permuting the subspaces $V_{i}$'s transitively. If $\dim(V_{i}) \geq 2$, then 
$\oplus^{n}_{i=1}\WB(V_{i})$ is a proper $G$-submodule of $\WB(V)$, a contradiction. If 
$\dim(V_{i}) = 1$, then the irreducibility of $\WB(V)$ implies that the $T$-module $V_{i}$'s are all 
nonisomorphic $T$-modules for $T := \cap^{d}_{i=1}Stab_{G}(V_{i})$, and the permutation action of $G$ on 
$\{V_{1}, \ldots ,V_{d}\}$ is $2$-homogeneous.

Now assume that the $G$-module $V$ is tensor decomposable. Then $V = A  \otimes B$ as a 
$G$-module, with $\dim A, \dim B > 1$. Clearly, $\SB(A) \otimes \WB(B)$ is a proper 
$G$-submodule of $\WB(V)$, again a contradiction.

Finally, assume that the $G$-module $V$ is tensor induced. In this case, 
$V = \otimes^{n}_{i=1}V_{i}$ with $G$ permuting the subspaces $V_{i}$'s 
transitively, and $\dim(V_{i}) \geq 2$. Hence, $\WB(V)$ contains the proper $G$-submodule 
$$\sum^{n}_{i=1}\left( \SB(V_{1}) \otimes \ldots \otimes \SB(V_{i-1})
  \otimes \WB(V_{i}) \otimes \SB(V_{i+1}) \otimes \ldots \otimes \SB(V_{n})\right),$$ 
a contradiction, unless $n = 2$. Consider the case $n = 2$. Since 
$$\WB(V_{1} \otimes V_{2}) \simeq (\SB(V_{1}) \otimes \WB(V_{2})) \oplus (\WB(V_{1}) \otimes \SB(V_{2}))$$
as $G_{1}$-modules and $(G:G_{1}) = 2$, $\SB(V_{i})$ 
and $\WB(V_{i})$ must be irreducible over $G_{1}$. Finally, assume $d \geq 5$. Then 
$\dim(V_{1}) = \dim(V_{2}) \geq 3$, and 
$$\WD(V_{1} \otimes V_{2}) \simeq (\SD(V_{1}) \otimes \WD(V_{2})) \oplus (\WD(V_{1}) \otimes \SD(V_{2}))
  \oplus (\SSS_{(2,1)}(V_{1}) \otimes \SSS_{(2,1)}(V_{2}))$$
as $G_{1}$-modules, whence $\WD(V)$ is reducible over $G$.    
\end{proof}

Note that both exceptions listed in Lemma \ref{step1a}(ii), (iii) do occur.

\begin{lemma}\label{step2}
{\sl Assume $G \leq GL(V)$ is an irreducible, primitive and tensor indecomposable subgroup, and $H \lhd G$.
Then $H$ is either central in $G$ or irreducible on $V$.}
\end{lemma}

\begin{proof}
By Clifford theory, the $H$-module $V$ has only one isotypic component, i.e.
$V|_{H} = eW$ with $W$ an irreducible $H$-module. If $f := \dim(W) = 1$ then $H \leq Z(G)$. 
So assume $e,f > 1$. Let $\Phi$, resp.
$\Psi$, denote the representation of $G$ on $V$, resp. of $H$ on $W$. In a suitable basis
of $V$, $\Phi(g) = (g_{ij})_{1 \leq i,j \leq e}$ with $g_{ij} \in \Mat_{f}(\FF)$ and 
$\Phi(h) = \diag(\Psi(h), \ldots ,\Psi(h))$ for $g \in G$ and $h \in H$. Since 
$\Psi^{g} \simeq \Psi$, $\Psi(ghg^{-1}) = \Theta(g)\Psi(h)\Theta(g)^{-1}$ for some 
$\Theta(g) \in GL_{f}(\FF)$. The identity 
$\Phi(g)\Phi(h)\Phi(g)^{-1} = \Phi(ghg^{-1})$ now implies that 
$\Theta(g)^{-1}g_{ij}$ commutes with $\Psi(h)$ for all $h \in H$. By Schur's Lemma,
$g_{ij} = \Lam_{ij}(g)\Theta(g)$ for some $\Lam_{ij}(g) \in \FB$. Thus 
$\Phi(g) = \Lam(g) \otimes \Theta(g)$ if we set 
$\Lam(g) := (\Lam_{ij}(g))_{1 \leq i,j \leq e} \in \Mat_{e}(\FF)$. In particular, 
$\Lam(g) \in GL_{e}(\FF)$. We have shown that 
$\Phi(G) \leq GL_{e}(\FF) \otimes GL_{f}(\FF)$. In other words, $G$ is contained in 
$GL(A) \otimes GL(B)$ for some decomposition $V = A \otimes B$ with $e = \dim(A)$ and 
$f = \dim(B)$, a contradiction. 
\end{proof}

\begin{lemma}\label{tensor1}
{\sl Let $S$ be a group, and let $V_{1}, \ldots ,V_{m}$ be $\FF S$-modules such that the 
resulting representation $\Phi$ of $S$ on 
$V := V_{1} \otimes \ldots \otimes V_{m}$ is irreducible. Let $g \in GL(V)$ be 
an element that normalizes $\Phi(S)$.  

{\rm (i)} If $V_{i}^{g} \simeq V_{i}$ for all $i$, then 
$g \in K := \otimes^{m}_{i=1}GL(V_{i})$.

{\rm (ii)} Assume $\dim(V_{i}) = f$ for all $i$, and that for every $i$ there is some $j$ such 
that $V_{i}^{g} \simeq V_{j}$. Then $g \in H := (\otimes^{m}_{i=1}GL(V_{i})) \cdot \SSS_{m}$.}
\end{lemma}

\begin{proof}
(i) Let $\Phi_{i}$ denote the representation of $S$ on $V_{i}$. By assumption, there is 
$h_{i} \in GL(V_{i})$ such that $h_{i}$ and $g$ induce the same automorphism on 
$\Phi_{i}(S)$. It follows that $h := \otimes^{m}_{i=1}h_{i} \in K$ 
and $g$ induce the same automorphism on $\Phi(S)$. Thus $h^{-1}g$ centralizes $\Phi(S)$ 
and so it is scalar by irreducibility, whence is in $K$, and so is $g$.

(ii) Let $\Phi_{i}$ denote the matrix representation of $S$ relative to a fixed basis of $V_{i}$. 
Without loss we may replace $S$ by $\Phi(S)$, and 
denote the matrix representation of $\la S,g \ra$ on $V$ also by $\Phi$.  
By assumption, there is an element $\tau \in \SSS_{m}$ and $h_{i} \in GL_{f}(\FF)$ such that 
$\Phi_{i}(gsg^{-1}) = h_{i}\Phi_{\tau(i)}(s)h_{i}^{-1}$ for all $s \in S$. 
We may find an element $\sigma$ of the subgroup $\SSS_{m}$ of $H$ such that 
$\sigma\Phi(s)\sigma^{-1} = \sigma(\otimes^{m}_{i=1}\Phi_{i}(s))\sigma^{-1} = 
 \otimes^{m}_{i=1}\Phi_{\tau(i)}(s)$ for all $s \in S$. Setting 
$h := \otimes^{m}_{i=1}h_{i} \in H$, we see that 
\begin{eqnarray*}
\Phi(g)\Phi(s)\Phi(g)^{-1}& = &\Phi(gsg^{-1}) =   \otimes^{m}_{i=1}h_{i}\Phi_{\tau(i)}(s)h_{i}^{-1} \\ 
 & = & h(\otimes^{m}_{i=1}\Phi_{\tau(i)}(s))h^{-1} = h\sigma\Phi(s)\sigma^{-1}h^{-1}, 
\end{eqnarray*}
whence $\sigma^{-1}h^{-1}\Phi(g)$ centralizes $\Phi(s)$ for all $s \in S$. It follows again 
by irreducibility that $\sigma^{-1}h^{-1}\Phi(g)$ is scalar, and so $\Phi(g) \in H$. 
\end{proof}

Slightly abusing the language, in the situations (i) or (ii) of Lemma \ref{tensor1} 
we will say that $g$ {\it permutes the spaces $V_{1}, \ldots ,V_{m}$}.
 
Lemma \ref{tensor1}(i) is not true without the assumption that $G$ permutes the set of isomorphism 
classes of $\FF S$-modules $V_{1}, \ldots ,V_{m}$. Indeed, the group $G = Sp_{2n}(5) \cdot 2$
has an irreducible complex representation $V$ such that 
$V|_{S} = A \otimes B = A' \otimes B'$ with $S = Sp_{2n}(5)$, $A,A'$ distinct irreducible 
(Weil) $S$-modules of dimension $(5^{n}-1)/2$ permuted by $G$, and $B,B'$ distinct irreducible 
(Weil) $S$-modules of dimension $(5^{n}+1)/2$ permuted by $G$, cf. \cite{MT1}. However, see 
\cite{Ra1} for an important case where tensor decomposition of a complex module is uniquely 
determined. 
  
\begin{corol}\label{tensor2}
{\sl Assume $G \leq GL(V)$ and a normal subgroup $S$ of $G$ is a central product
of subgroups $H_{1}, \ldots ,H_{m}$, which are permuted by $G$ via conjugation. Assume 
furthermore that $V|_{S} = V_{1} \otimes \ldots \otimes V_{m}$ is irreducible with each $H_{i}$ acting 
on $V_{i}$ and trivially on $V_{j}$ for $j \neq i$. Then $G$ also permutes the spaces 
$V_{1}, \ldots ,V_{m}$.}
\end{corol}

\begin{proof}
Notice that $V|_{H_{i}}$ is a direct sum of some copies of $V_{i}$. Viewing the $V_{i}$ as
$S$-modules, we see that $G$ permutes the set of 
isomorphism classes of $\FF S$-modules $V_{1}, \ldots ,V_{m}$. Hence we are done by Lemma 
\ref{tensor1}.
\end{proof}

\smallskip 
Now we prove a version of Aschbacher's Theorem \cite{A} which we need in the sequel and 
which may be applied to other situations as well: 

\begin{propo}\label{reduction}
{\sl Assume $V = \FF^{d}$, $\ell := \Char(\FF)$, and $\GC \in \{GL(V),Sp(V),GO(V)\}$.
Let $G \le \GC$ be a Zariski 
closed subgroup that satisfies the hypothesis $\cs$. Then there is a subgroup $H \leq \GC$ such that 
$Z(\GC)G = Z(\GC)H$, $Z(H)$ is finite, and moreover, one of the following statements holds.

{\rm (i)} $H^{\circ}$ is a simple algebraic group. Furthermore, $V|_{H^{\circ}}$ is a Frobenius 
twist of a restricted irreducible module if $\ell > 0$.

{\rm (ii)} $H$ is finite, $S$ is nonabelian simple, and $S \lhd H/Z(H) \leq \Aut(S)$ 
for some nonabelian simple group $S$. 

{\rm (iii)} $H$ is finite, and $H \leq N_{\GC}(P)$, with $P = Z(P)E$ and $E$ an 
extraspecial $p$-group for some $p \neq \ell$; furthermore, $\dim(V) = \sqrt{|E/Z(E)|}$.}
\end{propo}

\begin{proof}
In the case $\GC = Sp(V)$ or $GO(V)$, we take $H = G$. If $\GC = GL(V)$, then we 
choose $H = \la \det(g)^{-1/d}g \mid g \in G \ra$, which implies
$H \leq SL(V)$ and so $Z(H)$ is finite, and $Z(\GC)G = Z(\GC)H$. 

1) Assume $H^{\circ} \neq 1$. Since $Z(H)$ is finite, $H^{\circ} \not\leq Z(H)$, whence
$V|_{H^{\circ}}$ is irreducible (and faithful) by Lemma \ref{step2}. In particular, 
$H^{\circ}$ is reductive. Clearly, $Z(H^{\circ}) \leq Z(H)$ by Schur's Lemma, whence it is
finite and so $H^{\circ}$ is semisimple. Since $H$ acts on the set $\{ H_{1}, \ldots ,H_{n}\}$ 
of simple components of $H^{\circ}$, Corollary \ref{tensor2} and $\cs$ imply
that $n = 1$ and so $H^{\circ}$ is simple. Finally, if $\ell > 0$ then $H$ acts on the set 
of isomorphism classes of Steinberg factors of $V|_{H^{\circ}}$, whence the latter must 
be restricted (up to a Frobenius twist) by Lemma \ref{tensor1} and $\cs$. 
 
2) Now we may assume that $H$ is finite. Let $\bl$ be a minimal normal subgroup of $H/Z(H)$. Here we 
consider the case $\bl$ is nonabelian; in particular $\bl$ is perfect. Assume that $\bm$ is another minimal 
normal subgroup of $H/Z(H)$. Consider the complete inverse images $M$ and $L$ of $\bm$ and $\bl$ in $G$ and 
set $K := L^{(\infty)}$. Notice that $M, K \lhd H$, $M,K \not\leq Z(H)$, $[M,K] \leq M \cap K \leq Z(K)$, 
and $[K,K] = K$. Hence $[[M,K],K] = 1$ and so $[M,K] = 1$ by the Three-Subgroup Lemma. By 
Lemma \ref{step2}, $V|_{K}$ is irreducible. But then by Schur's Lemma, $M \leq Z(H)$, a contradiction.

Next we show that $\bl$ is simple. Write 
$\bl = \bar{S}_{1} \times \ldots \times \bar{S}_{n}$, where 
$\bar{S}_{1} \simeq \ldots \simeq \bar{S}_{n}$ are simple. Let $S_{i}$ be the 
complete inverse image of $\bar{S}_{i}$ in $H$ and let $R_{i} := S_{i}^{(\infty)}$.
Since $R_{i}$ is perfect and $[R_{i},R_{j}] \leq Z(G)$ for $i \neq j$, as above
we can check that $[R_{i},R_{j}] = 1$ for $i \neq j$. Again by Lemma \ref{step2},
$V|_{K}$ is irreducible; furthermore, $K = R_{1} * \ldots * R_{n}$. Hence 
$V|_{K} = V_{1} \otimes \ldots \otimes V_{n}$, where $V_{i}$ is an irreducible 
$R_{i}$-module. Notice that $H$ permutes the subgroups $R_{i}$'s transitively,
whence $H$ permutes the spaces $V_{i}$'s transitively by Corollary \ref{tensor2}. Thus the 
$H$-module $V$ is tensor induced if $n > 1$. By Lemma \ref{step1}, $n = 1$ and so $\bl =: S$ is simple.
Clearly, $H/Z(H)$ acts on $\bl$, and $C_{H/Z(H)}(\bl)$ intersects 
$\bl$ trivially. But $\bl$ is a unique minimal normal subgroup of 
$H/Z(H)$, hence $C_{H/Z(H)}(\bl) = 1$. We conclude that 
$\bl \lhd H/Z(H) \leq \Aut(\bl)$.

3) Now we may assume that $H/Z(H)$ has an elementary abelian, minimal normal 
$p$-subgroup $\bar{L}$ for some prime $p$. Let $R$ denote the complete inverse image of 
$\bar{L}$ in $H$. Then $R' \leq Z(H)$, whence $R$ is nilpotent. In particular, 
$R = O_{p}(E) \times O_{p'}(E)$ and $O_{p'}(E) \leq Z(H)$, and so $p \neq \ell$. Let $P$ be 
the subgroup generated by all elements of $R$ of order $p$ if $p > 2$ and of order $2$ or $4$ if 
$p = 2$. Then $P \lhd H$ and $P \not\leq Z(H)$. Moreover, by Lemma \ref{step2} any 
characteristic abelian subgroup of $P$ is cyclic. It follows that $P = Z(P)E$ for some 
extraspecial $p$-group (and either $P = E$, or $|Z(P)| = 4$). By Lemma \ref{step2}, 
$V|_{P}$ is irreducible, whence $\dim(V) = \sqrt{|E/Z(E)|}$. 
\end{proof}

In what follows, the subgroup $H$ described in Proposition \ref{reduction} will be 
referred to as the {\it normalized version} of $G$ and denoted by $\GN$. Notice that 
$\det(g) = \pm 1$ for all $g \in \GN$.

To deal with self-dual modules, we will need the following two statements.

\begin{lemma}\label{dual1}
{\sl Let $V$ be an $\FF G$-module, and let $M,N \leq G$ be such that $V|_{M \cap N}$ is 
irreducible, $V|_{M}$ and $V|_{N}$ are self-dual. Then $V|_{\la M,N \ra}$ is also self-dual.}
\end{lemma}

\begin{proof}
By the assumptions, $V$ affords a non-degenerate bilinear form $B_{M}$, resp. $B_{N}$, which
is $M$-invariant, resp. $N$-invariant. By irreducibility, $M \cap N$ admits a unique (up
to scalar) non-degenerate invariant bilinear form on $V$. Hence after a suitable rescaling we
have $B_{M} = B_{N}$ and so $B_{M}$ is $\la M,N \ra$-invariant.   
\end{proof}

\begin{lemma}\label{dual2}
{\sl Let $G \leq \GC := GL(V)$, and let $N \lhd G$ such that $V|_{N}$ is irreducible
and self-dual. Then there is $G^{*} \leq GL(V)$ such that $N \lhd G^{*}$, $Z(\GC)G = Z(\GC)G^{*}$,
and $V|_{G^{*}}$ is also self-dual.}
\end{lemma}

\begin{proof}
By the assumptions, $V$ affords a (unique up to scalar) non-degenerate $N$-invariant bilinear 
form $B_{N}$. Since $N \lhd G$, each $g$ changes $B_{N}$ by a scalar $\lam_{g} \in \FF^{\times}$,
in which case define $g^{*} = \lam_{g}^{-1/2}g$. Now just take $G := \la g^{*} \mid g \in G \ra$.
\end{proof}

\begin{lemma}\label{lift}
{\sl Let $G$ be a finite group, $\varphi \in \IBRL(G)$, 
and let $N \lhd G$ be such that $\varphi|_{N}$ is irreducible
and lifts to a complex character $\rho$ of $N$. Assume that $\rho$ extends to $G$. Then $\varphi$ also 
lifts to a complex character of $G$.}
\end{lemma}

\begin{proof}
By assumption, $\rho = \mu|_{N}$ for some $\mu \in \Irr(G)$. Hence $\hat{\mu}$ is an extension of 
$\varphi|_{N}$ to $G$. By Clifford theory, in this case $\varphi = \alpha \otimes \hat{\mu}$,
where $\alpha \in \IBRL(G)$ and $\alpha(1) = 1$. In particular, $\alpha$ is a Brauer character of 
$A := O_{\ell'}(G/G')$. Now we can view $\alpha$ as a complex character $\beta$ of 
$A \simeq (G/G')/O_{\ell}(G/G')$ and then of $G/G'$. Setting $\chi = \beta \otimes \mu$, we see that 
$\varphi = \widehat{\beta \otimes \mu}$, as stated.  
\end{proof}

\begin{lemma}\label{sub}
{\sl Let $N \lhd G$, $V$ an $\FF G$-module, and $A$ a submodule of $V|_{N}$. 

{\rm (i)} If $|G/N|$ is a prime and $A$ extends to $G$, then $V$ contains a simple $G$-module of 
dimension $\leq \dim(A)$.

{\rm (ii)} If $A$ is simple and $G$-invariant, then $V$ contains a simple $G$-module of dimension at 
most $m\dim(A)$, where $m$ is the largest degree of projective irreducible $\FF(G/N)$-representations.}
\end{lemma}

\begin{proof}
Consider a simple $N$-submodule $B$ of $A$. By assumption, 
$0 \neq \Hom_{N}(B,V|_{N})$.  By Frobenius reciprocity, 
$ \Hom_{N}(B,V|_{N}) \simeq \Hom_{G}(\Ind^{G}_{N}(B),V)$.

Consider the case of (i). If $B$ is $G$-invariant, then $B$ extends to $G$ and all
composition factors of $\Ind^{G}_{N}(B)$ are of dimension equal to $\dim(B) \leq \dim(A)$; in particular
a simple submodule of $G$ has this dimension. Otherwise $\Ind^{G}_{N}(B)$ is simple, and embeds in both 
$V$ and $A_{1}$, an extension of $A$ to $G$. 

In the case of (ii), $B = A$. By Clifford theory, any composition factor of $\Ind^{G}_{N}(B)$ is 
of the form $A_{2} \otimes X$ for some projective $\FF G$-representation $A_{2}$ of degree $\dim(A)$ 
and some irreducible projectice $\FF(G/N)$-representation $X$, whence the claim follows. 
\end{proof}

We will discard the groups $G$ with a normal subgroup contained in $GO(V)$ as follows:

\begin{lemma}\label{so-sym}
{\sl Assume $G \leq GL(V)$ has a normal subgroup $N \leq GO(V)$. 

{\rm (i)} Then the $N$-module $\SK(V)$ contains a submodule isomorphic to
$1_{N}$ if $k$ is even, and $V|_{N}$ if $k$ is odd. 

{\rm (ii)} Assume $|N| > 2$ and $\dim(V) \geq 2$. Then $\SK(V)$ cannot be irreducible 
over $G$ for any $k \geq 2$.}
\end{lemma}

\begin{proof}
(i) We realize $\SK(V)$ as the space $P_{k}$ of homogeneous polynomials of degree $k$ in variables 
$x_{1}, \ldots ,x_{d}$, with $V = \la x_{1}, \ldots ,x_{d} \ra_{\FF}$ as a $GO(V)$-module (and equipped with
the standard scalar product). Notice that $h := \sum^{d}_{i=1}x_{i}^{2}$ is $GO(V)$-invariant, and so the map 
$f \mapsto hf$ yields an injective $GO(V)$-homomorphism $P_{k-2} \hookrightarrow P_{k}$ (as 
$\FF[x_{1}, \ldots ,x_{d}]$ is an integral domain). Hence the claim follows.

(ii) Assume the contrary: $\SK(V)$ is irreducible for some $k \geq 2$. By Lemmas \ref{step1} and 
\ref{step2}, $\ell > 2$ and $N$ acts either scalarly or irreducibly on $V$. In the former
case, $|N| \leq 2$ as $N \leq GO(V)$, a contradiction. Hence $N$ is irreducible on $V$.
By (i) applied to the subgroup $G^{*}$ constructed in Lemma \ref{dual2}, $\SK(V)$ has an
$G^{*}$-submodule $A$ of dimension $1$ if $k$ is even and dimension $\dim(V)$ if $k$ is odd. Since $\FB$ acts 
scalarly on $\SK(V)$ and $\dim(V) \geq 2$, $A$ is a proper $G$-submodule in $\SK(V)$, a contradiction. 
\end{proof}
 
The main reduction is provided by the following: 

\begin{propo}\label{red}
{\sl Assume $\GC := GL(V)$, $Sp(V)$ or $GO(V)$, $d := \dim(V) \geq 2$, $G \leq \GC$ is Zariski closed, and 
that $\SK(V)$ is irreducible over $G$ for some $k \geq 2$. Then $G \not\leq GO(V)$ and $\SK(V)$ is 
irreducible over $H := \GN$. Moreover, $\GN$ satisfies one of the conclusions {\rm (i) -- (iii)} of
Proposition \ref{reduction}.}
\end{propo}

\begin{proof}
The claim $G \not\leq GO(V)$ follows from Lemma \ref{so-sym}. By Lemma \ref{step1}, $G$ satisfies 
$\cs$. Now one just applies Proposition \ref{reduction}.
\end{proof}

The following simple argument is useful in various situations:

\begin{lemma}\label{index}
{\sl Let $H$ be a subgroup of a finite group $G$, and $V$ an irreducible $\FF G$-module of dimension $d$.

{\rm (i)} Assume that any irreducible $\ell$-modular Brauer character of $H$ of degree $\leq d$ is 
of type $+$. If $\SK(V)$ is irreducible for some $k \geq 2$, then
$\dim(\SK(V)) \leq (G:Z(G)H)$ if $k$ is even, and $\dim(\SK(V)) \leq d(G:Z(G)H)$ if $k$ is odd.

{\rm (ii)} Assume $V|_{H}$ contains the submodule $B \oplus B^{*}$ for some $H$-module $B$.
If $\WK(V)$ is irreducible for some even $k \leq 2\dim(B)$, then $\dim(\WK(V)) \leq (G:Z(G)H)$.}
\end{lemma}

\begin{proof}
(i) Consider a simple submodule $U$ of $V|_{H}$. By assumption, $U$ is of type $+$. Hence by Lemma 
\ref{so-sym}, $\SK(V)|_{H}$ contains a submodule $A$, where $A = 1_{H}$ if $k$ is even and $A \simeq U$ if 
$k$ is odd. By Lemma \ref{step1}, $Z(G)$ acts scalarly on $V$ and on $\SK(V)$, so we may view
$A$ as a $Z(G)H$-module. Now the claim follows by Frobenius' reciprocity.

(ii) Clearly, 
$\WK(V)|_{H} \supset \wedge^{k/2}(B) \otimes \wedge^{k/2}(B^{*}) \simeq 
 \wedge^{k/2}(B) \otimes (\wedge^{k/2}(B))^{*} \supset 1_{H}$. Thus the $H$-fixed point subspace on 
$\WK(V)$ is nonzero, and $Z(G)$ certainly has a $1$-dimensional submodule in it, whence the claim again 
follows by Frobenius' reciprocity.  
\end{proof}

If $V$ is a self-dual simple $\CC G$-module, then the type of $V$ can be determined using the 
Frobenius-Schur indicator. In the modular case, the following result of Thompson is very useful:

\begin{lemma}\label{type} {\rm \cite{Th}} 
{\sl Let $G$ be a finite group and let $\chi \in \Irr(G)$ be a 
real-valued character. For an odd prime $\ell$, assume that $\varphi \in \IBR_{\ell}(G)$ is a real-valued 
constituent of odd multiplicity in $\hat{\chi}$. Then $\varphi$ has the same type as of $\chi$.
\hfill $\Box$}
\end{lemma}

\subsection{Reduction to lower symmetric/exterior powers.}

\begin{lemma}\label{sym1}
{\sl Assume $k,l \geq 1$ and that either $\Char(\FF) = 0$, or $\GC \in \{GL(V),Sp(V)\}$ and 
$\Char(\FF) \geq \max\{k,l\}+2$. Then $L(k\om_{1}) \otimes L(l\om_{1})^{*}$ embeds in 
$L((k+1)\om_{1}) \otimes L((l+1)\om_{1})^{*}$ as a $\GC$-submodule.}
\end{lemma}
 
\begin{proof}
First we assume $\GC = GL(V)$. 
Let $P_{l,k}$ be the subspace of   $\FF[x_{1}, \ldots ,x_{d},y_{1}, \ldots ,y_{d}]$
consisting of homogeneous polynomials of degree $l$ in $x_{1}, \ldots ,x_{d}$ and 
of degree $k$ in $y_{1}, \ldots ,y_{d}$. Furthermore, let $GL(V)$ act naturally on 
$\FF^{d}$ and let it act on $P_{l,k}$ via  
$g \cdot f(x,y) = f(^{t}gx,g^{-1}y)$. The condition on $\Char(\FF)$ ensures that 
$L(k\om_{1}) \otimes L(l\om_{1})^{*} \simeq P_{l,k}$ and 
$L((k+1)\om_{1}) \otimes L((l+1)\om_{1})^{*} \simeq P_{l+1,k+1}$. Observe that 
$h := \sum^{d}_{i=1}x_{i}y_{i}$ is $GL(V)$-invariant. Also, 
$\FF[x_{1}, \ldots ,x_{d},y_{1}, \ldots ,y_{d}]$ is an integral domain. So 
the multiplication by $h$ yields an injective $\GC$-homomorphism
$P_{l,k} \hookrightarrow P_{l+1,k+1}$. Notice that, under the given assumptions, the modules 
$L(m\om_{1})$ of $GL(V)$ and of $Sp(V)$, with $m \in \{k,l,k+1,l+1\}$, are the same, so 
we are also done with $Sp(V)$.

Next we consider the case $\GC = GO(V)$ and $\Char(\FF) = 0$. It is proved in \cite{DW} that
$L(k\om_{1}) \otimes L(l\om_{1})^{*} \simeq P_{l,k} \cap \HC$, where 
$$\HC := \left\{ f(x,y) \in \FF[x_{1}, \ldots ,x_{d},y_{1}, \ldots ,y_{d}] \mid 
  \sum^{d}_{i=1}\frac{\partial^{2}f}{\partial x_{i}^{2}} = 
  \sum^{d}_{i=1}\frac{\partial^{2}f}{\partial y_{i}^{2}} = 0\right\},$$
and that the operator $\sum^{d}_{i=1}\dfrac{\partial^{2}}{\partial x_{i} \partial y_{i}}$
yields a surjective $SO(V)$-homomorphism from $(P_{l+1,k+1} \cap \HC)$
to $ (P_{l,k} \cap \HC)$. In fact this
homomorphism is also a $GO(V)$-homomorphism.  
\end{proof}

The following is a theorem of Serre:

\begin{lemma}\label{serre} {\rm \cite{S}}
{\sl Assume $V_{1}, \ldots ,V_{m}$ are semisimple $\FF G$-modules and 
$\Char(\FF) > \sum^{m}_{i=1}(\dim(V_{i}) -1)$. Then the $G$-module 
$V_{1} \otimes \ldots \otimes V_{m}$ is also semisimple. 
\hfill $\Box$}
\end{lemma}

\begin{corol}\label{sym2}
{\sl {\rm (i)} Let $G$ be a subgroup of $GL(V)$. Assume that the $G$-module $\SK(V)$ is
reducible for some $k \geq 1$. Assume furthermore that either 
$\ell := \Char(\FF) = 0$ or $\ell > k(\dim(V)-1)$. Then 
for every $m \geq k$, the $G$-module $\SM(V)$ is also reducible.

{\rm (ii)} Let $\Char(\FF) = 0$ and let $G$ be any subgroup of $\GC := GO(V)$. Assume that the 
$\GC$-module $L(k\om_{1})$ is irreducible over $G$ for some $k \geq 1$. Then for every $m$, 
$1 \leq m \leq k$, the $\GC$-module $L(m\om_{1})$ is also irreducible over $G$.}
\end{corol}

\begin{proof} Clearly, we may assume $\dim(V) > 1$.
 
(i) Consider any $m \geq k$. By Lemma \ref{step1} we may assume that $G$ is irreducible (and faithful) on 
$V$ and that $\ell > m$. In particular, $\SM(V) = L(m\om_{1})$. Furthermore, the condition on $\ell$ implies 
that the $G$-module $\SK(V)$ is semisimple by Lemma \ref{serre}. Since $\SK(V)$ is reducible over
$G$, the semisimplicity implies that the fixed point subspace $M_{k}^{G}$ has dimension 
$\geq 2$, where we set 
$$M_{n} := \Sym^{n}(V) \otimes (\Sym^{n}(V))^{*} \simeq \Sym^{n}(V) \otimes \Sym^{n}(V^{*}).$$ 
But $M_{k}$ embeds in $M_{m}$ as a $G$-module by Lemma \ref{sym1}. It follows that 
$\dim(M_{m}^{G}) \geq 2$, and so the $G$-module $\SM(V)$ is reducible by Schur's Lemma. 

(ii) Recall that $L(k\om_{1})$ is a $\GC$-submodule of $V^{\otimes k}$. First we consider the case
the $G$-module $V$ is semisimple. By Lemma \ref{serre}, the $G$-module $V^{\otimes k}$ is 
semisimple, and so is $L(k\om_{1})$. Now we can apply Lemma \ref{sym1} and argue as above.

Next we consider the general case. If every simple $G$-submodule of $V$ is non-degenerate (w.r.t. the
bilinear form on $V$), then clearly the $G$-module $V$ is semisimple and so we are done. Otherwise
$G$ preserves a nonzero (and proper) totally singular subspace $W$ of $V$. It suffices to show that 
$\HC := Stab_{\GC}(W)$ is reducible on $L(m\om_{1})$. Let $\UC$ be the unipotent radical of $\HC$. 
The kernel of the action of $\GC$ on $L(m\om_{1})$ is obviously normal in $\GC$ and therefore has order 
$\leq 2$. It follows that $\UC$ acts nontrivially on $L(m\om_{1})$ and so its fixed point subspace $F$ on 
$L(m\om_{1})$ gives a nonzero proper $\HC$-submodule.
\end{proof}

Similarly, the following statement holds for exterior powers:

\begin{lemma}\label{wedge}
{\sl Assume that $(d+1)/2 \geq k \geq 1$ and that either $\ell := \Char(\FF) = 0$ or $\ell > 2k(d-k)$. 

{\rm (i)} Then $\wedge^{k-1}(V) \otimes \wedge^{k-1}(V)^{*}$ embeds in 
$\WK(V) \otimes \WK(V)^{*}$ as an $SL(V)$-submodule.

{\rm (ii)} Assume that $G \leq GL(V)$ and that $\WK(V)$ is irreducible over $G$. Then 
for every $m$ with $1 \leq m \leq k$, the $G$-module $\WM(V)$ is also irreducible.}
\end{lemma}
 
\begin{proof}
If $d$ is odd and $k = (d+1)/2$ then $\wedge^{k-1}(V) \simeq \WK(V)^{*}$. So we may assume that 
$k \leq d/2$. 

(i) First we consider the case $\ell = 0$. Let $\om_{1}, \ldots ,\om_{d-1}$ denote the 
fundamental weights of $SL(V)$. Using \cite[Prop. 15.25]{FH} one can show that   
$$\WK(V) \otimes \WK(V)^{*} = \bigoplus_{0 \leq i \leq k}L(\om_{i} + \om_{d-i})$$
as an $SL(V)$-module. Hence the claim follows.

Now we assume that $\ell > 2k(d-k)$. According to \cite{McN}, $\WM(V) \otimes \WM(V^{*})$ is semisimple 
over $\GC := SL(V)$ for $1 \leq m \leq k$. Consider the complex Lie group $\GCC = SL_{d}(\CC)$ and its
natural module $V_{\CC} = \CC^{d}$, and label the fundamental weights of $\GCC$ in the same way as we did
for $\GC$. Notice that, for a given highest weight $\om$, the Weyl module $V(\om)$ of $\GC$ can be 
obtained by a reduction modulo $\ell$ of the irreducible module $L_{\CC}(\om)$ of $\GCC$, and
$L_{\CC}(\om_{m}) = \WM(V_{\CC})$. So the above claim applied to the $\GCC$-module
$V_{\CC}$ now implies that the multiplicity of each $\GC$-composition factor in 
$\wedge^{k-1}(V) \otimes \wedge^{k-1}(V)^{*}$ is at most that of the same composition factor in 
$\WK(V) \otimes \WK(V)^{*}$. Hence our claim follows by semisimplicity.        
 
(ii) As in the proof of Proposition \ref{red} we may assume that $G \leq SL(V)$. If $G$ is reducible on 
$V$ then it is easy to see that $G$ is also reducible on $\WM(V)$. So we may assume $G$ is irreducible
on $V$, whence $\wedge^{m}(V)$ is semisimple by \cite{McN}. Now argue as in the proof of Corollary 
\ref{sym2}(i). 
\end{proof}

Now we provide analogues of Proposition \ref{red} and Lemma \ref{so-sym} for exterior powers:

\begin{propo}\label{red-a}
{\sl Assume $\GC := GL(V)$, $Sp(V)$ or $GO(V)$, $d := \dim(V) \geq 4$, $G \leq \GC$ is Zariski closed, and 
that $\WK(V)$ is irreducible over $G$ for some $k$, $2 \leq k \leq d/2$. Then $G \not\leq Sp(V)$ and 
$\WK(V)$ is irreducible over $H := \GN$. Moreover, one of the following statements holds.

{\rm (a)} $G$ satisfies $\cs$, and $H$ satisfies one of the conclusions {\rm (i) -- (iii)} of 
Proposition \ref{reduction}. 

{\rm (b)} Either $V = \oplus^{d}_{i=1}V_{i}$ is a sum of $1$-spaces and $G$ acts $k$-homogeneously on
$\{V_{1}, \ldots ,V_{d}\}$, or $k = 2$ and $V = V_{1} \otimes V_{2}$ with $G$ permuting $V_{1}$ and 
$V_{2}$ transitively.}
\end{propo}

\begin{proof}
It is well known, cf. \cite{Se1} that $G \not\leq Sp(V)$. By Lemmas \ref{step1a} and \ref{wedge} we
may assume that $\ell = \Char(\FF) > 0$ and that $G$ fails (a). Then $G$ fails the 
condition $\cs$. Arguing as in the proof of Lemma \ref{step1a} we see that $G$ is irreducible on $V$.
Now $G$ stabilizes an imprimitive decomposition, a tensor decomposition, or a 
tensor induced decomposition of $V$. Fix a basis $(e_{1}, \ldots ,e_{d})$ for $V$ that is compatible 
with this $G$-invariant decomposition. Then we use this basis to define the space 
$\VQ := \la e_{1}, \ldots ,e_{d} \ra_{\BQ}$ and the $R$-module 
$\VR := \la e_{1}, \ldots ,e_{d} \ra_{R}$, where $R$ is the ring of all 
algebraic integers in $\BQ$. Notice that if $\pi$ is a maximal ideal of $R$ that contains $\ell$, then 
$R/\pi \simeq \overline{\FF}_{\ell}$ can be embedded in $\FF$. Moreover, 
$V_{1} := \la e_{1}, \ldots ,e_{d} \ra_{R/\pi}$ and $\WM(V_{1})$ can be obtained by reducing $\VR$ and 
$\WM(\VR)$ modulo $\pi$. Now let $\HC$, resp. $\HCQ$, $\HC_{R}$, $\HC_{1}$, denote the stabilizer of 
the aforementioned $G$-invariant decomposition in $\GC$, resp. in $GL(\VQ)$, $GL(\VR)$, $GL(V_{1})$. Then 
the action of $\HC_{1}$ on $\WK(V_{1})$ can be obtained by reducing modulo $\pi$ the action of $\HC_{R}$ 
on $\WK(\VR)$. Since $\WK(V)$ is irreducible over $G \leq \HC$ and $\HC_{1}$ is Zariski dense in $\HC$, 
$\WK(V_{1})$ is irreducible over $\HC_{1}$. 

On the other hand, the statement in the characteristic zero case applied to $\BQ$ implies that $\WK(\VQ)$
is reducible over $\HCQ$. Let $U \neq 0$ be a proper $\HCQ$-submodule in $\WK(\VQ)$. Notice that 
$L = \WK(\VR)$ is the free $R$-module spanned by $w_{j}$ of the form 
$e_{i_{1}} \wedge \ldots \wedge e_{i_{k}}$, 
$1 \leq i_{1} < \ldots < i_{k} \leq d$. Claim that the $R$-module $L/M$ is torsion free, where 
$M := U \cap L$. (Indeed, assume $0 \neq r \in R$, $v \in L$ and $rv \in M$. Then $v = r^{-1} rv \in U$ 
and $v \in L$, whence $v \in M$.) Observe that $R$ is a Bezout domain of dimension $1$, i.e. every 
finitely generated ideal of $R$ is principal and every finitely generated torsion free $R$-module is free. 
Since $L/M$ is finitely generated, it follows that $L/M$ has an $R$-basis
$(f_{1} + M, \ldots ,f_{s} + M)$ for some $f_{i} \in L$, $1 \leq i \leq s$. Notice
that $s > 0$ as otherwise $U \supseteq L$ and so $U$ would not be proper in $\WK(\VQ)$. Setting 
$J := \la f_{1}, \ldots ,f_{s} \ra_{R}$, we see that $L = M \oplus J$. Reducing modulo $\pi$,
we get $L/\pi L = (M+\pi L)/\pi L \oplus (J + \pi L)/\pi L$ as $R/\pi$-spaces. Claim that 
$(M + \pi L)/\pi L \neq 0$. (Indeed, since $U \neq 0$ and $\BQ = {\rm Quot}(R)$, we can find 
$0 \neq u = \sum_{j}a_{j}w_{j} \in U$ for some $a_{j} \in R$. Now the ideal of $R$ generated by the 
$a_{j}$ is finitely generated and so a principal ideal, say $bR$ with $0 \neq b \in R$. In this case, 
$b^{-1}u = \sum_{j}b^{-1}a_{j}w_{j} \in M$, and not all $b^{-1}a_{j}$ can belong to $\pi$, whence 
$b^{-1}u \notin \pi L$.) Consequently, $(M+\pi L)/\pi L$ is a nonzero proper $\HC_{1}$-invariant subspace
in $L/\pi L = \WK(V_{1})$, a contradiction.
\end{proof}
 
\begin{lemma}\label{sp-wed}
{\sl Assume $G \leq GL(V)$ has a normal subgroup $N \leq Sp(V)$. 

{\rm (i)} Then the $N$-module $\WK(V)$, where $1 \leq k \leq \dim(V)-1$, contains a submodule isomorphic to
$1_{N}$ if $k$ is even, and $V|_{N}$ if $k$ is odd. 

{\rm (ii)} Assume $N$ is finite and that $O^{\ell}(N)$ is not abelian. Then $\WK(V)$ cannot be 
irreducible over $G$ for any $k$, $2 \leq k \leq \dim(V)-2$.}
\end{lemma}

\begin{proof}
(i) We fix a symplectic basis $(e_{1}, \ldots ,e_{d})$ for $V$ with $d := \dim(V)$ and use it to define a 
complex space $\VC$ and the corresponding symplectic group $Sp(\VC)$, as well as the $R$-module 
$\VR := \la e_{1}, \ldots ,e_{d} \ra_{R}$ and the symplectic group $Sp(\VR)$, where $R$ is the ring of all
algebraic integers in $\CC$. Then the contraction map \cite[p. 260]{FH}
$$\partial~:~v_{1} \wedge \ldots \wedge v_{m} \mapsto \sum_{i < j}(v_{i},v_{j})(-1)^{i+j-1}
  v_{1} \wedge \ldots \wedge v_{i-1} \wedge v_{i+1} \wedge \ldots \wedge v_{j-1} 
  \wedge v_{j+1} \wedge \ldots \wedge v_{m}$$
is an $Sp(\VC)$-homomorphism $\WK(\VC) \to \wedge^{k-2}(\VC)$. In particular, if $s = \lfloor k/2 \rfloor$
and $r = k-2s$, then $\partial^{k}~:~\WK(\VC) \to \wedge^{r}(\VC)$ is surjective, where $\wedge^{0}(\VC)$ 
is defined to be the trivial module $\CC$, and we are done if $\Char(\FF) = 0$. Assume 
$\ell = \Char(\FF) > 0$. Observe that 
$\{e_{i_{1}} \wedge \ldots \wedge e_{i_{m}} \mid 1 \leq i_{1} < \ldots < i_{m} \leq d\}$ is a basis for
$\WM(\VC)$ for each $m$, and relative to these bases, $\partial$ has an integer matrix. Let $\ell^{c}$ be 
the highest power of $\ell$ that divides all the coefficients of the matrix of $\partial^{s}$ and let 
$\sigma = \ell^{-c}\partial^{s}$. Then $\sigma$ commutes with $Sp(\VR)$. Since the matrix of $\sigma$ has 
integer entries, we can use this matrix to define a map 
$\sigma_{1}~:~\WK(V_{1}) \to \wedge^{r}(V_{1})$ that 
commutes with $Sp(V_{1})$, where $V_{1} := \la e_{1}, \ldots ,e_{d} \ra_{R/\pi}$ and $\pi$ is a 
maximal ideal of $R$ containing $\ell$. Since $Sp(V_{1})$ is Zariski dense in $Sp(V)$, $\sigma_{1}$ 
is an $Sp(V)$-homomorphism $\WK(V) \to \wedge^{r}(V)$. Since $\WM(V)$ is a self-dual module over 
$Sp(V)$, we are done. 

(ii) We may assume $k \leq d/2$. As mentioned in (i), $N$ has an $1$-dimensional trivial submodule 
but does not act trivially on $\WB(V)$ (otherwise $O^{\ell}(N) \leq Z(G)$ as shown in the proof of 
\cite[Lem. 3.6]{GT2}), so $G$ is reducible on $\WB(V)$. If $k \geq 3$, we can argue as 
in the proof of Lemma \ref{so-sym}(ii), using Proposition \ref{red-a}. The only exception that may arise 
here is that $N$ preserves every component of a decomposition of $V$ into a direct sum of $1$-spaces; but
in this case $N = O^{\ell}(N)$ is abelian.
\end{proof} 

\section{The defining characteristic case}

\begin{theor}\label{defi}
{\sl Assume $G \leq \GC := GL(V)$ is Zariski closed and $\SK(V)$ is irreducible over $G$ for some
$k \geq 2$. 

{\rm (i)} Assume that $\GNC$ is a simple algebraic group. Then 
$\GNC = SL(V)$ or $Sp(V)$.

{\rm (ii)} Assume that $\ell > 0$ and that $\bl \lhd \GN/Z(\GN) \leq \Aut(\bl)$ for a nonabelian simple  
group $\bl \in Lie(\ell)$. Then $S \lhd \GN \leq N_{\GC}(S)$ with $S = SL_{d}(q)$, 
$SU_{d}(q)$, or $Sp_{d}(q)$, for some power $q = \ell^{a}$ and $d = \dim(V)$.}
\end{theor}

\begin{proof}
(i) By Proposition \ref{red} (and Proposition \ref{reduction}(i)), we know that $V|_{\HC} = L(\om)$, 
an irreducible module with (restricted if $\ell > 0$) highest weight $\om$,
where $\HC := \GNC$ and we fix a maximal torus $\TC$ of $\HC$. Then $k\om$ is the highest 
weight in $\SK(V)$, and moreover there is a unique (up to scalar) vector $v$ corresponding to
the weight $k\om$ of $\TC$. As usual, we may assume $\ell > k$. Then $L(k\om)$ is a 
composition factor of the $\HC$-module $\SK(V)$. Since $\GN$ normalizes $\HC$, $\om$ is 
$\GN$-invariant, and so is $k\om$. By Clifford's Theorem, it follows that every composition
factor of the $\HC$-module $\SK(V)$ is isomorphic to $L(k\om)$. But then the uniqueness of 
$v$ implies that $\SK(V)|_{\HC} \simeq L(k\om)$ and so $\SK(V)$ is irreducible over 
$\HC$. Now we can apply the fundamental result of Dynkin \cite{Dyn} (in the case $\ell = 0$) and
of Seitz \cite{Se1} (in the case $\ell > 0$, see also \cite{Su}), to $\HC$ and conclude that 
$\HC = SL(V)$ or $Sp(V)$.

(ii) Let $M$ be the complete inverse image of $\bl$ in $H$ and let $S := M^{(\infty)}$. Then
$S$ is quasisimple. Since $V|_{S}$ is irreducible by Lemma \ref{step2}, $O_{\ell}(Z(S)) = 1$.
It follows that there is a simple simply connected algebraic group $\SC$ in characteristic 
$\ell$ and a Frobenius map $F$ on $\SC$ such that $S$ is a quotient of $\SC^{F}$. Without loss
we may assume that $S = \SC^{F}$. Since $\GN$ preserves the set of isomorphism classes of 
Steinberg factors of $V|_{S}$, Lemmas \ref{tensor1} and \ref{step1} imply that $V|_{S}$ is a 
Frobenius twist of a restricted module. So without loss we may assume that 
$V|_{S} = L(\om)|_{S}$ for some irreducible $\SC$-module $L(\om)$ with restricted highest 
weight $\om$. Let $\Phi~:~\SC \to GL(V)$ be the representation afforded by $L(\om)$ (where
we identify the spaces $L(\om)$ and $V$) and let $\LC := \Phi(\SC)$. Also let
$q = \ell^{a}$ be the absolute value of eigenvalues of $F$.

Recall that $\ell > k \geq 2$. Notice that $(\SC,\ell) \neq (G_{2},3)$. (For otherwise
$V|_{S}$ is self-dual and furthermore it has odd dimension by \cite{JLPW}, whence 
$\Phi(S) \leq SO(V)$). 
Since $k = 2$ in this case, we get a contradiction by Lemma \ref{so-sym}.) 

Consider an arbitrary (rational) representation of $\SC$ on $V$ that extends $\Phi|_{S}$, say 
yielding an irreducible module $L(\gam)$ with highest weight $\gam$. Claim that this 
representation has image equal to $\LC$. 

As shown in \cite{MT1}, \cite[(1.6)]{Se1} (and the assumption on $\ell$) implies that 
$L(\om)|_{S}$ is tensor indecomposable. Observe that $\gam = p^{m}\beta$ for some restricted 
weight $\beta$. Otherwise using Steinberg's tensor product theorem we would see that 
$L(\om)|_{S} = L(\gamma)|_{S}$ is tensor decomposable. Now if $F$ is untwisted, then
the equality $L(\om)|_{S} = L(\gam)|_{S}$ implies by the classification of 
irreducible $S$-modules that $\gam = q^{b}\om$ for some integer $b \geq 0$, whence $L(\gam)$ 
can be obtained from $L(\om)$ by twisting it using the Frobenius twist 
$(x_{ij}) \mapsto (x_{ij}^{q})$ and so the claim follows. 
If $F$ is twisted, then $F = q\rho$ and $\rho$ induces an automorphism $\sigma$ of 
$\SC$. In this case, $L(q\beta)|_{S} = L(\rho^{-1}(\beta))|_{S}$, whence
by the classification of irreducible $S$-modules we obtain 
$\gam = q^{b}\rho^{c}\om$ for some integers $b,c \geq 0$. Thus $L(\gam)$ can be obtained from
$L(\om)$ by twisting it using the Frobenius twist $(x_{ij}) \mapsto (x_{ij}^{q})$ and 
the automorphism $\sigma$, whence our claim follows.   
 
Next we show that $\GN \leq N_{GL(V)}(\LC)$. Indeed, for any $g \in \GN$ the representation 
$x \mapsto g\Phi(x)g^{-1}$ of $\SC$ on $V$ extends $(\Phi|_{S})^{g} \simeq \Phi|_{S}$ and so has image 
equal to $\LC$ by our claim. Thus $g\LC g^{-1} = \LC$. Since 
$N_{GL(V)}(\LC) = Z(\GC) \cdot N_{SL(V)}(\LC)$, it now follows that $N_{SL(V)}(\LC)$ is 
also irreducible on $\SK(V)$. Notice that the irreducibility of $V|_{S}$ and the simplicity 
of $\SC$ implies that $\LC$ is a quotient of $\SC$ by a finite subgroup, and that
$\LC$ has finite index in $N_{SL(V)}(\LC)$; in particular, 
$N_{SL(V)}(\LC)^{\circ} = \LC$. Applying (i) to $N_{SL(V)}(\LC)$, we see that 
$\LC = SL(V)$ or $Sp(V)$. As $\SC$ is simply connected, we get $\SC = SL(V)$ or $Sp(V)$.
Taking the $F$-fixed points, we obtain $S = SL_{d}(q)$, $SU_{d}(q)$, or $Sp_{d}(q)$.       
\end{proof}

Now we provide an analogue of Theorem \ref{defi} for exterior powers:

\begin{propo}\label{defi-a}
{\sl Assume $G \leq \GC := GL(V)$ is Zariski closed, and that $\WK(V)$ is irreducible 
over $G$ for some $k$, where $d/2 \geq k \geq 2$. Assume furthermore that either 
$\ell = 0$, or $\ell > 2k(d-k)$, or $k = 2$. 
Assume in addition that $\GNC$ is a simple algebraic group. Then one of the following holds.

{\rm (i)} $\GNC$ = $SL(V)$ or $SO(V)$.

{\rm (ii)} $k = 2$ and $(\GNC,\GC)$ has type $(A_{n},A_{n(n+3)/2})$ or $(A_{n},A_{(n-1)(n+2)/2})$.

{\rm (iii)} $k \leq 3$ and $(\GNC,\GC)$ has type $(D_{5},A_{15})$.

{\rm (iv)} $k \leq 4$ and $(\GNC,\GC)$ has type $(E_{6},A_{26})$.\\
In particular, if $G \not\geq SL(V)$ then $\SB(V)$ is reducible over $G$.}
\end{propo}

\begin{proof}
Set $\HC := \GNC$. Since $V|_{\HC}$ is irreducible, $C_{\GC}(\HC) = Z(\GC)$. If 
$\GN \leq Z(\GC)\HC$, then $\WK(V)$ is irreducible over $\HC$ and we can apply the results of 
\cite{Se1} (see also \cite{Su}) and arrive at (i) -- (iv). Assume $\GN \not\leq Z(\GC)\HC$ and 
$\GNC \neq SO(V)$. We claim that $V|_{\HC}$ is self-dual in this case. 
(For, in the case $\HC$ is of type $D_{2m}$ the claim follows from
\cite[Prop. 5.4.3]{KlL}. In all the remaining cases, $\GN$ induces an outer automorphism
$\varphi$ of $\HC$ which stabilizes $V|_{\HC}$. Moreover, modulo inner automorphisms of $\HC$, $\varphi$ 
is just an involutive graph automorphism of $\HC$, and $\varphi$ sends any (finite dimensional) 
irreducible $\FF \HC$-module to its dual. It follows that $V|_{\HC}$ is self-dual in this case as well.) 
We will replace $G$ by $\GN^{*}$ using the construction in Lemma \ref{dual2}. Notice that now 
$\HC = G^{\circ}$ and $G \leq Sp(V)$ or $G \leq GO(V)$. The former case is impossible, as otherwise $G$ is 
reducible on $\WK(V)$. In the latter case, $\WB(V)$ can be identified with the adjoint module 
$Lie(SO(V))$, which contains $Lie(\HC)$ as a $G$-submodule. By our assumptions,  
$\dim(\HC) < \dim(SO(V))$ and so $G$ is reducible on $\WB(V)$; in particular, $k > 2$. But this contradicts
Lemma \ref{wedge}(ii). The final statement follows from Theorem \ref{defi}.        
\end{proof}

\begin{remar}
{\em An extension of Seitz's results \cite{Se1} to the disconnected (simple) case has been made in 
\cite{Fo}. However, we cannot apply results of \cite{Fo} to determine the disconnected subgroups
$G$ of $GL(V)$ such that $G^{\circ}$ is simple and some $W \in \{\SK(V),\WK(V)\}$ is irreducible over 
$G$, as \cite{Fo} imposes the condition that all $G^{\circ}$-composition factors of $W$ have restricted 
highest weights.}
\end{remar}

\section{The cross characteristic case}

The aim of this section is to prove the following theorem:

\begin{theor}\label{cross}
{\sl Theorem \ref{main} holds true in the case where $G$ is finite, $S := \soc(G/Z(G))$ is a finite 
simple group of Lie type in characteristic $p$, and $\Char(\FF) = \ell \neq p$.}
\end{theor} 

Let $L = G^{(\infty)}$ be the perfect inverse image of $S$ in $G$. 

\subsection{Generalities.}
Usually, $G$ contains a $p$-subgroup of {\it special type}, that is,
$[Q,Q] = Z(Q) = \Phi(Q)$ has exponent $p$ and $[x,Q] = Z(Q)$ for all
$x \in Q \setminus Z(Q)$. In particular, $|Q/Z(Q)| \geq 4$. 
Moreover, $Z(Q)$ is a long-root subgroup of $G$, and $P := N_{G}(Q) = N_{G}(Z(Q))$ is a 
parabolic subgroup of $G$. Let
$$\OV := \{\lam \in \IBRL(Z(Q)) \mid \lam
 \mbox{ occurs in }V|_{Z(Q)}\},~\OSV := \OV \setminus \{1_{Z(Q)}\}.$$
For any $\lam \in \OV$, let $\VL$ be the $\lam$-eigenspace for $Z(Q)$ on
$V$. Also let $\dvl := \dim(\VL)/\sqrt{|Q/Z(Q)|}$ for $\lam \in \OSV$. Given any 
nontrivial $\lam \in \IBRL(Z(Q))$, there is a unique irreducible $\FF Q$-module 
$\QL$ on which each $z \in Z(Q)$ acts as the scalar $\lam(z)$ and in fact
$\QL$ affords the $Z(Q)$-character $\sqrt{|Q/Z(Q)|}\lam$, cf. \cite[Lem. 2.3]{LS}.  We will 
consider the following condition imposed on $C_{G}(Z(Q))$:
$$\sta~:~\QL \mbox{ extends to an }\FF C_{G}(Z(Q))\mbox{-module }\EL.\hspace{6cm}$$

The key ingredient of our treatment of the cross characteristic case is the following:

\begin{propo}\label{key1}
{\sl Let $G$ be a finite group with a $p$-subgroup $Q$ of special type and
$\Char(\FF) \neq p$. Assume $C := C_{G}(Z(Q))$ satisfies the condition $\sta$. Let $V$
be an $\FF G$-module such that there is a
$\lam \in \IBRL(Z(Q))$ with $\lam,\lam^{-1} \in \OSV$ and
$\lam^{-1} \neq \lam$. 

{\rm (i)} Then $\Sym^{2k}(V)$ has a $C$-submodule $F$ of dimension
$$\begin{pmatrix}\dvl+k-1\\k \end{pmatrix} \cdot 
 \begin{pmatrix}\dwl+k-1\\k \end{pmatrix}~.$$ 

{\rm (ii)} Assume that $\Sym^{n}(V)$ is irreducible over $G$ for some $n \geq 4$, and in addition
that $|\OV| \geq 3$ if $n$ is odd. Then $|Q/Z(Q)|^{2} < (3/2)\cdot(G:C)$.}
\end{propo}

\begin{proof}
Since $\Char(\FF) \neq p$, we can write
$V = \left(\bigoplus_{\mu \in \OV}V_{\mu}\right) \oplus V_{1}$, where $1$ stands for $1_{Z(Q)}$ 
for short. Clearly, $\VL|_{Q}$ is the direct sum of $\dvl$ copies of $\QL$. Hence the condition
$\sta$ implies that $\VL = \EL \otimes A$ for some $C/Q$-module $A$. Without loss we may
choose $\ELB$ to be $\ELC$, the dual of $\EL$. Then, again by $\sta$, 
$\VLB = \ELC \otimes B$ for some $C/Q$-module $B$. Denote $E:= \dim(\EL) = \sqrt{|Q/Z(Q)|} \geq 2$,
$a := \dim(A) \geq 1$, $b := \dim(B) \geq 1$, $c := a+b$, and $v := \dim(V_{\nu})$ if there exists
$\nu \in \OV \setminus \{\lam,\lam^{-1}\}$. Clearly, 
$\dim(V) \geq Ec+v$ and $ab(a+1)(b+1) \leq c^{4}/4$.

(i) Observe that $\Sym^{2k}(V)|_{C}$ contains the submodules 
$$\SK(\EL \otimes A) \otimes \SK(\ELC \otimes B) \supset 
  \SK(\EL) \otimes \SK(\ELC) \otimes \SK(A) \otimes \SK(B).$$
Since $\SK(\ELC) \simeq \SK(\EL)^{*}$, $\Sym^{2k}(V)|_{C}$ contains the submodule 
$F := \SK(A) \otimes \SK(B)$ which obviously has the indicated dimension. 

(ii) Now assume that $\Sym^{n}(V)$ is irreducible over $G$ for some $n \geq 4$. 
First consider the case $n = 2k$ is even. Then by (i) and by Frobenius' reciprocity,
$\dim(\Sym^{2k}(V)) \leq \dim(F) \cdot (G:C)$. Furthermore,    
$$\frac{\dim(\Sym^{2k}(V))}{\dim(F)} \geq 
  \frac{\begin{pmatrix}Ec+2k-1\\2k \end{pmatrix}}
  {\begin{pmatrix}a+k-1\\k \end{pmatrix} \cdot 
 \begin{pmatrix}b+k-1\\k \end{pmatrix}} \geq  
  \frac{\begin{pmatrix}Ec+3\\4 \end{pmatrix}}
  {\begin{pmatrix}a+1\\2 \end{pmatrix} \cdot 
 \begin{pmatrix}b+1\\2 \end{pmatrix}} > $$ 
$$> \frac{(Ec)^{4}/24}{c^{4}/16} = \frac{2E^{4}}{3},$$
proving the claim.

Next we consider the case $n = 2k+1$ is odd. Then $\Sym^{2k+1}(V)|_{C}$ contains the submodules 
$$\SK(\EL \otimes A) \otimes \SK(\ELC \otimes B) \otimes V_{\nu} 
 \supset \SK(\EL) \otimes \SK(\ELC) \otimes \SK(A) \otimes \SK(B) \otimes V_{\nu}.$$
It follows that $\Sym^{2k+1}(V)|_{C}$ contains the submodule 
$F' := \SK(A) \otimes \SK(B) \otimes V_{\nu}$. By Frobenius' reciprocity,
$\dim(\Sym^{2k+1}(V)) \leq \dim(F') \cdot (G:C)$. Furthermore,    
$$\frac{\dim(\Sym^{2k+1}(V))}{\dim(F')} \geq 
  \frac{\begin{pmatrix}Ec+v+2k\\2k+1 \end{pmatrix}}
  {\begin{pmatrix}a+k-1\\k \end{pmatrix} \cdot 
 \begin{pmatrix}b+k-1\\k \end{pmatrix} \cdot v} \geq 
  \frac{\begin{pmatrix}Ec+v+4\\5 \end{pmatrix}}
  {\begin{pmatrix}a+1\\2 \end{pmatrix} \cdot 
 \begin{pmatrix}b+1\\2 \end{pmatrix} \cdot v} > $$
$$> \frac{(Ec+v)^{5}/120}{c^{4}v/16} = E^{4} \cdot \frac{16}{120} \cdot 
    \frac{(1+v/Ec)^{5}}{v/Ec} > \frac{8E^{4}}{5},$$
since $(1+v/Ec)^{5} > 12v/Ec$ (indeed, on $(0,+\infty)$ the function $(1+t)^{5}/t$ attains its 
minimum at $t = 1/4$).
\end{proof}

Of course, if $p = 2$ then any $\lam \in \OV$ is self-dual. In this case we need the following
analogue of Proposition \ref{key1}:

\begin{propo}\label{key2}
{\sl Let $G$ be a finite group with a $2$-subgroup $Q$ of special type and
$\Char(\FF) \neq 2$. Let $V$ be an $\FF G$-module such that $\OSV \neq \emptyset$. Choose 
$\lam \in \OSV$ such that $\dim(\VL) = \min\{\dim(V_{\mu}) \mid \mu \in \OSV\}$. Assume that
$C := C_{G}(Z(Q))$ satisfies the condition $\sta$ for $\lam$, and that $\EL$ is of type $+$.

{\rm (i)} Then $\Sym^{2k}(V)$ has a $C$-submodule $F$ of dimension
$\begin{pmatrix}\dvl+2k-1\\2k \end{pmatrix}$. 

{\rm (ii)} Assume that $\Sym^{n}(V)$ is irreducible over $G$ for some $n \geq 4$ and in addition
that $|\OV| \geq 2$ if $n$ is odd. Then 
$|Q/Z(Q)|^{2} < (15/2) \cdot(G:C)$ if $\dim(V) \geq 2\dim(\EL)$, and 
$|Q/Z(Q)|^{2} < 24 \cdot(G:C)$ otherwise.}
\end{propo}

\begin{proof}
As in the proof of Proposition \ref{key1}, we can write 
$V = \left(\bigoplus_{\mu \in \OV}V_{\mu}\right) \oplus V_{1}$, and $\VL = \EL \otimes A$ for some 
$C/Q$-module $A$. By Lemma \ref{so-sym}(i), $\Sym^{2k}(\EL)$ contains the submodule $1_{C}$ as $\EL$ is 
of type $+$. Hence $\Sym^{2k}(V)|_{C}$ contains the submodules 
$$\Sym^{2k}(\EL \otimes A) \supset 
  \Sym^{2k}(\EL) \otimes \Sym^{2k}(A) \supset F := \Sym^{2k}(A),$$
proving (i). Denote $E:= \dim(\EL) = \sqrt{|Q/Z(Q)|} \geq 2$, 
$a := \dim(A) \geq 1$, and $v := \dim(V_{\nu})$ if there exists
$\nu \in \OV \setminus \{\lam\}$, and $d := \dim(V)$ as 
usual. Then $d \geq Ea +v$. 

Now assume that $\Sym^{n}(V)$ is irreducible over $G$ for some $n \geq 4$. 
First consider the case $n = 2k$ is even. Then by (i) and by Frobenius' reciprocity,
$\dim(\Sym^{2k}(V)) \leq \dim(F) \cdot (G:C)$. Furthermore,   
$$\frac{\dim(\Sym^{2k}(V))}{\dim(F)} = 
  \frac{\begin{pmatrix}d+2k-1\\2k \end{pmatrix}}
  {\begin{pmatrix}a+2k-1\\2k \end{pmatrix}} \geq  
  \frac{\begin{pmatrix}d+3\\4 \end{pmatrix}}
  {\begin{pmatrix}a+3\\4 \end{pmatrix}} > \frac{(Ea)^{4} \cdot (d/Ea)^{4}}{a(a+1)(a+2)(a+3)},$$
which is at least $2E^{4}/15$ if $a \geq 2$ or if $a = 1$ and $d \geq 2E$, and at least
$E^{4}/24$ if $a = 1$ and $d < 2E$. Hence (ii) follows. In fact we can replace the constant 
$2/15$ in (ii) by $2/9$ if $d \geq 5E/2$, and we will need this remark later.

Next we consider the case $n = 2k+1$ is odd. Then $\Sym^{2k+1}(V)|_{C}$ contains the submodules 
$$\Sym^{2k}(\EL \otimes A) \otimes V_{\nu} 
 \supset \Sym^{2k}(\EL) \otimes \Sym^{2k}(A) \otimes V_{\nu}.$$
It follows that $\Sym^{2k+1}(V)|_{C}$ contains the submodule 
$F' := \Sym^{2k}(A) \otimes V_{\nu}$. By Frobenius' reciprocity,
$\dim(\Sym^{2k+1}(V)) \leq \dim(F') \cdot (G:C)$. Furthermore,   
$$\frac{\dim(\Sym^{2k+1}(V))}{\dim(F')} = 
  \frac{\begin{pmatrix}d+2k\\2k+1 \end{pmatrix}}
  {\begin{pmatrix}a+2k-1\\2k \end{pmatrix} \cdot v} \geq 
  \frac{\begin{pmatrix}d+4\\5 \end{pmatrix}}
  {\begin{pmatrix}a+3\\4 \end{pmatrix} \cdot v},$$ 
which is at least 
$$\frac{(Ea+v)^{5}/120}{(15/2) \cdot a^{4}v/24} = E^{4} \cdot \frac{4}{150} \cdot 
    \frac{(1+v/Ea)^{5}}{v/Ea} > \frac{16E^{4}}{50}$$
if $a \geq 2$ (since $(a+1)(a+2)(a+3) \leq (15/2) \cdot a^{3}$ and $(1+v/Ea)^{5} > 12v/Ea$). If
$a = 1$, then 
$$\frac{\dim(\Sym^{2k+1}(V))}{\dim(F')} > \frac{d^{5}}{120v} > E^{4} \cdot 
  \frac{(d/E)^{4} \cdot (d/v)}{120},$$
which is at least $2E^{4}/15$ if $d \geq 2E$ (in fact at least $2E^{4}/9$ if $d \geq 5E/2$), 
and at least $$\frac{(E+v)^{5}}{120v} = E^{4} \cdot \frac{1}{120} \cdot 
    \frac{(1+v/E)^{5}}{v/E} > \frac{E^{4}}{10},$$  
if $d < 2E$. 
\end{proof}

Next we will verify various conditions set in Propositions \ref{key1} and \ref{key2}. {\it Weil 
representations} of symplectic groups will account for most exceptions where not all the conditions
are met; we refer to \cite{GMST} for necessary information about them.

\begin{lemma}\label{spec}
{\sl Let $G$ be a finite (quasi-simple) Lie-type group of simply connected type defined over 
$\FQ$, $q = p^{f}$ for a prime $p$, $Z$ a long-root subgroup of $G$, and let $V$ be any 
nontrivial irreducible $\FF G$-representation. Then one of the following holds.

{\rm (i)} $p > 2$, $\OSV = \IBRL(Z) \setminus \{1_{Z}\}$,
$|\OV| \geq 3$, and there are $\lam,\lam^{-1} \in \OSV$ with $\lam \neq \lam^{-1}$. 

{\rm (ii)} $p = 2$, $\OSV = \IBRL(Z) \setminus \{1_{Z}\}$, $|\OV| \geq 2$, and $\OSV \neq \emptyset$. 

{\rm (iii)} $p > 2$, $q \equiv 1 (\mod 4)$, $G = Sp_{2n}(q)$, $V$ is a Weil representation 
of $G$, $|\OV| \geq 3$, and there are $\lam,\lam^{-1} \in \OSV$ with $\lam \neq \lam^{-1}$.

{\rm (iv)} $p > 2$, $q \equiv 3 (\mod 4)$, $G = Sp_{2n}(q)$, and $V$ is a Weil representation 
of $G$.

{\rm (v)} $G \in \{SL_{2}(5),SU_{3}(3),Sp_{4}(3)\}$.}
\end{lemma}

\begin{proof}
Let $\Om^{*} := \IBRL(Z) \setminus \{1_{Z}\}$. First we consider the case $p = 2$. Then 
$\OSV = \Om^{*}$ and $|\OV| \geq 2$ by \cite[Lem. 2.9]{MMT}, and we arrive at (ii). 
Assume $p > 2$. Clearly (i) holds if $\OV = \IBRL(Z)$. Otherwise by \cite[Lem. 2.9]{MMT} and 
\cite{GMST}, one of the following cases hold.

\smallskip
Case 1: $q > p$ and $\OV = \Om^{*}$. In this case (i) holds.

\smallskip
Case 2: $G = SU_{3}(p)$ and $\OV = \Om^{*}$. Then either (i) holds or $G = SU_{3}(3)$. 

\smallskip
Case 3: $G = Sp_{2n}(q)$, $n \geq 1$, and either $V$ is a Weil representation and 
$|\OV| = (q+1)/2$, or $n \leq 2$ and $\OV = \Om^{*}$, or $n = 1$ and $|\OV| = (q-1)/2$. In fact 
the possibilities for $\OV$ were described in \cite[p. 386]{MMT}.  
 
Assume $q \equiv 1 (\mod 4)$. If $n \geq 2$, or if $n = 1$ but $q \geq 9$, then either (i) or (iii)
holds. Otherwise $G = SL_{2}(5)$.

Assume $q \equiv 3 (\mod 4)$, but (iv) does not hold. If $n = 2$, then either (i) holds, or
$G = Sp_{4}(3)$. The case $n = 1$ is now impossible as $G \neq SL_{2}(3)$.   
\end{proof}

\begin{propo}\label{root}
{\sl Let $G$ be a finite (quasi-simple) Lie-type group of simply connected type in characteristic 
$p$, $Z$ a long-root subgroup of $G$, $C := C_{G}(Z)$, and $Q := O_{p}(C)$. Assume 
that $G \notin \{SL_{2}(p^{f}),SU_{3}(3),Sp_{4}(3)\}$. 
If $p = 2$, assume that $G$ is not of types $\ta B_{2}$, $B_{n}$, $C_{n}$, $F_{4}$, or 
$\ta F_{4}$. If $p = 3$, assume that $G$ is not of types $G_{2}$ and $\ta G_{2}$.
Let $V$ be any nontrivial irreducible $\FF G$-representation. 

{\rm (i)} Then $C$ satisfies $\sta$ for any $\lam \in \OSV$.

{\rm (ii)} Assume $p = 2$ and $\lam \in \OSV$. If $G$ is not of type $\ta A_{2n}$, then 
$\EL$ can be chosen to have type $+$. Otherwise $\EL$ can be chosen of type $-$.}
\end{propo}

\begin{proof}
The assumption on $G$ implies that $Q$ is of special type. We will frequently aim to show that there
is a character $\chi \in \Irr(G)$ afforded by a $\QQ G$-module $W$ such that for any nontrivial 
$\lam \in \Irr(Z)$ the $\lam$-eigenspace $W_{\lam}$ of $Q$ on $W$ has dimension equal to 
$E := \sqrt{|Q/Z|}$. Since $W_{\lam}$ is clearly $C$-invariant, it then follows that $\EL$ can be 
taken to be the reduction modulo $\ell$ of the $C$-module $W_{\lam}$. Moreover, if $p = 2$ then,
since $\lam = \lam^{-1}$ and $W$ is rational, $W_{\lam}$ is of type $+$ and so is $\EL$. 
Denote $\Om^{*} := \Irr(Z) \setminus \{1_{Z}\}$ and $P := N_{G}(Z)$. 

1) Here we consider the symplectic groups $G = Sp_{2n}(q)$ with $q = p^{f}$. Then $p > 2$, 
$n \geq 2$, and $G$ has two irreducible Weil characters $\eta_{1}, \eta_{2}$ of degree 
$(q^{n}-1)/2$ which together afford all irreducible characters of $Z$. Restricting them to $C$, 
we obtain (i).   

Next let $G = SU_{n}(q)$ with $n \geq 3$. Then $G$ has a rational-valued irreducible Weil 
character $\zeta^{0}_{n,q}$ of degree $(q^{n}-(-1)^{n}q)/(q+1)$, cf. \cite{TZ2}. Since 
$(\zeta^{0}_{n,q}|_{Z},\lam)_{Z} = q^{n-2} = E$ for all $\lam \in \Om^{*}$, we arrive at (i). 
Assume in addition that $p = 2$. We will show that $\zeta^{0}_{n,q}$ has Schur 
index $1$ over $\RR$ if $n$ is even, and $2$ if $n$ is odd. The claim is clear when $n = 2$ as 
$\zeta^{0}_{2,q}$ is the Steinberg character and when $n = 3$ as $\zeta^{0}_{3,q}$ is the cuspidal 
unipotent character, see \cite{Ge}. When $n \geq 4$, one can check that $C \geq C' := Q:SU_{n-2}(q)$ and 
$\zeta^{0}_{n,q}|_{C'}$ contains $\zeta^{0}_{n-2,q}$ (inflated from $SU_{n-2}(q)$ to $C'$)
with multiplicity $1$, so we are done by induction hypothesis.

Assume $G = SL_{n}(q)$ with $n \geq 3$. Then the doubly transitive action of $G$ on $1$-spaces 
of its natural module affords the character $1+\tau$ with $\tau(1) = (q^{n}-q)/(q-1)$, and 
$\tau(t) = (q^{n-1}-q)/(q-1)$ for $1 \neq t \in Z$. It follows that $(\tau|_{Z},\lam)_{Z} = q^{n-2} = E$ 
for all $\lam \in \Om^{*}$, and so we are done. 

For the remaining Lie-type groups, any nontrivial irreducible (cross characteristic) character
of $G$ affords all $\lam \in \Om^{*}$ with equal multiplicity, since $P$ acts transitively
on $\Om^{*}$. Assume $G = Spin_{2n+1}(q)$ with $p > 2$. Then $G$ has a nontrivial irreducible 
character $\mu$ of degree $(q^{2n}-1)/(q^{2}-1)$, cf. \cite{TZ1}. Since 
$\mu(1) < 2q^{2n-3}(q-1) = 2E \cdot |\Om^{*}|$, we are done. Next, assume $G = Spin^{\eps}_{2n}(q)$
with $n \geq 4$. Then we choose $\chi$ to be an irreducible constituent of degree 
$(q^{n}-\eps)(q^{n-1}+\eps q)/(q^{2}-1)$ of the rank $3$ permutation character $\rho$ of $G$ acting 
on the singular $1$-spaces of its natural module, cf. \cite{ST}. Since $(\rho,\chi)_{G} = 1$, $\chi$ 
is rational. Notice that if $q > 2$ then $\chi(1) < 2q^{2n-4}(q-1) = 2E \cdot |\Om^{*}|$ and so we are 
done. Now assume that $q = 2$. Then $\rho(1) = (2^{n}-\eps)(2^{n-1}+\eps)$ and 
$\rho(t) = 3+4(2^{n-2}-\eps)(2^{n-3}+\eps)$ for $1 \neq t \in Z$. It follows that 
$(\rho|_{Z},\lam)_{Z} = 3 \cdot 2^{2n-4} = 3E$. Hence either $\chi$ or the other nontrivial 
constituent $\psi$ of $\rho$ affords $E\lam$, and so we are done again.

2) Now we handle the exceptional groups of Lie type. 
Consider the case $G = E_{7}(q)$. Then we choose $\chi$ to be the irreducible
character of smallest degree $q\Phi_{7}\Phi_{12}\Phi_{14}$, cf. \cite{Lu}, where $\Phi_{n}$ is the 
value of the $n^{\rm {th}}$ cyclotomic polynomial at $q$. Claim that $(\chi|_{Z},\lam)_{Z} = E$ for all 
$\lam \in \Om^{*}$. (Indeed, if $q > 2$ then $\chi(1) < 2q^{16}(q-1) = 2E \cdot |\Om^{*}|$, whence 
the claim. Assume $q = 2$ but the claim is false. Since $\chi(1) < 3q^{16}(q-1) = 3E \cdot |\Om^{*}|$,
we see that $W_{\lam}|_{Q}$ is the sum of two copies of the unique irreducible representation of 
degree $q^{16}$ of $Q$. Notice that $1 \neq t \in Z$ is $G$-conjugate to some element 
$t' \in Q \setminus Z$. Since $\chi(t) = \chi(t')$, it follows that $\chi$ has to afford some 
nontrivial linear characters of $Q$. The lengths of $P$-orbits on $\Irr(Q/Z)$ are given in 
\cite{Hof}. It follows that 
$141,986 = \chi(1) \geq 2^{17} + (2^{3}+1)(2^{5}+1)(2^{8}-1)$, a contradiction.) It remains to 
prove (ii) for even $q$. Arguing as above using $\chi(t) = \chi(t')$, one sees that $\chi$ affords
exactly one $P$-orbit, of length $(q^{3}+1)(q^{5}+1)(q^{8}-1)$, on nontrivial linear characters of 
$Q$. It was shown in \cite{Hof} that the subgroup $L' = \Omega^{+}_{12}(q)$ of the Levi subgroup $L$ 
in $P$ cannot act trivially on $C_{W}(Q)$. Let $\mu$ be an irreducible character of degree $> 1$ 
afforded by $L$ on $C_{W}(Q)$. Then $\chi$ is contained in the Harish-Chandra induction 
$R^{G}_{L}(\mu)$. Since $\chi$ is unipotent and the Harish-Chandra induction respects Lusztig series,
$\mu$ is unipotent. Notice that
$$\chi(1) = q^{16}(q-1)+(q^{3}+1)(q^{5}+1)(q^{8}-1)+(q^{6}-1)(q^{5}+q)/(q^{2}-1)+1$$
and either $\mu(1) = (q^{6}-1)(q^{5}+q)/(q^{2}-1)$ or $\mu(1) > q^{10}$ by \cite[Prop. 7.2]{TZ1}.
Hence $\mu(1) = (q^{6}-1)(q^{5}+q)/(q^{2}-1)$ and $(\chi|_{C},1_{C})_{C} = 1$. We have shown that
$\chi$ enters the permutation character $\rho = 1^{G}_{C}$ with multiplicity $1$, whence $\chi$ is 
rational and (ii) follows. 

Next we consider the case $G = E_{8}(q)$. Then we choose $\chi$ to be the irreducible
character of smallest degree $q\Phi_{4}^{2}\Phi_{8}\Phi_{12}\Phi_{20}\Phi_{24}$, cf. \cite{Lu}. 
Arguing as in the case of $E_{7}$, we see that $(\chi|_{Z},\lam)_{Z} = E$ for all $\lam \in \Om^{*}$,
whence (i) is proved. The assertion (ii) follows from a remark on \cite[p. 203]{KlL}.        

Now assume that $G = E_{6}^{\eps}(q)$, with $\eps = +$ for $E_{6}(q)$ and $\eps = -$ for 
$\ta E_{6}(q)$. Then we choose $\chi$ to be the irreducible character of smallest degree 
$q(q^{4}+1)(q^{6} + \eps q^{3}+1)$, cf. \cite{Lu}. Arguing as in the case of $E_{7}$, we see that 
$(\chi|_{Z},\lam)_{Z} = E$ for all $\lam \in \Om^{*}$, whence (i) is proved. Also, $\chi$ is an 
irreducible constituent of multiplicity $1$ of the permutation character $\rho = 1^{G}_{P_{1}}$, where 
$P_{1}$ is the parabolic subgroup of $G$ corresponding to the $A_{5}$-subdiagram of $E_{6}$ 
(before twisting if $\eps = -$). It follows that $\chi$ is rational. 

Assume $G = F_{4}(q)$. Then $p > 2$, and we choose $\chi$ to be the irreducible character of smallest 
degree $q^{8}+q^{4}+1$, cf. \cite{Lu}. Since $\chi(1) < 2q^{7}(q-1) = 2E \cdot |\Om^{*}|$, we are done.

Assume $G = \tb D_{4}(q)$. Then we choose $\chi$ to be the irreducible character of smallest 
degree $q(q^{4}-q^{2}+1)$, cf. \cite{Lu}. Since $\chi(1) < 2q^{4}(q-1) = 2E \cdot |\Om^{*}|$, (i) 
follows. Assume in addition that $q$ is even. Then the proof of \cite[Thm. 4.1]{MMT} shows that 
$(\chi|_{C},1_{C})_{C} = 1$. Thus $\chi$ enters the permutation character $\rho = 1^{G}_{C}$ with 
multiplicity $1$, whence $\chi$ is rational. 

Finally, assume $G = G_{2}(q)$ with $q \equiv \eps = \pm 1(\mod 3)$ and $q \geq 4$. Then we choose 
$\chi$ to be the irreducible character of smallest degree $q^{3}+\eps$, cf. \cite{Lu}. Since 
$\chi(1) < 2q^{2}(q-1) = 2E \cdot |\Om^{*}|$, (i) follows. Assume in addition that $q$ is even. 
The uniqueness of $\chi$ shows that $\chi$ is rational-valued. Also, since $\chi(1)$ is odd, $\chi$ has 
Schur index $1$ over $\QQ$, and so we are done.
\end{proof}

\subsection{Non-generic cases.}\label{nongeneric}
Since the unitary groups $SU_{2n+1}(q)$ with $q$ even fall out from the general scheme of arguments,
we handle them separately first. 

\begin{propo}\label{su}
{\sl Let $S \leq G/Z(G) \leq \Aut(S)$ for the simple group $S = PSU_{n}(q)$, where either 
$2 \leq n \leq 5$, or $q$ is even and $n$ is odd. Let $V$ be a faithful irreducible 
$\FF G$-representation in characteristic $\ell$ coprime to $q$ of dimension $>4$. Then $\Sym^{k}(V)$ is 
reducible for every $k \geq 4$.}
\end{propo}

\begin{proof}
Assume the contrary: $\SM(V)$ is irreducible for some $m \geq 4$. Then it is clear that 
\begin{equation}\label{max}
  \dim(V)^{4} < 24\ml(G)~.
\end{equation}
We will also use the estimates $d := \dim(V) \geq \dl(S)$, $\ml(G) \leq \ml(\HS) \cdot |\Out(S)|$, 
and $|\Out(S)| \leq q(q+1)$, where $\HS$ is the universal cover of $S$. If 
$$(n,q) \in \{ (2,4), (2,5), (2,7), (2,9), (2,11), (3,3), (3,4), (4,2), (4,3), (5,2)\},$$
then using \cite{Atlas} and \cite{JLPW} it is straightforward to check that either (\ref{max})
cannot hold, or else $\SK(V)$ is reducible for all $k \geq 4$. The same applies to $(n,q) = (5,3)$,
where we use the bound $\ml(S) \leq \ml(GU_{5}(q)) = q(q+1)(q^{4}-1)(q^{5}+1)$ that follows from 
\cite{Noz}. Henceforth we will assume that $(n,q)$ is none of the above pairs.    

1) First let $S = PSL_{2}(q)$. If $2|q$, then $d = \dim(V) \geq q-1$ and $\ml(G) \leq q(q+1)/2$,
violating (\ref{max}) as $q \geq 8$. If $q$ is odd, then $d = \dim(V) \geq (q-1)/2$ and 
$\ml(G) \leq 2q(q+1)/9$, violating (\ref{max}) as $q \geq 13$. If $S = PSU_{3}(q)$, then 
$d \geq q(q-1)$ and $\ml(G) \leq 3q(q+1)(q^{2}-1)$, violating (\ref{max}) as $q \geq 5$.
If $S = PSU_{4}(q)$, then $d \geq (q^{2}+1)(q-1)$ and $\ml(G) \leq 4q(q+1)(q^{2}+1)(q^{3}+1)$, 
violating (\ref{max}) as $q \geq 4$. If $S = PSU_{5}(q)$, then 
$d \geq q(q^{2}+1)(q-1)$ and $\ml(G) \leq 5q(q+1)(q^{4}-1)(q^{5}+1)$, violating (\ref{max}) as 
$q \geq 4$. So we may assume $n \geq 7$ and $q$ is even.

2) Next we consider the case $n \geq 9$ and $m \geq 6$. Without loss we may assume $L = SU_{n}(q)$. 
Consider a long-root subgroup $Z_{1}$ of $M = SU_{n-1}(q)$, and its centralizer 
$C_{1} := C_{M}(Z_{1}) = Q_{1}.(SU_{n-3}(q).\ZZ_{q+1})$ in $M$, of index 
$q^{n-1}(q^{n}+1)(q^{n-1}-1)(q^{n-2}+1)/(q+1)$ in $L$. By Proposition \ref{root}, the condition 
$\sta$ and the assumptions of Proposition \ref{key2} hold for $C_{1}$ as $n-1$ is even. Now we will 
argue as in the proof of Proposition \ref{key2}, and denote $V_{\lam} = \EL \otimes A_{1}$, 
$E_{1} := \dim(\EL) = \sqrt{|Q_{1}/Z_{1}|} = q^{n-3}$, $a_{1} := \dim(A_{1})$, $v_{1} := \dim(V_{\nu})$ 
for some $\nu \in \OV \setminus \{\lam\}$.  

First assume that $m = 2k$. Then $\Sym^{2k}(V)$ has a $C_{1}$-submodule $F$ of dimension
$\begin{pmatrix}\dvl+2k-1\\2k \end{pmatrix}$. A simple submodule of $F$ will certainly extend to
$Z(G)C_{1}$. So by Frobenius' reciprocity, $\dim(\Sym^{2k}(V)) \leq \dim(F) \cdot (G:Z(G)C_{1})$. 
Notice that if $a_{1} \geq 21$ then $\begin{pmatrix}a_{1}+5\\6 \end{pmatrix} < a_{1}^{6}/360$ and so  
$$\frac{\dim(\Sym^{2k}(V))}{\dim(F)} = 
  \frac{\begin{pmatrix}d+2k-1\\2k \end{pmatrix}}
  {\begin{pmatrix}a_{1}+2k-1\\2k \end{pmatrix}} \geq  
  \frac{\begin{pmatrix}d+5\\6 \end{pmatrix}}
  {\begin{pmatrix}a_{1}+5\\6 \end{pmatrix}} > 
  \frac{(E_{1}a_{1})^{6} \cdot (d/E_{1}a_{1})^{6}}{2a_{1}^{6}} > E_{1}^{6}/2.$$
Next we consider the case $n = 2k+1$ is odd. Then $\Sym^{2k+1}(V)|_{C_{1}}$ contains the submodule 
$F' := \Sym^{2k}(A_{1}) \otimes V_{\nu}$. By Frobenius' reciprocity,
$\dim(\Sym^{2k+1}(V)) \leq \dim(F') \cdot (G:Z(G)C_{1})$. Again,
$$\frac{\dim(\Sym^{2k+1}(V))}{\dim(F')} = 
  \frac{\begin{pmatrix}d+2k\\2k+1 \end{pmatrix}}
  {\begin{pmatrix}a_{1}+2k-1\\2k \end{pmatrix} \cdot v_{1}} \geq 
  \frac{\begin{pmatrix}d+6\\7 \end{pmatrix}}
  {\begin{pmatrix}a_{1}+5\\6 \end{pmatrix} \cdot v_{1}},$$ 
which is at least 
$$\frac{(Ea_{1}+v_{1})^{7}/5040}{a_{1}^{6}v_{1}/360} = E_{1}^{6} \cdot \frac{1}{14} \cdot 
    \frac{(1+v_{1}/E_{1}a_{1})^{7}}{v_{1}/E_{1}a_{1}} > E_{1}^{6}/2$$
if $a_{1} \geq 21$ (since $(1+v_{1}/E_{1}a_{1})^{7} > 7v_{1}/E_{1}a_{1}$). We have shown that if 
$a_{1} \geq 21$, then 
$$q^{6n-18} = E_{1}^{6} < 2(G:Z(G)C_{1}) \leq 
  \frac{2q^{n}(q^{n}+1)(q^{n-1}-1)(q^{n-2}+1)(n,q+1)}{q+1}$$
which is impossible as $n \geq 9$. Thus $a_{1} \leq 20$, which means that $\dim(\VL) \leq 20E_{1}$ for 
all $\lam \in \OSV$. In this case, the $(-1)$-eigenspace for $1 \neq t \in Z$ on $V$ has dimension
$a_{1}(q/2)E_{1} \leq 10q^{n-2}$. Notice that $n$ conjugates of $t$ generates $L$ by \cite{GS}. Hence 
$d \leq 10nq^{n-2}$ by \cite[Lem. 3.2]{GT3}. This in turn implies by \cite[Thm. 2.7]{GMST} that 
every composition factor of $V|_{L}$ is trivial or a Weil module. Since $V$ is primitive, we conclude 
that $V|_{L}$ is in fact a Weil module.  

3) Here we consider the case $n \geq 9$ and $m \leq 5$ and return to the notation 
$V_{\lam} = \EL \otimes A$, $E := \dim(\EL) = \sqrt{|Q/Z|} = q^{n-2}$, $a := \dim(A)$, $v := \dim(V_{\nu})$ 
for some $\nu \in \OV \setminus \{\lam\}$, where $Z$ is a long-root subgroup in $L$, $C := C_{L}(Z)$ 
and $Q := O_{2}(C)$. 

First assume that $a \geq 4$. Then $\Sym^{4}(V)|_{C}$ contains the submodules 
$$\SE(\EL \otimes A) \supset \wedge^{4}(\EL) \otimes \wedge^{4}(A) \supset \wedge^{4}(A)$$ 
as $\EL$ is of type $-$ by Proposition \ref{key2}. Hence if $m = 4$, then by Frobenius' reciprocity,
$\dim(\Sym^{4}(V)) \leq \dim(\wedge^{4}(A)) \cdot (G:Z(G)C)$. Furthermore,   
$$\frac{\dim(\Sym^{4}(V))}{\dim(\wedge^{4}(A))} = 
  \begin{pmatrix}d+3\\4 \end{pmatrix}\left/ \right.\begin{pmatrix}a\\4 \end{pmatrix} > E^{4}.$$
Next we consider the case $m = 5$. Then $\Sym^{5}(V)|_{C}$ contains the submodules 
$\Sym^{4}(\EL \otimes A) \otimes V_{\nu} \supset \wedge^{4}(A) \otimes V_{\nu}$. 
By Frobenius' reciprocity, $\dim(\Sym^{5}(V)) \leq \dim(\wedge^{4}(A) \otimes V_{\nu}) \cdot (G:Z(G)C)$.
Furthermore,   
$$\frac{\dim(\Sym^{5}(V))}{\dim(\wedge^{4}(A) \otimes V_{\nu})} = 
  \begin{pmatrix}d+4\\5 \end{pmatrix}\left/ \right.
  \left(\begin{pmatrix}a\\4 \end{pmatrix} \cdot v\right) >
  \frac{(Ea+v)^{5}/120}{a^{4}v/24} = E^{4} \cdot \frac{(1+v/Ea)^{5}}{5v/Ea} > E^{4}.$$
It follows that in both cases,
$$q^{4n-8} = E^{4} < (G:Z(G)C) \leq q(q^{n}+1)(q^{n-1}-1),$$
a contradiction since $n \geq 9$.
 
Thus $1 \leq a \leq 3$. Arguing as in 2), we conclude that $V|_{L}$ is in fact a Weil module.   
  
4) Assume that $n = 7$ and $V|_{L}$ contains a non-Weil irreducible constituent. Then
by \cite[Thm. 2.7]{GMST}, $d \geq (q^{7}+1)(q^{6}-q^{2})/(q^{2}-1)(q+1)-1 > (5/3)q^{9}$, whereas
$\ml(G) \leq 7q^{28}$, contradicting (\ref{max}). As in 3), we can conclude that $V|_{L}$ is again a 
Weil module.

5) We have shown that $n \geq 7$ and $V|_{L}$ is a Weil module. Hence $V_{L}$ lifts to a  
complex Weil module of $L$, with character $\zeta := \zeta^{i}_{n,q}$ for some $i$, $0 \leq i \leq q$. 
These characters, together with their branching to $M = SU_{n-1}(q)$, are described in \cite{TZ2}. 

If $i = 0$, then $\zeta^{0}_{n,q}|_{M}$ contains $\alpha + \bar{\alpha}$, with 
$\alpha := \zeta^{1}_{n-1,q}$, whence
$$\Sym^{2k}(\zeta)|_{M} \supset \SK(\alpha) \otimes \SK(\bar{\alpha}) \supset 1_{M}.$$
If in addition $q > 2$, then in fact $\zeta^{0}_{n,q}|_{M}$ contains $\alpha + \bar{\alpha} + \beta$, 
with $\beta := \zeta^{2}_{n-1,q}$, whence
$$\Sym^{2k+1}(\zeta)|_{M} \supset \SK(\alpha) \otimes \SK(\bar{\alpha}) \otimes \beta \supset \beta.$$
If $i \neq 0$, then $\zeta^{i}_{n,q}|_{M}$ contains $\beta + \gam$, with 
$\beta := \zeta^{j}_{n-1,q}$ for some $j \neq i,0$, and $\gamma := \zeta^{0}_{n-1,q}$ is of type 
$+$, whence
$$\Sym^{2k}(\zeta)|_{M} \supset \Sym^{2k}(\gam) \supset 1_{M},~~~~
  \Sym^{2k+1}(\zeta)|_{M} \supset \Sym^{2k}(\gam) \otimes \beta \supset \beta.$$   
It follows that $\Sym^{2k}(V)$ has a composition factor of dimension
$\leq (G:Z(G)M) \leq q^{n}(q^{n}+1)(q+1)$. If $k \geq 2$, then the latter is less than 
$\dim(\Sym^{2k}(V))$ as $d \geq (q^{n}-q)/(q+1)$, whence $\Sym^{2k}(V)$ is reducible. So 
$m = 2k+1 \geq 5$. In this case, if $(q,i) \neq (2,0)$ then $\Sym^{2k+1}(V)$ has a composition factor 
of dimension $\leq (G:Z(G)M)\beta(1) \leq q^{2n-1}(q^{n}+1)$ which is again less than 
$\dim(\Sym^{2k+1}(V))$, whence $\Sym^{2k+1}(V)$ is reducible. 

Finally, assume that $(q,i) = (2,0)$. We aim to show that $V$ lifts to a complex module, in which case 
$\SM(V)$ is reducible for every $m \geq 4$ by Corollary \ref{sym2} as we have already shown that 
$\Sym^{4}(V)$ is reducible. Since $V|_{L}$ affords the character $\hat{\zeta}$, by Lemma \ref{lift} it 
suffices to show that $\zeta$ extends to $G$. Here $L = S = PSU_{n}(2)$ as $\zeta^{0}_{n,q}$ is trivial at 
$Z(SU_{n}(q))$; also $Z(G)S = Z(G) \times S$. Next, $\zeta$ is the unique irreducible character of $L$ of
degree $(2^{n}-2)/3$, so it is invariant under $\Aut(S) = PGU_{n}(2) \cdot 2$. Since $\zeta$ is real-valued 
and $L$ has odd index in $H := PGU_{n}(2)$, by \cite[Lem. 2.1]{NT}, it has a unique real-valued extension 
$\tilde{\zeta}$ to $H$. Now $\tilde{\zeta}$ is the unique irreducible, real-valued, character of degree 
$(2^{n}-2)/3$ of $H$, hence it is invariant under $H \cdot 2$. Observe that if $|G/(Z(G) \times S)| \leq 2$,
then $\zeta \otimes 1_{Z(G)}$ is $G$-invariant and so it extends to $G$ as required. In particular we 
are done if $(n,3) = 1$ as in this case $H = S$. Assume $3|n$ and $|G/Z(G)S| > 2$; in particular 
$\zeta(1) = d$ is coprime to $3$. By Proposition \ref{red} we may replace $G$ by its normalized version
$\GN$ and assume that $\det(\Phi(g)) = \pm 1$ for all $g \in G$, if $\Phi$ denotes the representation 
of $G$ on $V$. But $(d,3) = 1$ and $\Phi$ is faithful, so $O_{3}(Z(G)) = 1$. Let $K$ be the 
complete inverse image of $H = S:3$ in $G$; in particular, $|G/K| \leq 2$. Since $O_{3}(Z(G)) = 1$, we 
see that $K \simeq Z(G) \times H$. Now $\tilde{\zeta} \otimes 1_{Z(G)}$ is an extension of $\zeta$ to 
$K$ which is $G$-invariant, and so it extends to $G$ as required.       
\end{proof}

\begin{propo}\label{weil}
{\sl Let $S \leq G/Z(G) \leq \Aut(S)$ for the simple group $S = PSp_{2n}(q)$, with 
$n \geq 2$ and $q \equiv 3 (\mod 4)$ is odd. Let $V$ be a faithful irreducible 
$\FF G$-module in characteristic $\ell$ coprime to $q$ and assume that an irreducible constituent of 
$V|_{L}$ is a Weil module. Then $X^{k}(V)$ is reducible for every $k \geq 4$ and 
$X \in \{\Sym,\wedge\}$, except for $(n,q,X) = (2,3,\wedge)$.}
\end{propo}

\begin{proof}
Assume the contrary. The case $(n,q) = (2,3)$ can be checked directly, so we will assume that 
$(n,q) \neq (2,3)$. Now $V$ is primitive, hence $V|_{L}$ is a Weil module of dimension 
$(q^{n} \pm 1)/2$. In particular, $V|_{L}$ lifts to a complex Weil module $W$. Notice that, since 
$q \equiv 3 (\mod 4)$, any field automorphism of $L$ has odd order, and so it stabilizes each of the 
two complex Weil modules of dimension $d$. Since the character of $V$ takes different values at 
the two $L$-classes of transvections, cf. \cite{TZ2}, and they are fused under the outer 
diagonal automorphism $\gam$ of $L$, we see that $G/Z(G)$ cannot induce $\gam$. Thus $G/Z(G)S$ 
is cyclic and induces only field automorphisms of $S$. We conclude that $W$ extends to a complex 
module of $G$. It follows by Lemma \ref{lift} that $V$ lifts to a complex module which without loss
we will denote also by $V$. So by Corollary \ref{sym2} and Lemma \ref{wedge}, 
$X^{4}(V)$ must be irreducible. Consider a 
long-root subgroup $Z$ in $L$ and let $P := N_{L}(Z)$, $C := [P,P]$, $Q := O_{p}(P)$. Write 
$\dim(V) = (q^{n}+\eps)/2$ for some $\eps = \pm 1$. As we have shown, $(G:Z(G)L) \leq f$
if $q = p^{f}$ for a prime $p$. Notice that $V|_{L} = V_{1} \oplus \sum_{\lam \in \OSV}\EL$ and 
$V_{1}$ is actually a $P/Q$-module of dimension $(q^{n-1}+\eps)/2$.
 
1) First we consider the case $(q,3) = 1$. We may identify $\OSV$ with the subset 
$A := \{x^{2} \mid x \in \FQB\}$. Claim that there are $a,b \in A$ such that $a+b+1 = 0$ and 
$(a,b) \neq (1,1)$. Indeed, the two subsets $\{1+x^{2} \mid x \in \FQ\}$ and $\{-y^{2} \mid y \in \FQ\}$
of $\FQ$ both have cardinality $(q+1)/2$, hence they intersect, i.e. $1 + x^{2}+ y^{2} = 0$ for some 
$x,y \in \FQ$. Since $q \equiv 3 (\mod 4)$, $xy \neq 0$, so $a := x^{2} \in A$, $b := y^{2} \in B$. 
Also, $(a,b) \neq (1,1)$ as $(q,3) = 1$. In the character language, this means that there are
$\alpha, \beta,\gam \in \OSV$ such that $\alpha\beta\gam = 1_{Z}$ and $\alpha \neq \beta$.   

Assume $\gam \notin \{\alpha,\beta\}$. Then 
$E_{\alpha} \otimes E_{\beta} \otimes E_{\gam}$ affords the $Q$-character $q^{n-1}\rho_{Q/Z}$, where
$\rho_{Y}$ denotes the regular character of a finite group $Y$. In particular, the subspace $F$ of 
$Q$-fixed points on $E_{\alpha} \otimes E_{\beta} \otimes E_{\gam}$ has dimension $q^{n-1}$, and it is
stabilized by $C$. Next assume that $\gam = \beta$. Then $E_{\alpha} \otimes X^{2}(E_{\beta})$ affords the
$Q$-character $(q^{n-1} \pm 1)/2 \cdot \rho_{Q/Z}$. In particular, the subspace $F$ of $Q$-fixed points on 
$E_{\alpha} \otimes X^{2}(E_{\beta})$ has dimension $(q^{n-1}+1)/2$, and it is stabilized by $C$.
Notice that $X^{4}(V)|_{C}$ contains $E_{\alpha} \otimes E_{\beta} \otimes E_{\gam} \otimes V_{1}$, 
resp. $E_{\alpha} \otimes X^{2}(E_{\beta}) \otimes V_{1}$. Thus we have shown that $X^{4}(V)|_{C}$ has 
a submodule $F \otimes V_{1}$ of dimension $\leq q^{n-1}(q^{n-1}+\eps)/2$. In fact, $F \otimes V_{1}$ is
a $C/Q$-module, so a simple submodule of it extends to $P$, as $P/Q = (C/Q) \times \ZZ_{q-1}$. 
Hence by Frobenius' reciprocity,
$$\dim(X^{4}(V)) \leq \dim(F \otimes V_{1}) \cdot (G:Z(G)P) \leq \frac{q^{n-1}(q^{n-1}+\eps)(q^{2n}-1)f}
  {2(q-1)},$$
which is a contradiction since $n \geq 2$ and $p \geq 7$.      
 
2) Next we consider the case $q = 3^{f} > 3$. Fix $\lam \in \OSV$. Since $Q/\Ker(\lam)$ has exponent 
$3$, direct calculation shows that the subspace $F$ of $Q$-fixed points on $X^{3}(\EL)$ has 
dimension $(q^{n-1} \pm 1)/2$, and it is stabilized by $C$. As above, $X^{4}(V)|_{C}$ contains 
$X^{3}(\EL) \otimes V_{1}$ and so it has a $C$-submodule $F \otimes V_{1}$, some simple submodule of 
which extends to $P$. Hence by Frobenius' reciprocity,
$$\dim(X^{4}(V)) \leq \dim(F \otimes V_{1}) \cdot (G:Z(G)P) \leq 
  \frac{(q^{n-1}+1)(q^{n-1}+\eps)(q^{2n}-1)f}{4(q-1)},$$
which is a contradiction since $n \geq 2$ and $q \geq 27$.

3) Now we may assume that $q = 3$ and $G = Sp_{2n}(3)$. Recall that $G$ has four Weil characters,
$\xi$, $\xib$ of degree $(3^{n}+1)/2$, and $\eta$, $\etab$ of degree $(3^{n}-1)/2$. By 
\cite[Prop. 5.4]{MT1}, $\SB(\xi) = \SB(\xib)$, $\WB(\eta) = \WB(\etab)$, $\SB(\eta)$, $\SB(\etab)$, 
$\WB(\xi)$, $\WB(\xib)$, $\xi\eta = \xib\etab$, $\xi\etab \neq \xib\eta$ are all irreducible. Next, 
$$1 = (\xi\eta,\xib\etab) = (\xi^{2},\etab^{2}) = (\SB(\xi) + \WB(\xi),\SB(\etab) + \WB(\etab)),$$
whence $\WB(\xi) = \SB(\etab)$. On the other hand,
$0 = (\xi\etab,\xib\eta) = (\xi^{2},\eta^{2})$, so 
$$\SB(\etab) = \WB(\xi) \neq \SB(\eta) = \WB(\xib).$$  
It follows that 
$$\begin{array}{c}(\SB(\xi) \otimes \xi,\xib) = (\SB(\xi),\SB(\xib)+\WB(\xib)) = 1,\\
  (\WB(\xi) \otimes \xi,\xib) = (\WB(\xi),\SB(\xib)+\WB(\xib)) = 0.\end{array}$$
Notice that 
$$\SB(\xi) \otimes \xi = \SD(\xi) + {\mathbb S}_{2,1}(\xi),~~~
  \WB(\xi) \otimes \xi = \WD(\xi) + {\mathbb S}_{2,1}(\xi),$$
where ${\mathbb S}_{2,1}$ is a Schur functor, cf. \cite{FH}. Consequently, $\SD(\xi)$ contains $\xib$ 
with multiplicity $1$ and so it is reducible. Similarly, 
$$\begin{array}{c}(\WB(\eta) \otimes \eta,\etab) = (\WB(\eta),\WB(\etab)+\SB(\etab)) = 1,\\
  (\SB(\eta) \otimes \eta,\etab) = (\SB(\eta),\WB(\etab)+\SB(\etab)) = 0,\end{array}$$
whence $\WD(\eta)$ contains $\etab$ with multiplicity $1$ and so it is reducible. 

4) Finally, we will use the Deligne-Lusztig theory, cf. \cite{DM}, to show that $\SE(\eta)$ and 
$\WE(\xi)$ are reducible. If $n = 3$ then this can be verified using \cite{Atlas}. Notice that 
$$\begin{array}{c}D_{1} := \dim(\SE(\eta)) = (3^{2n}-1)(3^{n-1}+1)(3^{n}+5)/128,\\
  D_{2} := \dim(\WE(\xi)) = (3^{2n}-1)(3^{n-1}-1)(3^{n}-5)/128\end{array}$$
are both coprime to $3$, so it suffices to show that they are not equal to the degree of any semisimple
character of $G$. If $n = 4$, then $43|D_{1}$ and $19|D_{2}$ but $(|G|, 43 \cdot 19) = 1$. If $n = 5$, 
then $31|D_{1}$ and $17|D_{2}$ but $(|G|, 31 \cdot 17) = 1$. If $n = 6$, then $367|D_{1}$ and 
$181|D_{2}$ but $(|G|, 367 \cdot 181) = 1$. So we may assume that $n \geq 7$. 

Consider the dual group $G^{*} = SO_{2n+1}(3)$ and its natural module $N = \FF_{3}^{2n+1}$. We need to
show that there is no semisimple element $s \in G^{*}$ such that $E := (G^{*}:C_{G^{*}}(s))_{3'}$ 
equals to $D_{1}$ or $D_{2}$. Assume the contrary. For each $m \geq 3$, by \cite{Zs} there is a prime 
$\ell_{m}$ that divides $3^{m}-1$ but not $\prod^{m-1}_{i=1}(3^{i}-1)$. 

Claim that if $C:=C_{G^{*}}(s)$ preserves any orthogonal decomposition $N = N^{1} \perp N^{2}$ with 
$\dim(N^{1}) \geq \dim(N^{2}) \geq 1$ then $\dim(N^{2}) \leq 3$. Otherwise $\ell_{n-1}\ell_{2n-2}$ 
divides $E$, but $(\ell_{n-1},D_{1}) = (\ell_{2n-2},D_{2}) = 1$, a contradiction. Furthermore,
if $\dim(N^{2}) = 3$, then $E$ is divisible by $(3^{2n}-1)(3^{n-1} \pm 1)/16$, whence $E \neq D_{2}$ as 
$3^{n} \not\equiv 5 (\mod 8)$.         

5) Observe that the eigenspaces of $s$ on $N \otimes_{\FF_{3}} \overline{\FF}_{3}$ gives rise to 
a $C$-invariant orthogonal decomposition $N = N_{+} \perp N_{1} \ldots \perp N_{t}$, where 
$N_{+} := \Ker(s-1)$ has odd dimension. Furthermore, if $1 \leq i \leq t$ then no eigenvalues of $s$ on 
$N_{i}$ are equal to $1$, and $C_{GO(N_{i})}(s) = GL_{a}(3^{b})$ or $GU_{a}(3^{b})$ with 
$\dim(N_{i}) = 2ab$ if $s+1$ is non-degenerate on $N_{i}$. We label $\Ker(s+1)$ by $N_{t}$ if it is 
nonzero. The above claim implies that one of the following two cases must occur.

Case 1: $t = 2$, $\dim(N_{+}) = 1$, $\{\dim(N_{1}),\dim(N_{2})\} = \{2,2n-2\}$. Notice that 
$\dim(\Ker(s+1)) \leq 2n-4$ as otherwise $C \geq SO^{\pm}_{2n-2}(3)$ and so $E$ divides 
$(3^{2n}-1)(3^{n-1} \pm 1)$, a contradiction. Hence if we label $N_{1}$ to have dimension $2n-2$ then
$s|_{N_{1}}$ is a semisimple element $s_{1}$ with no eigenvalue equal to $\pm 1$. Also by our claim in 4) 
(taking $N^{2} = N_{+} \oplus N_{2}$), we see that $E = D_{1}$, $N_{1}$ is of type $+$, and $E_{1}$ 
divides $(3^{n}+5)/8$ and is at least $(3^{n}+5)/128 > 2$, where 
$E_{1}:= (SO(N_{1}):C_{SO(N_{1})}(s_{1}))_{3'}$. Since $SO(N_{1}) = SO^{+}_{2n-2}(3)$ is self-dual, 
it follows that it has an irreducible character of degree $E_{1}$, contradicting \cite{TZ1}.    

Case 2: $t = 1$, and $\dim(N_{+}) = 1$, $3$, $2n-1$. If $\dim(N_{+}) = 3$ then we can argue as 
in Case 1 to get a contradiction. If $\dim(N_{+}) = 2n-1$, then $C \geq SO_{2n-1}(3)$ and so $E$ divides 
$3^{2n}-1$, a contradiction. Thus $\dim(N_{+}) = 1$, $\dim(N_{1}) = 2n$, $C = GU_{a}(3^{b})$ or 
$GL_{a}(3^{b})$ with $n = ab$. In the former case $E$ is coprime to $\ell_{2n}$, but $\ell_{2n}$ divides
both $D_{1}$ and $D_{2}$, a contradiction. In the latter case 
$E > 3^{n(n+1)/2} \geq 3^{4n} > \max\{D_{1}, D_{2}\}$, again a contradiction.
\end{proof}

\begin{remar}\label{rajan}
{\em The identities $\SB(\xi) = \SB(\xib)$ and $\WB(\eta) = \WB(\etab)$ mentioned in the proof of Proposition
\ref{weil} (there are similar examples with $Sp_{2n}(5)$ as well, cf. \cite{MT1}), show that (irreducible) 
representations of finite quasisimple groups cannot be recovered from their symmetric square, resp. 
exterior square, as opposed to complex simple Lie groups, see \cite{Ra2}.}
\end{remar}

\begin{lemma}\label{even}
{\sl Let $S \leq G/Z(G) \leq \Aut(S)$, where either $S = PSp_{2n}(q)$ with $n \geq 2$ and $q$ is even, 
or $S \in \{ F_{4}(q), \ta F_{4}(q)'\}$ with $q = 2^{f}$, 
or $S = \ta B_{2}(q)$ with $q = 2^{f} \geq 8$, or $S \in \{G_{2}(q), \ta G_{2}(q)'\}$ with $q = 3^{f}$.
Let $V$ be a faithful irreducible $\FF G$-representation in characteristic $\ell$ coprime 
to $q$. Then $\Sym^{k}(V)$ is reducible for every $k \geq 4$.}
\end{lemma}

\begin{proof}
Assume the contrary. Then we can apply (\ref{max}) to $G$ and $V$.

1) First we consider the case $S = PSp_{2n}(q)$. Notice that 
$d := \dim(V) \geq \dl(S) = (q^{n}-1)(q^{n}-q)/2(q+1)$, cf. \cite{GT1}, and 
$\ml(G) \leq q\prod^{n}_{i=1}(q^{2i}-1)/(q-1)^{n}$ as $|\Out(S)| \leq q$. Hence (\ref{max}) cannot
hold for $2 \leq n \leq 5$ and $q \geq 4$. The cases $(n,q) = (2,2)$, $(3,2)$, or $(4,2)$ can be 
checked directly.

Assume $(n,q) = (5,2)$; in particular $\Aut(S) = S$ and so we may assume $G = S$. Then (\ref{max}) 
implies that $d \leq 365$. By \cite[Thm. 1.1]{GT1}, $V$ is a unitary-Weil module. We restrict $V$ to 
$P_{1} = Q_{1}.Sp_{8}(2)$, the stabilizer of a $1$-space in the natural module of $G$. Notice that
any complex unitary-Weil character of $G$ is real-valued, and by \cite[Prop. 7.4]{GT1} its restriction 
to $P_{1}$ contains some complex unitary-Weil character of $Sp_{8}(2)$ (which is of type $+$ by 
\cite{Atlas}) with multiplicity $1$. It follows that any complex unitary-Weil module of $G$ is of type 
$+$. But $V$ is a composition factor of multiplicity $1$ in a complex unitary-Weil module, hence $V$ 
is of type $+$, a contradiction by Lemma \ref{so-sym}(i).

Thus we may assume $n \geq 6$. Since $\Mult(S) = 1$, $S \lhd G$. Consider the subgroup 
$M = SO^{-}_{2n}(q)$ in $S$, its long-root subgroup $Z$ and 
$C := C_{M}(Z) = Q.(SL_{2}(q) \times SO^{-}_{2n-4}(q))$. Then the assumptions of Proposition
\ref{key2} hold for $C$, and moreover $\dim(V) > (5/2) \dim(\EL)$ and $|\Out(S)| \leq q/2$. Hence 
the proof of Proposition \ref{key2} implies that
$$q^{8n-16} < (9/2) \cdot (G:Z(G)C) \leq (9/8) \cdot 
  q^{n+1}(q^{2n}-1)(q^{2n-2}-1)(q^{n-2}-1)/(q^{2}-1),$$
a contradiction as $n \geq 6$.  

2) Assume $S = F_{4}(q)$. If $q \geq 4$, then $d \geq \dl(S) = q^{2}(q^{3}-1)(q^{6}-q^{3}+q^{2}-1)/2$ 
by \cite{T2}, and $\ml(G) \leq q^{27}$ as $|\Out(S)| \leq q$, contradicting (\ref{max}). If $q = 2$, 
then (\ref{max}) implies $d \leq 162$, whence $d = 52$ and $V|_{L}$ is of type $+$ by \cite{HM}, and so 
$\SK(V)$ is reducible. Assume $S = \ta F_{4}(q)$ and $q \geq 8$. Then 
$d \geq \dl(S) = (q^{4}+q^{3}+q)(q-1)\sqrt{q/2}$ by \cite{T2}, and $\ml(G) \leq q^{14}/2$ as 
$|\Out(S)| \leq q/2$, contradicting (\ref{max}). Assume $S = G_{2}(q)$ and $q \geq 9$. Then 
$d \geq \dl(S) = q^{4}+q^{2}$, cf. \cite{T2}, and $\ml(G) < q^{8}$ as $|\Out(S)| < q$, contrary to 
(\ref{max}). Assume $S = \ta G_{2}(q)$ and $q \geq 27$. Then $d \geq \dl(S) = q(q-1)$ and 
$\ml(G) \leq q^{9/2}/9$ as $|\Out(S)| \leq q/9$, contrary to (\ref{max}). Assume $S = \ta B_{2}(q)$ and 
$q \geq 32$. Then $d \geq \dl(S) = (q-1)\sqrt{q/2}$ and $\ml(G) \leq q^{7/2}/2$ as $|\Out(S)| < q/2$, 
contrary to (\ref{max}). The cases $S = \ta F_{4}(2)'$, $G_{2}(3)$, $\ta G_{2}(3)'$, or $\ta B_{2}(8)$ 
can be checked directly.
\end{proof}

\subsection{Generic cases.} 
In view of the results of \S\ref{nongeneric}, we may now assume that $C_{L}(Z)$ satisfies the 
assumptions of Propositions \ref{key1} and \ref{key2} for a long-root subgroup $Z$ of $G$. Assume that
$\SK(V)$ is irreducible for some $k \geq 4$ and $d := \dim(V) > 4$. We will then apply (\ref{max}) and 
Propositions \ref{key1} and \ref{key2} to $G$ and $V$ to get a contradiction. 

Suppose that $S = PSL_{n}(q)$, $n \geq 3$. If $(n,q) = (3,2)$, $(3,3)$, $(4,2)$, $(4,3)$, 
or $(5,2)$, then it is easy to check that $\SK(V)$ is reducible for $k \geq 4$. Assume $(n,q) = (3,4)$. 
Then again it is easy to check that $\SK(V)$ is reducible for $k \geq 4$, except possibly when $d = 6$ and
$L=6 \cdot S$. In this case $V$ lifts to a complex module $\VC$, and $\SE(\VC)|_{L}$ contains irreducible 
constituents of dimensions $21$ and $84$, so $\SK(V)$ is reducible for $k \geq 4$. In all other cases, 
$d \geq \dl(S) \geq (q^{n}-q)/(q-1)-1$, and $\ml(G) \leq q^{(n^{2}+3)/2}$ as 
$|\Out(S)| \leq q(q-1)$. Hence (\ref{max}) cannot hold for $n \leq 5$. Assume $n \geq 6$. In the notation
of Proposition \ref{key2} we have $\dim(\EL) = q^{n-2} < d/2$. Hence Propositions \ref{key1} and
\ref{key2} imply 
$$q^{4n-8} = (Q:Z)^{2} < (15/2) \cdot (G:Z(G)C) \leq (15/2) \cdot (q^{n}-1)(q^{n}-q),$$
a contradiction.

Suppose that $S = PSU_{n}(q)$, $n \geq 6$, $(n,q) \neq (6,2)$, and $q$ is odd if $n$ is odd. Then 
$$q^{4n-8} = (Q:Z)^{2} < 24 \cdot (G:Z(G)C) \leq 24 \cdot (q^{n}-(-1)^{n})(q^{n}+(-1)^{n}q),$$
a contradiction. 

Suppose that $S = PSp_{2n}(q)$, $n \geq 2$, $q$ is odd, $(n,q) \neq (2,3)$, $(2,5)$ 
(and $q \equiv 1 (\mod 4)$ if $V|_{L}$ involves a Weil module). Then 
$$q^{4n-4} = (Q:Z)^{2} < (3/2) \cdot (G:Z(G)C) \leq (3/4) \cdot (q^{2n}-1)q,$$
a contradiction if $n \geq 3$. If $n = 2$, then $q \geq 7$, $d \geq \dl(S) \geq (q^{2}-1)/2$, and 
$\ml(G) \leq 4q(q^{2}-1)(q^{4}-1)/9(q-1)^{2}$ as $|\Out(S)| \leq 4q/9$, whence (\ref{max}) cannot hold.

Suppose that $S = \Om_{2n+1}(q)$, $n \geq 3$, $q$ is odd, $(n,q) \neq (3,3)$. Then 
$$q^{8n-12} = (Q:Z)^{2} < (3/2) \cdot (G:Z(G)C) \leq (3/2) \cdot (q^{2n}-1)(q^{2n-2}-1)q/(q^{2}-1),$$
a contradiction. 

Suppose that $S = P\Om^{\eps}_{2n}(q)$, $n \geq 4$, $(n,q) \neq (4,2)$. Then 
$d \geq (q^{n}-1)(q^{n-1}-q)/(q^{2}-1) > 2\dim(\EL)$ and $|\Out(S)| \leq 12q$, whence  
$$q^{8n-16} = (Q:Z)^{2} < (15/2) \cdot (G:Z(G)C) \leq (15/2) \cdot 
  \frac{(q^{n}-1)(q^{2n-2}-1)(q^{n-2}+1)}{q^{2}-1} \cdot 12q,$$
a contradiction. 

Suppose that $S = E_{8}(q)$. Then 
$$q^{112} = (Q:Z)^{2} < 24 \cdot (G:Z(G)C) \leq 12q \cdot 
  \frac{(q^{20}-1)(q^{24}-1)(q^{30}-1)}{(q^{6}-1)(q^{10}-1)},$$
a contradiction. If $S = E_{7}(q)$, then 
$$q^{64} = (Q:Z)^{2} < 24 \cdot (G:Z(G)C) \leq 24q \cdot 
  \frac{(q^{12}-1)(q^{14}-1)(q^{18}-1)}{(q^{4}-1)(q^{6}-1)},$$
a contradiction. If $S = E_{6}(q)$, then 
$$q^{40} = (Q:Z)^{2} < 24 \cdot (G:Z(G)C) \leq 72q \cdot 
  \frac{(q^{8}-1)(q^{9}-1)(q^{12}-1)}{(q^{3}-1)(q^{4}-1)},$$
a contradiction. If $S = \ta E_{6}(q)$ and $q > 2$, then 
$$q^{40} = (Q:Z)^{2} < 24 \cdot (G:Z(G)C) \leq 72q \cdot 
  \frac{(q^{8}-1)(q^{9}+1)(q^{12}-1)}{(q^{3}+1)(q^{4}-1)},$$
a contradiction. If $S = F_{4}(q)$ with $q$ odd, then 
$$q^{28} = (Q:Z)^{2} < (3/2) \cdot (G:Z(G)C) \leq (3q/4) \cdot (q^{4}+1)(q^{12}-1),$$
a contradiction. If $S = \tb D_{4}(q)$ and $q > 2$, then 
$$q^{16} = (Q:Z)^{2} < 24 \cdot (G:Z(G)C) \leq 36q \cdot (q^{8}+q^{4}+1)(q^{2}-1),$$
a contradiction. Assume $S = G_{2}(q)$ with $(q,3) = 1$ and $q \geq 5$, then
$d \geq q^{3}-1 > 2\dim(\EL)$, whence 
$$q^{8} = (Q:Z)^{2} < (15/2) \cdot (G:Z(G)C) \leq (15/4) \cdot q(q^{6}-1),$$
again a contradiction.

The cases $S = PSp_{4}(3)$, $PSp_{4}(5)$, $\Om_{7}(3)$, $\Om^{\pm}_{8}(2)$, $\tb D_{4}(2)$  can be 
handled easily using \cite{Atlas} and \cite{JLPW}. The case $S = G_{2}(4)$ leads to the example 
$d = 12$, $L = 2 \cdot S$; furthermore, if $\ell = 0$ then $\SK(V)$ is irreducible over $L$ 
for $k \leq 4$ and $\SF(V)$ is irreducible over $L \cdot 2$. Finally, assume $S = \ta E_{6}(2)$. 
Since $\ml(G) \leq (1.66) \cdot 10^{12}$ and $d \geq 1536$, $\SK(V)$ is reducible for $k \geq 5$; also
$\SE(V)$ is reducible if $d \geq 2513$. Assume $d \leq 2512$.  Observe that $G$ contains a subgroup 
$H \in \{F_{4}(2), 2 \cdot F_{4}(2) \}$. The irreducible $\ell$-modular Brauer characters of 
$2 \cdot F_{4}(2)$ are known \cite{ModAt}. In particular, any such a character of degree $\leq 2512$ is 
of type $+$. It follows by Lemma \ref{index}(i) that $d \leq 458$, a contradiction. This completes
the proof of Theorem \ref{cross}. 

\section{Normalizers of Extraspecial Groups}
The aim of this section is to prove the following

\begin{theor}\label{extra}
{\sl Assume $\GC := GL(V)$, $Sp(V)$ or $GO(V)$, $G \leq \GC$ is Zariski closed, $\dim(V) > 4$, and that 
the conclusion {\rm (iii)} of Proposition \ref{red} holds. Then, for $X \in \{\Sym,\wedge\}$, 
$X^{k}(V)$ is reducible, if $k > 2$ and $p > 2$, or if $k > 3$ and $p = 2$, unless 
$(\dim(V),X) = (5,\wedge)$.}
\end{theor}

\begin{proof}
Assume the contrary. Without loss we may replace $G$ by $\GN$. Recall that either $P = E$, or $p = 2$ and 
$P = \ZZ_{4} * E$; furthermore, $d := \dim(V) = \sqrt{|E/Z(E)|} = p^{n}$. It is well known that the 
$\ell$-modular representation $V|_{P}$ is liftable to a complex representation which extends to $G$. Hence 
by Lemma \ref{lift} we may assume that $\ell = 0$.

1) First we consider the case $p > 2$. Then we may assume $G = p^{1+2n}_{+} \cdot Sp_{2n}(p)$ which is
a split extension. Assume $p = 3$. Direct computation shows that the fixed point subspace $\SD(V)^{P}$, 
resp. $\WD(V)^{P}$, has dimension $(d + 1)/2$, resp. $(d - 1)/2$, whence $\SD(V)$ and $\WD(V)$ are reducible.
Now we may assume that $p \geq 5$ and consider the central involution $j$ of $Sp_{2n}(p)$. It is 
well known that $V = V_{+} \oplus V_{-}$, where $j$ acts on $V_{\delta}$ as the scalar $\delta 1$ for 
$\delta = \pm$, and $\{\dim(V_{+}),\dim(V_{-})\} = \{a,a+1\}$ with $a = (d-1)/2$. Set $\eps := 1$ if 
$X = \Sym$ and $\eps := -1$ if $X = \wedge$. By Corollary \ref{sym2} and Lemma \ref{wedge}, $X^{3}(V)$ is 
irreducible, and it affords the $Z(P)$-character $d(d+\eps)(d+2\eps)/6 \cdot \lam^{3}$, if $V|_{Z(P)}$ 
affords the character $d\lam$. Let $V_{3}$ denote the unique complex $P$-representation with 
$Z(P)$-character $d\lam^{3}$. Notice that $V_{3}$ extends to $G$ and, by Clifford theory, 
$X^{3}(V) = V_{3} \otimes A$ for some (irreducible) $G/P$-module $A$ of dimension $(d+\eps)(d+2\eps)/6$. 
If $n = 1$ and $(d,X) \neq (5,\wedge)$, $(7,\wedge)$, then $d = p \geq 5$ and $\dim(C) > p+1 = \ml(G/P)$, 
whence $X^{3}(V)$ is reducible. If $(d,X) = (7,\wedge)$, then $\dim(C) = 5$ does not divide $|G|$ and so 
$X^{3}(V)$ is again reducible. So we may assume $n \geq 2$. Now $j$ acts scalarly on $A$, whence the 
difference $D$ between the dimensions of the $1$-eigenspace and the $(-1)$-eigenspace of $j$ on $X^{3}(V)$ 
must be divisible by $\dim(A)$. On the other hand, since the $1$-eigenspace of $j$ on 
$X^{3}(V)$ is just $X^{3}(V_{+}) \oplus V_{+} \otimes X^{2}(V_{-})$, we see that $|D| = a+1$ if 
$X = \Sym$ and $|D| = a$ if $X = \wedge$, a contradiction. Observe that $(V \otimes V^{*})/1_{G}$ is 
irreducible, therefore both $\SB(V)$ and $\WB(V)$ are irreducible over $G = p^{1+2n}_{+} \cdot Sp_{2n}(p)$, 
cf. \cite{GT2}. 

2) Now we consider the case $p = 2$. Then $E = 2^{1+2n}_{\eps} \leq P \leq \ZZ_{4} *E$ for some 
$\eps = \pm$. Direct computation shows that the fixed point subspace $\SE(V)^{P}$, resp. $\WE(V)^{P}$,
has dimension $(d+1)(d+2)/6$, resp. $(d-1)(d-2)/6$, whence $\SE(V)$ and $\WE(V)$ are reducible. 
Observe however that $\SD(V)$ and $\WD(V)$ are irreducible over $N_{\GC}(E)$ by \cite{GT2}.       
\end{proof} 

\section{Alternating groups, symmetric groups, and their covers} 

This section is devoted to the proof of the following theorem

\begin{theor}\label{alt}
{\sl Theorem \ref{main} holds true in the case $G$ is finite and $S := \soc(G/Z(G))$ is the alternating 
group $\AN$ for some $n \geq 5$.}
\end{theor}

We begin with considering the case $n \geq 8$. Then the representation $\Phi$ of $G$ on $V$ yields a 
projective representation of $G/Z(G) = \AN$ or $\SN$. It follows that $Z(GL(V))G = Z(GL(V))\Psi(H)$, 
where $H = \HAN$, resp. $\TSN$ (the double cover of $\SN$ in which transpositions lift to elements of 
order $4$), and $\Psi~:~H \to GL(V)$ is an irreducible representation. Ignoring the faithfulness of $G$ 
acting on $V$, we may therefore assume that $G \in \{\HAN,\TSN\}$. In view of Lemma \ref{step1} we 
will assume that $\Char(\FF) = \ell > 3$. Whenever we consider a subgroup $\HA_{m}$ or $\TS_{m}$ of
$\HAN$ or $\TSN$, we will mean a standard one, that is the one fixing $1,2, \ldots , n-m$ in the natural
permutation representation of $\SSS_{n}$. Also by a sum of simple modules we mean the sum in the 
Grothendieck group. We will also fix a preimage $\ts$ of order $3$ in $\HAN$ of a $3$-cycle.  

We begin with the following observations: 

\begin{lemma}\label{rat}
{\sl {\rm (i)} Every element of $\HAN$ is rational in $\TSN$.

{\rm (ii)} Assume $\varphi \in \IBRL(\TSN)$ and $\varphi|_{\HAN}$ is reducible. Then $\varphi$ is 
rational-valued.

{\rm (iii)} Assume $\varphi \in \IBRL(\HAN)$ and $\varphi$ extends to $\TSN$. Then $\varphi$ is 
rational-valued.}
\end{lemma}

\begin{proof}
(i) Let $\pi~:~\TSN \to \SN$ be the natural projection and $g \in \HAN$. Clearly, $g$ is rational in 
$\TSN$ if $\pi^{-1}(\pi(g)^{\SN})$ is a single $\TSN$-conjugacy class. Otherwise by 
\cite[Thm. 3.8]{HH} $\pi(g)$ is a product of disjoint cycles of odd lengths. In particular, 
$\pi(g)$ has odd order $k$. It follows that $g^{k} = z^{i}$ and $(zg)^{k} = z^{1+i}$ for some 
$i \in \ZZ$ and $Z(\AN) = \la z \ra$. Replacing $g$ by $zg$ if necessary, we may assume that $i = 0$ 
and so $|g| = k$, $|zg| = 2k$. In this case, $\pi^{-1}(\pi(g)^{\SN}) = g^{\TSN} \cup (zg)^{\TSN}$. Hence, all 
generators of $\la g \ra$ belong to $g^{\TSN}$ and so $g$ is rational, in which case $zg$ is also 
rational.     

(ii) By assumptions, $\varphi = \Ind^{\TSN}_{\HAN}(\psi)$ for some $\psi \in \IBRL(\HAN)$, whence 
$\varphi = 0$ on $\TSN \setminus \HAN$. On the other hand, $\varphi|_{\HAN}$ is rational by (i).

(iii) follows from (i).
\end{proof}

Among all the irreducible representations of $G$, the {\it basic spin} and {\it second basic spin} 
representations, cf. \cite{Wa} and \cite{KT}, will require special attention. For a fixed $\ell$,
define $\kn$ to be $1$ if $0 < \ell|n$ and $0$ otherwise. Then the ($\ell$-modular) basic spin modules
of $\HAN$, resp. of $\TSN$, have dimension $\DA_{n} := 2^{\lfloor (n-2-\kn)/2 \rfloor}$, resp.
$\DB_{n} := 2^{\lfloor (n-1-\kn)/2 \rfloor}$. The second basic 
spin modules of $\TSN$ have dimension at least $2^{(n-3)/2}(n-4)$ unless $2|n$ and $\ell|(n-1)$
in which case they have dimension $2^{(n-4)/2}(n-4)$. Let $\DCN$ denote the
smallest one among the dimensions of second basic spin representations of $\TS_{n}$ and $\TS_{n-1}$.
Then 
\begin{equation}\label{dadb}
  \DCN \geq \max\{(n-5) \cdot \DA_{n-2}, (n-5)/2 \cdot \DB_{n-2}\}.
\end{equation}

Basic spin modules are distinguished by the following property:

\begin{lemma}\label{basic1} {\rm \cite[Thm. 8.1]{Wa}} 
{\sl Let $V$ be an irreducible $\FF$-representation of $G \in \{ \HAN,\TSN\}$ such that the action of
$\ts$ (a $3^{\mathrm {rd}}$ order preimage in $G$ of a $3$-cycle) on $V$ has a quadratic minimal 
polynomial. Then $V$ is a basic spin module.
\hfill $\Box$}
\end{lemma} 

\begin{lemma}\label{basic2}
{\sl Let $G \in \{\HAN, \TSN\}$ and let $V$ be an irreducible $\FF G$-module which is not a basic spin 
module. 

{\rm (i)} Assume $G = \HAN$ and $\soc(V|_{\HA_{n-2}})$ contains a basic spin module of $\HA_{n-2}$. Then 
$\dim(V) \geq \DCN/2 \geq \DA_{n-2} \cdot (n-5)/2$.

{\rm (ii)} Assume $G = \TSN$ and $\soc(V|_{\TS_{n-2}})$ contains a basic spin module of $\TS_{n-2}$. 
Then $\dim(V) \geq \DCN \geq \DB_{n-2} \cdot (n-5)/2$.}
\end{lemma}

\begin{proof}
(i) Let $W \in \IBRL(\TSN)$ such that $V$ is a submodule of $W|_{\HAN}$. Then we can choose 
$x \in C_{\TSN}(\HA_{n-2})$ such that $W|_{\HAN} = V$ or $W|_{\HAN} = V \oplus xV$. By assumptions,
$W|_{\HA_{n-2}}$ contains a basic spin module $U$, whence $W$ is a quotient of 
$\Ind^{\TSN}_{\TS_{n-2}}(\Ind^{\TS_{n-2}}_{\HA_{n-2}}(U))$. Since 
$\Ind^{\TS_{n-2}}_{\HA_{n-2}}(U)$ is a sum of basic spin modules of $\TS_{n-2}$, there is a basic
spin module $Y_{n-2}$ of $\TS_{n-2}$ such that $W$ is a quotient of 
$\Ind^{\TSN}_{\TS_{n-2}}(Y_{n-2})$. Next, by \cite{Wa} $\Ind^{\TS_{n-1}}_{\TS_{n-2}}(Y_{n-2})$ is a sum 
of basic and second basic spin modules of $\TS_{n-1}$. It follows that $W$ is a quotient of 
$\Ind^{\TSN}_{\TS_{n-1}}(X)$, where either $X$ is a basic spin module $Y_{n-1}$ of $\TS_{n-1}$, or 
a second basic spin module $T_{n-1}$ of $\TS_{n-1}$. In the former case, $W$ must be a second basic 
spin module of $\TSN$ (as $V$ is not basic), whence $\dim(W) \geq \DCN$. In the latter case, 
$\soc(W|_{\TS_{n-1}})$ contains a second basic spin module of $\TS_{n-1}$, hence $\dim(W) \geq \DCN$ 
again. So we are done by (\ref{dadb}).

(ii) can be proved similarly.  
\end{proof}

\begin{corol}\label{basic3}
{\sl Let $G \in \{\HAN,\TSN\}$ and let $V$ be an irreducible $\FF G$-module which is 
not a basic spin module. Assume $\soc(V|_{\HA_{n-2}})$ contains a basic spin module of $\HA_{n-2}$.
Then $\SK(V)$ is reducible for any $k \geq 3$ if $n \geq 20$, and for any $k \geq 4$ if $n \geq 13$.}
\end{corol}

\begin{proof}
Assume the contrary. Define $H := \HA_{n-2}$ if $G = \HAN$, and $H := \TS_{n-2}$ if $G = \TSN$. By 
assumptions, $\soc(V|_{H})$ contains a basic spin module $U$ of $H$, whence $\soc(\SK(V)|_{H})$ 
contains $\SK(U)$. By Frobenius' reciprocity, $\dim(\SK(V)) \leq \dim(\SK(U)) \cdot (G:H)$. Since 
$(G:H) = n(n-1)$ and $\dim(V) \geq e(n-5)/2$ for $e := \dim(U)$ by Lemma \ref{basic2}, this implies 
$n(n-1) \geq \begin{pmatrix}e(n-5)/2+k-1\\k \end{pmatrix}\left/ \right. 
 \begin{pmatrix}e+k-1\\k \end{pmatrix}$. The last inequality cannot hold if $n \geq 20$ and $k \geq 3$, 
or if $n \geq 13$ and $k \geq 4$, since $e \geq \DA_{n-2} = 2^{\lfloor (n-5)/2 \rfloor}$.
\end{proof}

\begin{propo}\label{nbasic}
{\sl Let $G \in \{\HAN,\TSN\}$ and let $V$ be an irreducible $\FF G$-module of dimension $d > 1$. 
If $m$ is odd or if $(G,\ell) = (\HA_{14},7)$, assume in addition that $V$ is not a basic spin module. 
Then $\SM(V)$ is reducible, if $m \geq 3$ and $n \geq 23$, or if $m \geq 4$ and $n \geq 14$.}
\end{propo}

\begin{proof}
1) Assume the contrary. Write $V = V_{0} \oplus V_{1} \oplus V_{2}$, where 
$\ts$ acts on $V_{j}$ as the scalar $\omega^{j}$, $\omega$ a primitive cubic root of unity in $\FF$, and
$V_{1}, V_{2} \neq 0$. By Lemma \ref{basic1}, $V_{0} \neq 0$ if $V$ is not basic. We give the proof
for the case of $m \geq 4$, the case with $m = 3$ is proceeded similarly.

Case 1: Assume $G = \TSN$ and $V|_{\HAN}$ is reducible. Then $V$ is self-dual by Lemma \ref{rat}. 
Setting $C := C_{G}(\ts)$, we see that $V_{2} \simeq V_{1}^{*}$ as $C$-modules. Clearly, 
$\Sym^{2k}(V)|_{C} \supset \SK(V_{1}) \otimes \SK(V_{2}) \supset 1_{C}$, and 
$\Sym^{2k+1}(V)|_{C} \supset V_{0} \otimes \SK(V_{1}) \otimes \SK(V_{2}) \supset V_{0}$. By Frobenius'
reciprocity, if $m = 2k \geq 4$ then $\dim(\SE(V)) \leq \dim(\SM(V)) \leq (G:C)$. If $m = 2k+1 \geq 5$ 
then $\dim(\SF(V)) \leq \dim(\SM(V)) \leq (G:C) \cdot \dim(V_{0})$. In either case we obtain
$(d+1)(d+2)(d+3)(d+4) \leq 120(G:C) = 40n(n-1)(n-2)$.

Case 2: Assume that either $G = \TSN$ and $V|_{\HAN}$ is irreducible, or $G = \HAN$ and $V$ extends to
$\TSN$. In either case, $V|_{\HAN}$ is irreducible and self-dual by Lemma \ref{rat}. So we can
set $C := C_{\HAN}(\ts)$ and repeat the above argument to get 
$(d+1)(d+2)(d+3)(d+4) \leq 120(G:C) \leq 80n(n-1)(n-2)$.

Case 3: Assume that $G = \HAN$ and $V$ does not extend to $\TSN$. In this case, we can embed 
$H := \TS_{n-2}$ in $G$ (as the inverse image in $G$ of a subgroup of index $2$ in 
$S_{n-2} \times S_{2}$ that contains $\ts$). Consider a simple submodule $U$ of smallest dimension of 
$V|_{H}$.

Assume $\dim(U) = 1$. Setting $C := H' = \HA_{n-2}$, we see that $U|_{C} = 1_{C}$, whence 
$\SM(V)|_{C} \supset \SM(1_{C}) = 1_{C}$, and so $d(d+1)(d+2)(d+3) \leq 2(G:C) = 2n(n-1)$.

Now we may assume that $\dim(U) > 1$. Again write $U = U_{0} \oplus U_{1} \oplus U_{2}$,
where $\ts$ acts on $U_{j}$ as the scalar $\omega^{j}$, and $U_{1}, U_{2} \neq 0$.  
If $U|_{H'}$ is reducible, then $U$ is self-dual by Lemma \ref{rat}, whence 
$U_{2} \simeq U_{1}^{*}$ as modules over $C := C_{H}(\ts)$; set $N := N_{H}(\la \ts \ra)$ 
in this case. If $U|_{H'}$ is irreducible, then $U|_{H'}$ is self-dual by Lemma 
\ref{rat}, whence $U_{2} \simeq U_{1}^{*}$ as modules over $C := C_{H'}(\ts)$; set 
$N := N_{H'}(\la \ts \ra)$ in this case. Now if $m = 2k \geq 4$, then as in Case 1 we see that
$\SM(V)|_{C}$ contains $1_{C}$. The proof of Lemma \ref{sub} shows that $\SM(V)|_{N}$ contains a
submodule of dimension $1$, so $\dim(\SE(V)) \leq (G:N)$. Assume $m = 2k+1 \geq 5$. By Corollary 
\ref{basic3} we may assume that $\soc(V|_{\HA_{n-2}})$ does not contain a basic spin module of 
$\HA_{n-2}$, so $U_{0} \neq 0$ by Lemma \ref{basic1}. Now $\Sym^{2k+1}(V)|_{C}$ contains a submodule 
$F \simeq U_{0}$ inside $U_{0} \otimes \SK(U_{1}) \otimes \SK(U_{2})$. Since $U_{0}$ is $N$-invariant, 
by Lemma \ref{sub} $\Sym^{2k+1}(V)|_{C}$ contains a submodule of dimension $\leq \dim(U_{0})$. It follows
that $\dim(\SF(V)) \leq (G:N) \cdot \dim(U_{0})$. In either case, we obtain
$(d+1)(d+2)(d+3)(d+4) \leq 120(G:N) \leq 20n(n-1)(n-2)(n-3)(n-4)$.    

2) We have shown that in all cases
$$(d+1)(d+2)(d+3)(d+4) \leq 20n(n-1)(n-2)(n-3)(n-4).$$ 
Assume that $V$ is faithful. Then $d \geq 2^{\lfloor (n-2-\kn)/2 \rfloor}$ by \cite{KT}, so we get a
contradiction if $n \geq 15$, or if $n = 14$ but $(G,\ell) \neq (\HA_{14},\ell)$. Now assume that $V$ 
is not faithful and that $V|_{\HAN}$ is not the heart $\DC$ of the natural permutation module. Then 
$d \geq (n^{2}-5n+2)/2$ by \cite[Lem. 6.1]{GT3} and its proof, so we again get a contradiction when 
$n \geq 14$. Thus $V|_{\HAN} \simeq \DC$. Since $\DC$ is of type $+$, $\SM(V)$ is reducible by
Lemma \ref{so-sym}.     
\end{proof}

\begin{lemma}\label{cbasic}
{\sl Let $G \in \{\HAN,\TSN\}$ with $n \geq 12$ and let $V$ be a complex basic spin $G$-module. Then 
$\SK(V)$ is reducible for all $k \geq 2$, and $\WK(V)$ is reducible for all $2 \leq k \leq d/2$.}
\end{lemma}

\begin{proof}
It suffices to prove the statement for $k = 2$. Recall, cf. \cite{T1}, that if $m$ is even then 
$\HA_{m}$ has a unique basic spin character $\alpha_{m}$ which is real-valued, whereas if $m$ is odd then
$\HA_{m}$ has two basic spin characters which are real-valued if and only if $m \equiv 1 (\mod 4)$. 
Choose $m = n$ if $2|n$ or if $n \equiv 1 (\mod 4)$, and $m = n-1$ if $n \equiv 3 (\mod 4)$, and let 
$H := \HA_{m-2}$. Let $\alpha_{m}$ be an irreducible character afforded by the $\HA_{m}$-module $V$; in 
particular, $\alpha_{m}$ is real-valued. All irreducible constituents of $\alpha_{m}|_{H}$ are of degree 
$\alpha(1)/2$ and are basic. If $m$ is even, then the uniqueness of $\alpha_{m-2}$ implies that 
$\alpha_{m}|_{H} = 2\alpha_{m-2}$. If $m$ is odd, then $m-2 \equiv 3 (\mod 4)$ and so the basic spin 
characters of $H$ are not real-valued, whence $\alpha_{m}|_{H} = \alpha_{m-2} + \overline{\alpha}_{m-2}$.
We have show that $V|_{H}$ contains $U \oplus U^{*}$ for some $H$-module $U$. 

Now assume that $X(V)$ is irreducible for some $X \in \{\SB,\WB\}$. Clearly, $X(V)|_{H}$ contains 
$U \otimes U^{*} \supset 1_{H}$ and so $\dim(X(V)) \leq (G:H)$; in particular, 
$d(d-1) \leq 4n(n-1)(n-2)$. The last inequality cannot hold for $n \geq 18$ as 
$d \geq 2^{\lfloor n/2 \rfloor -1}$. If $n \in \{12,13,14,16,17\}$ then $(G:H) \leq 2n(n-1)$
as $m = n$, again yielding a contradiction. If $n = 15$ and $G = \TSN$ then $d = 2^{7}$, leading to 
a contradiction. Finally, assume $G = \HA_{15}$, so $H = \HA_{12}$. We have shown that $X(V)|_{H}$ 
contains $1_{H}$. On the other hand, $G > K := \HA_{12} \times \ZZ_{3}$. Hence $X(V)|_{K}$ contains 
a $1$-dimensional submodule and so $\dim(X(V)) \leq (G:K)$, yielding a contradiction.  
\end{proof}

Notice that $\WB(V)$ is irreducible for a complex spin module $V$ of $\HA_{11}$.

\begin{propo}\label{mbasic}
{\sl Let $G \in \{\HAN,\TSN\}$ and let $V$ be an $\ell$-modular basic spin $G$-module. 
Then $\SK(V)$ is reducible if $k \geq 2$ and $n \geq 16$, or if $k \geq 4$ and $n \geq 14$.}
\end{propo}

\begin{proof}
Assume the contrary. By Lemma \ref{cbasic}, $V$ cannot lift to a complex module. It follows by 
\cite{KT} that $\ell|n$, and either $G = \TSN$ and $n$ is odd, or $G = \HAN$ and $n$ is even.

\smallskip
1) Here we assume that $G = \TSN$ and $n$ is odd; in particular, $d = 2^{(n-3)/2}$. Hence 
$V|_{\HAN}$ is irreducible, self-dual by Lemma \ref{rat}, and lifts to a complex module $W$.
Clearly, $\SK(V)|_{\HAN}$ is a sum of at most two irreducible constituents, all of the same degree.
The same must be true also for $\SK(W)$. 

\smallskip
Case 1: Assume $n \equiv 3 (\mod 8)$. By \cite[p. 106]{T1}, $V|_{\HAN}$ is of type $+$, whence $\SK(V)$ is 
reducible if $n \geq 11$ by Lemma \ref{so-sym}.

\smallskip
Case 2: Assume $n \equiv 1 (\mod 8)$ and $n \geq 17$. Then $W$ is of type $+$, cf. \cite{T1}, and so is $
V|_{\HAN}$, so we can argue as in Case 1.  

\smallskip
Case 3: Assume $n \equiv 7(\mod 8)$ and $n \geq 15$. Consider a subgroup $H = \HA_{n-5}$. By \cite{T1}, 
any simple submodule $U$ of $W|_{H}$ is basic spin, real, and of dimension $d/4$. Thus $\SK(W)|_{H}$ 
contains $U$ if $k \geq 3$ is odd, resp. $1_{H}$ if $k$ is even. Consider $K = H * \HA_{5}$ inside $\HAN$.
Since $K/H = A_{5}$ and $\ml(\HA_{5}) = 6$, by Lemma \ref{sub} $\SK(W)|_{K}$ has a submodule $T$ of 
dimension $\leq 6\dim(U) = 3d/2$, resp. $\leq 6$. Thus $\SK(W)$ has a subquotient of dimension 
at most $\dim(T)(\HAN:K) = \dim(T)(G:K)/2$. Since $\SK(V)$ is irreducible, $\dim(\SK(V))$ is at most 
$\dim(T)(G:K)$, which is $3d/2 \cdot (G:K)$, resp. $6(G:K)$. In particular,
$2^{n-3} < 2n(n-1)(n-2)(n-3)(n-4)/5$, a contradiction if $n \geq 31$. The upper bound on $\dim(\SK(V))$ 
also gives a contradiction if $n = 15, 23$ and $k \geq 4$. Assume $n = 23$ and $k = 2,3$. Then 
$41$ divides $\dim(\SK(W))$ but not $|\HAN|$, hence $\SK(W)$ cannot be a sum of $1$ or $2$ irreducible
constituents of the same degree.  

\smallskip
Case 4: Assume $n \equiv 5(\mod 8)$ and $n \geq 21$. Consider a subgroup $H = \HA_{n-3}$. By \cite{T1}, 
any simple submodule $U$ of $W|_{H}$ is basic spin, real, and of dimension $d/2$. Thus $\SK(W)|_{H}$ 
contains $U$, resp. $1_{H}$, if $k \geq 3$ is odd, resp. if $k$ is even. Consider 
$K = H \times \ZZ_{3}$ inside $\HAN$. By Lemma \ref{sub} $\SK(V)|_{K}$ has a submodule of dimension 
$\leq \dim(U) = d/2$, resp. $\leq 1$. Thus $\dim(\SK(V))$ is at most $d/2 \cdot (G:K)$, resp. $(G:K)$. 
In particular, $2^{n-3} < 2n(n-1)(n-2)$, a contradiction as $n \geq 21$.

\smallskip
2) Here we assume that $G = \HAN$ and $n$ is even; in particular, $d = 2^{(n-4)/2}$. Hence 
$V|_{\HA_{n-1}}$ is irreducible and lifts to a complex module $W$.

\smallskip
Case 5: Assume $n \equiv 0 (\mod 8)$. Since $2 < \ell|n$, we may assume that $n \geq 24$. 
First we assume that $\ell \geq 5$; in particular $n \geq 40$. Consider a subgroup $H = \HA_{n-6}$. By 
\cite{T1}, any irreducible constituent of $W|_{H}$ is basic spin, real, irreducible modulo $\ell$ (as 
$(\ell,n-6) = 1$) and of dimension $d/4$. It follows that $V|_{H}$ contains a simple submodule 
$U$ of type $+$ and dimension $d/4$. Hence $\SK(V)|_{H}$ contains $U$ if $k \geq 3$ is odd, resp. 
$1_{H}$ if $k$ is even. Consider $K = H * \HA_{6}$ inside $G$. Since $K/H = A_{6}$ and $\ml(6A_{6}) = 15$, 
by Lemma \ref{sub} $\SK(V)|_{K}$ has a submodule of dimension $\leq 15\dim(U) = 15d/4$, resp. $\leq 15$. 
Thus $\dim(\SK(V))$ is at most $15d/4 \cdot (G:K)$, resp. $15(G:K)$. In particular,
$2^{n-4} < n(n-1)(n-2)(n-3)(n-4)(n-5)/12$, a contradiction as $n \geq 40$. 

Now we assume that $\ell = 3$ and consider a subgroup $H = \HA_{n-5}$. Notice that $\HA_{n-1}$ contains 
an overgroup $H_{1} \simeq \TS_{n-5}$ of $H$. If $\Phi$ denotes the representation of $G$ on $V$, 
set $H_{2} := \la \Phi(x), \sqrt{-1}\Phi(y) \mid x \in H, y \in H_{1} \setminus H \ra$. Then 
$H_{2}$ is a subgroup of $GL(V)$ that is isomorphic to $\HS_{n-5}$. In the same way we can make 
$H_{2}$ act on $W$. By \cite{T1}, any irreducible constituent of $W|_{H_{2}}$ is 
basic spin, real, irreducible modulo $\ell$ (as $(\ell,n-5) = 1$) and of dimension $d/2$. It 
follows that $V|_{H_{2}}$ contains a simple submodule $U_{2}$ of type $+$ and dimension $d/2$. Hence, if 
$k \geq 3$ is odd then $\SK(V)|_{H_{2}}$ contains $U_{2}$, and so $\SK(V)|_{H}$ contains a basic spin 
module $U$ of $H$ of dimension $d/4$. If $k$ is even, then $\SK(V)|_{H_{2}}$ contains $1_{H_{2}}$, and so
$\SK(V)|_{H}$ contains $1_{H}$. Consider $K = H * \HA_{5}$ inside $G$. Since $K/H = A_{5}$ and 
$\ml(\HA_{5}) = 10$, by Lemma \ref{sub} $\SK(V)|_{K}$ has a submodule of dimension 
$\leq 6\dim(U) = 3d/2$ if $k$ is odd, resp. $\leq 6$ if $k$ is even. Thus $\dim(\SK(V))$ is at most 
$3d/2 \cdot (G:K)$, resp. $6(G:K)$. In particular, $2^{n-4} < n(n-1)(n-2)(n-3)(n-4)/5$, a contradiction as 
$n \geq 24$. 

\smallskip
Case 6: Assume $n \equiv 6 (\mod 8)$ and $n \geq 14$. Consider a subgroup $H = \HA_{n-4}$. By 
\cite{T1}, any irreducible constituent of $W|_{H}$ is basic spin, real, irreducible modulo $\ell$ (as 
$(\ell,n-4) = 1$) and of dimension $d/2$. It follows that $V|_{H}$ contains a simple submodule 
$U$ of type $+$ and dimension $d/2$. Hence $\SK(V)|_{H}$ contains $U$ if $k \geq 3$ is odd, resp. 
$1_{H}$ if $k$ is even. Consider $K = H * \HA_{3}$ inside $G$. By Lemma \ref{sub} $\SK(V)|_{K}$ has a 
submodule of dimension $\leq \dim(U) = d/2$, resp. $\leq 1$. Thus $\dim(\SK(V))$ is at most 
$d/2 \cdot (G:K)$, resp. $(G:K)$. In particular, $2^{n-4} < n(n-1)(n-2)(n-3)$, a contradiction if 
$n \geq 22$. If $n = 14$ and $k \geq 4$, then the upper bound on $\dim(\SK(V))$ yields 
$2^{2n-8} < 20n(n-1)(n-2)(n-3)$, again a contradiction. 

\smallskip
Case 7: Assume $n \equiv 2,4 (\mod 8)$ and $n \geq 18$. Consider a subgroup $H = \HA_{n-2}$. By 
\cite{T1}, $W|_{H}$ is basic spin, real, irreducible modulo $\ell$ (as $(\ell,n-2) = 1$). It follows that 
$V|_{H}$ is of type $+$. Hence $\SK(V)|_{H}$ contains $V|_{H}$ if $k \geq 3$ is odd, resp. 
$1_{H}$ if $k$ is even. Thus $\dim(\SK(V))$ is at most $d(G:H)$, resp. $(G:H)$. In particular,
$2^{n-4} < 6n(n-1)$, a contradiction as $n \geq 18$. 
\end{proof}

The proof of Theorem \ref{alt} is now completed by the following lemma:

\begin{lemma}\label{tiny}
{\sl Assume $S = \soc(G/Z(G)) = \AN$ with $5 \leq n \leq 13$, $G < GL(V)$, $\ell \neq 2,3$, and $d > 4$. 

{\rm (i)} If $5 \leq n \leq 10$ then $\SK(V)$ is reducible for $k \geq 4$.

{\rm (ii)} Assume $11 \leq n \leq 13$. Then $\SK(V)$ is reducible for $k \geq 4$. Furthermore, either
$\WK(V)$ is reducible for $k \geq 4$, or $G^{(\infty)} = \AN$ and $V|_{\AN} \simeq \DC$, the heart of 
the natural permutation module.}
\end{lemma}
  
\begin{proof}
(i) follows by inspecting \cite{Atlas} and \cite{JLPW}. Assume (ii) is false. We need to look at the 
modules $V$ with $\dim(\WE(V)) \leq \ml(\TSN)$. We give the details of the computation for $n = 13$. Here
$\ml(\TSN) = 41600$, so $d \leq 33$. It follows that $d = 32$ and either $V$ lifts to a complex basic spin
module, or $(G,\ell) = (\TS_{13},13)$. In the former case we are done by Lemma \ref{cbasic}. The latter
case can be checked directly using \cite{Atlas} and \cite{JLPW}.    
\end{proof}

\section{Sporadic groups}
In this section we prove the following theorem:

\begin{theor}\label{spor}
{\sl Theorem \ref{main} holds true in the case $G$ is finite and $S := \soc(G/Z(G))$ is a sporadic 
simple group.}
\end{theor}

\begin{proof}
Obviously, we need to consider only the cases where 
\begin{equation}\label{bound1}
  \dim(\WD(V)) < \ml(G)~.
\end{equation}
For brevity, we take the convention that the condition $k \geq k_{0}$ for $\WK(V)$ will actually mean
that $k_{0} \leq k \leq d-k_{0}$. In fact, we will also work with $\WK(V)$ and we assume that $\ell > 3$.
The detailed results are listed in Table I (below), where in the third column we list the values 
$(k_{1},d_{1})$ such that $k_{1}$ is the (known) highest possible $k \geq 3$ for which $\SK(V)$ is 
irreducible (over some extension $G$ of $L$ and for some $G$-module of dimension $d_{1}$ in some 
characteristic $\ell$), and the fifth column we list the values $(k_{2},d_{2})$ such that $k_{2}$ is the 
(known) highest possible $k \geq 3$ for which $\WK(V)$ is irreducible (over some extension $G$ of $L$ and 
for some $G$-module of dimension $d_{2}$ in some characteristic $\ell$). The cases marked by$\diam$
are the ones where we only look at $\SK(V)$ and $\WK(V)$ with $k \geq 4$. The cases marked by$\heart$
are the ones where we only look at $\SK(V)$ with $k \geq 4$ and $\WK(V)$ with $k \geq 5$. The cases 
marked by$\spade$ are the ones where we only look at $\SK(V)$ with $k \geq 4$ and $\WK(V)$ with $k \geq 6$.

\begin{figure}[ht]
\centerline
{{\sc Table} I. Groups $G$ with $S := \soc(G/Z(G))$ being a sporadic finite simple group.}
\vspace{0.3cm}
\begin{tabular}{|c||c|c|c|c|} \hline
   $S$ & $\SK$ reducible when & Irreducible $\SK$ & $\WK$ reducible when & Irreducible $\WK$ 
    \\ \skipa \hline \hline
   $M_{11}$ & $k \geq 3$ & & $3 \leq k \leq d-3$ & \\ \hline
   $M_{12}$ & $k \geq 4$ & $k=3,~d=10$ & $4 \leq k \leq d-4$ & $k=3,~d=10,12$\\ \hline
   $J_{1}$ & $k \geq 3$ & & $3 \leq k \leq d-3$ & \\ \hline
   $M_{22}$ & $k \geq 3$ & & $6 \leq k \leq d-6$ & $k=5,~d=10$\\ \hline
   $J_{2}$ & $k \geq 6$ & $k=5,~d=6$ & $3 \leq k \leq d-3$ & \\ \hline
   $M_{23}$ & $k \geq 3$ & & $3 \leq k \leq d-3$ & \\ \hline
   $HS$ & $k \geq 3$ & & $4 \leq k \leq d-4$ & $k=3,~d=22$\\ \hline
   $J_{3}$ & $k \geq 4$ & $k=3,~d=18$ & $4 \leq k \leq d-4$ & $k=3,~d=18$\\ \hline
   $M_{24}$ & $k \geq 3$ & & $4 \leq k \leq d-4$ & $k=3,~d=23$\\ \hline
   $McL$ & $k \geq 3$ & & $4 \leq k \leq d-4$ & $k=3,~d=22$\\ \hline
   $He$ & $k \geq 3$ & & $3 \leq k \leq d-3$ & \\ \hline \hline
   $Ru\diam$ & $k \geq 4$ & & $6 \leq k \leq d-6$ & $k=5,~d=28$\\ \hline
   $Suz\diam$ & $k \geq 6$ & $k=5,~d=12$ & $4 \leq k \leq d-4$ & $k=3,~d=12$\\ \hline
   $O'N\diam$ & $k \geq 4$ & & $4 \leq k \leq d-4$ & \\ \hline
   $Co_{3}\diam$ & $k \geq 4$ & & $4 \leq k \leq d-4$ & $k=3,~d=23$ \\ \hline
   $Co_{2}\diam$ & $k \geq 4$ & & $4 \leq k \leq d-4$ & $k=3,~d=23$ \\ \hline
   $Fi_{22}\diam$ & $k \geq 4$ & & $4 \leq k \leq d-4$ &  \\ \hline
   $HN\diam$ & $k \geq 4$ & & $4 \leq k \leq d-4$ &  \\ \hline 
   $Fi_{23}\diam$ & $k \geq 4$ & & $4 \leq k \leq d-4$ &  \\ \hline    
   $J_{4}\diam$ & $k \geq 4$ & & $4 \leq k \leq d-4$ & $k=3,~d=1333$  \\ \hline \hline
   $Ly\heart$ & $k \geq 4$ & & $5 \leq k \leq d-5$ &  \\ \hline
   $Th\heart$ & $k \geq 4$ & & $5 \leq k \leq d-5$ &  \\ \hline
   $Co_{1}\heart$ & $k \geq 4$ & & $5 \leq k \leq d-5$ & $k=3,~d=24$ \\ \hline
   $Fi'_{24}\heart$ & $k \geq 4$ & & $5 \leq k \leq d-5$ &  \\ \hline \hline
   $B\spade$ & $k \geq 4$ & & $6 \leq k \leq d-6$ &  \\ \hline
   $M\spade$ & $\begin{array}{l}k \geq 4,~\ell \neq 5,7\\k \geq 5,~\ell = 5\\k \geq 6,~\ell = 7\end{array}$ 
             & & $6 \leq k \leq d-6$ &  \\ \hline
\end{tabular}
\end{figure}

1) For the first $11$ sporadic groups, the $\ell$-modular decomposition matrix is completely known, 
cf. \cite{JLPW}. In these cases, it is straightforward to verify the above statements. 
Assume $S = M_{11}$. Then $\ml(G) = 55$ and $d \geq 10$, and so (\ref{bound1}) cannot hold. 
Assume $S = M_{12}$. Then $\ml(G) = 320$ and so (\ref{bound1}) implies $d \leq 13$. Using \cite{Atlas} 
and \cite{JLPW}, we can check that $\SK(V)$ and $\WK(V)$ are reducible for $k \geq 4$, $\SD(V)$ is 
reducible except when $d= 10$, and $\WD(V)$ is reducible except when $d = 10, 12$. The cases $S = J_{1}$, 
$M_{23}$ are similar.

Assume $S = M_{22}$. Then $\ml(G) = 1120$ and so $d \leq 19$. It follows that either $d = 10$, 
in which case $\SK(V)$ is reducible for $k \geq 3$, $\WK(V)$ is reducible for $k \geq 6$ but irreducible
(over $L \cdot 2$) when $\ell = 0$ and $k \leq 5$, or $(d,\ell) = (16,7)$, in which case 
$\SK(V)$ is reducible for $k \geq 3$, $\WK(V)$ is reducible for $k \geq 4$, but irreducible for $k = 3$. 
The cases $S = J_{2}$, $HS$, $J_{3}$, $M_{24}$, $McL$, $He$ are similar.

2) For the remaining $15$ sporadic simple groups which are not included in \cite{JLPW}, we will 
work with the stronger bound 
\begin{equation}\label{bound2}
  \dim(\WE(V)) < \ml(G)~.
\end{equation}
The lower bound for $d$ is listed in \cite{Jan}. In some cases, the modules $V$ satisfying (\ref{bound1}) 
are determined using \cite{HM}. We also use the decomposition matrices available online at 
\cite{ModAt}. Assume $S = Ru$. Then $\ml(G) = 250,560$ and so (\ref{bound2}) implies 
$d \leq 52$, whence $d = 28$ by \cite{HM}. It follows that $V$ lifts to a complex module $\VC$. Using 
\cite{Atlas} and \cite{JLPW}, we can check that $\SK(\VC)$ is reducible for $k \geq 4$, $\WK(\VC)$ is 
reducible for $k \geq 6$ and irreducible for $k \leq 5$. The cases $S = Suz$, $O'N$, $CO_{3}$,
$Co_{2}$, $Fi_{22}$, $HN$, $Fi_{23}$, $J_{4}$ are similar. 

3) For the $6$ largest sporadic simple groups, there is only very scarce information about the 
irreducible $\ell$-modular Brauer characters of them and their covers. Even in the case the Sylow 
$\ell$-subgroups of $L$ are cyclic, the shape of the Brauer tree is not known in some cases, cf. \cite{HL}.

Assume $S = Ly$. Then (\ref{bound2}) implies $d \leq 203$. Hence $d = 111$ and $V$ is of type $+$ by 
\cite{HM}, and so $\SK(V)$ is reducible for $k \geq 2$ and $\WK(V)$ is reducible for $k \geq 5$.

Assume $S = Th$. Then $d \geq 248$, whence $\dim(\WF(V)) > \ml(G)$ and so $\SK(V)$ and $\WK(V)$ are 
reducible for $k \geq 5$. Claim that $\SE(V)$ is reducible as well. Assume the contrary. Consider subgroups
$H = \tb D_{4}(2)$ and $K = H \cdot 3$ of $G$. Then all irreducible $\ell$-modular Brauer characters of 
$H$ are of type $+$. The proof of Lemma \ref{index}(i) implies that $\SE(V)|_{H}$ contains $1_{H}$, whence
$\SE(V)|_{K}$ contains a $1$-dimensional submodule by Lemma \ref{sub}(i). By Frobenius' reciprocity, 
$\dim(\SE(V)) \leq (G:Z(G)K)$ and so $d \leq 242$, a contradiction.  

Assume $S = Co_{1}$. Then (\ref{bound2}) implies $d \leq 398$. In fact if $d \geq 170$ then 
$\dim(\WF(V)) > \ml(G)$ and so $\SK(V)$ and $\WK(V)$ are reducible for $k \geq 5$. If $d < 170$ then 
$d = 24$ and $V$ lifts to a complex module $\VC$, cf. \cite{HM}, with reducible $\WF(\VC)$ and 
$\SB(\VC)$. Claim that $\SE(V)$ is reducible if $170 \leq d \leq 398$. Assume the contrary. Consider the 
subgroup $H = Co_{2}$ of $G$. Using \cite{ModAt}, one can check that all irreducible $\ell$-modular 
Brauer characters of $H$ of degree $\leq 398$ are of type $+$. Now Lemma \ref{index}(i) implies that 
$d \leq 39$, a contradiction.

Assume $S = Fi'_{24}$. Then $d \geq 781$, whence $\dim(\WF(V)) > \ml(G)$ and so $\SK(V)$ and $\WK(V)$ are 
reducible for $k \geq 5$. In fact if $d \geq 2726$ then $\dim(\SE(V)) > \ml(G)$ and so $\SE(V)$ is
reducible. Claim that $\SE(V)$ is also reducible if $2 \leq d \leq 2725$. Assume the contrary. Consider the 
subgroup $H = Fi_{23}$ of $G$. Using \cite{ModAt}, one can check that all irreducible $\ell$-modular 
Brauer characters of $H$ of degree $\leq 2725$ are of type $+$ if $\ell \neq 17$. In the case 
$\ell = 17$, $H$ has exactly $17$ complex irreducible characters of positive $17$-defect and they all
belong to the principal $17$-block. The shape of the Brauer tree of this block is determined in 
\cite{HL}. Using this information we can show that the irreducible Brauer characters in the
block are either trivial, or equal to $\hat{\chi}-1_{H}$, or of degree $\geq 3588$, where $\chi$ is
the unique character of degree $783$ in $\Irr(H)$ and $\hat{\chi}$ denotes the restriction of $\chi$ to 
$\ell'$-elements. Now it is easy to verify our claim for $\ell = 17$. Hence Lemma \ref{index}(i) implies 
that $d < 83$, a contradiction.

4) Assume $S = B$. Then $d \geq 4370$, whence $\dim(\wedge^{6}(V)) > \ml(G)$ and so $\SK(V)$ and $\WK(V)$ are 
reducible for $k \geq 6$. In fact if $d \geq 29,130$ then $\dim(\SE(V)) > \ml(G)$ and so $\SE(V)$ is
reducible. Claim that $\SE(V)$ and $\SF(V)$ are also reducible if $2 \leq d \leq 29,129$ (in particular 
$L = S$) and $\ell \neq 5,7$. Assume the contrary. Consider the subgroup 
$H = C_{G}(t) \simeq 2 \cdot \ta E_{6}(2)$ for some involution $t \in G$. Then $V_{0} := \Ker(t-1)$ is 
actually a nonzero $\ta E_{6}(2)$-module. Using \cite{Atlas} and \cite{ModAt}, one can check that all 
irreducible $\ell$-modular Brauer characters of $\ta E_{6}(2)$ of degree $\leq 29,129$ are of type $+$ if 
$\ell \geq 11$. This is also true for $\ell = 7$, cf. \cite{Mu1} and for $\ell = 5$, cf. \cite{Mu2}. 
Now any simple $H$-submodule $U$ of $V_{0}$ is of type $+$. Hence Lemma \ref{index}(i) and its proof imply 
that $d < 1600$, a contradiction.

5) From now on we assume $S = M$. Since $\Mult(M) = 1$, $L = S$, $G = Z(G) \times S$. Without loss
we may assume $G = M$. Now $d \geq 196,882$ and $\ml(G) < (2.6) \cdot 10^{26}$, whence 
$\dim(\wedge^{6}(V)) > \ml(G)$ and so $\SK(V)$ and $\WK(V)$ are reducible for $k \geq 6$. In fact if 
$d \geq D := (8.9) \cdot 10^{6}$ then $\dim(\SE(V)) > \ml(G)$ and so $\SE(V)$ is reducible. We will show that
if $2 \leq d < D$ and $\ell \neq 5,7$ then $\SK(V)$ is also reducible for $k = 4,5$. Assume the contrary. 
According to \cite{Atlas} and \cite{ModAt}, any such a $V$ is of type $+$ if $\ell = 0$, $17$, $19$, $23$, 
or $31$. So $\ell \in \{11,13,29,41,47,59,71\}$. 

Consider the subgroup $H = 2 \cdot B = C_{G}(t)$, where $t$ is an involution of class $2A$ of $G$. Since
$\ell \neq 2$, we have $V = V_{0} \oplus V_{1}$, where $V_{j} = \Ker(t-(-1)^{j}) \neq 0$. According to 
\cite{Jan}, $\dim(V_{1}) \geq 96256$. Claim that any Brauer character $\psi$ of degree $<D$, in $\IBRL(H)$ if 
$\ell \neq 47$ and in $\IBRL(B)$ if $\ell = 47$, is of type $+$. Using \cite{Atlas} and \cite{ModAt}, one 
readily checks the claim for $\ell \in \{11,29,41,59,71\}$. Also, one needs to consider only characters 
$\psi$ belonging to $\ell$-blocks of positive defect. Notice that $\Irr(H)$ have a unique character of 
degree $1$, resp. $4371$, $96,255$, $1,139,374$, $96,256$, all of type $+$, and we denote them by 
$\chi_{1}$, resp. $\chi_{2}$, $\chi_{3}$, $\chi_{4}$, $\chi_{5}$. Assume $\ell = 13$. Then $H$ has $7$ 
blocks of positive defect (all with cyclic defect group), and the shapes of their Brauer trees have been 
determined in \cite{HL}. Now we can find the degrees of Brauer characters in these blocks. It follows that 
such a $\psi$ is $\hat{\chi}_{i}$ with $1 \leq i \leq 5$, and so the claim follows. Assume $\ell = 47$. Then
the only block of positive defect of $B$ is the principal block (and it has cyclic defect group). There are 
five possible shapes for the Brauer tree of this block, as shown in \cite{HL}. Now we can find the degrees 
of Brauer characters in all of these five cases. It follows that such a $\psi$ is either $\hat{\chi_{i}}$ 
with $i = 1,3,4$ or $\hat{\chi}_{2}-1_{B}$, and so the claim follows again.
  
6) First we handle the case $\ell \neq 47$. The above claim implies that, for $j = 0,1$, $V_{j}|_{H}$ 
contains a simple submodule $U_{j}$ of type $+$. Now if $k = 4$, then 
$$\SE(V)|_{H} \supset \SE(U_{0}) \oplus \SE(U_{1}) \supset 2 \cdot 1_{H}.$$
Thus 
$$2 \leq \dim\Hom_{H}(1_{H},\SE(V)|_{H}) = \dim\Hom_{G}(\Ind^{G}_{H}(1_{H}),\SE(V))$$ 
and so $\dim(\SE(V)) \leq (G:H)/2$ by Frobenius' reciprocity. It follows that $d \leq 186,120$, a 
contradiction. Assume $k = 5$. Then  
$$\SF(V)|_{H} \supset \SF(U_{0}) \oplus \SF(U_{1}) \supset U_{0} \oplus U_{1}.$$
Among $U_{0}$ and $U_{1}$ we choose $U_{j}$ of smaller dimension. Then Frobenius' reciprocity implies
$\dim(\SF(V)) \leq (G:H)d/2$, whence $d \leq 278,315$ and $\dim(V_{0}) \leq d-96,256 = 182,059$. Thus 
all composition factors of $V_{0}|_{H}$ have dimension $1$, $4371$, or $96,255$, and all composition factors 
of $V_{1}|_{H}$ have dimension $96,256$. Claim that $V_{0}|_{H}$ has a simple submodule or a simple quotient,
call it $U$, of dimension $\leq 4371$. (If not, all simple submodules and simple quotients of it have 
dimension $96,255$. But $\dim(V_{0}) < 2 \cdot 96,255$, so in fact $\dim(V_{0}) = 96,255$. Now we have 
$$100,627 = 196,882 - 96,255 \leq \dim(V_{1}) \leq 182,060 = 278,315 - 96,255,$$
which is impossible as all composition factors of $V_{1}|_{H}$ have dimension $96,256$.) Since $V_{0}$ is a 
direct summand of $V|_{H}$, Frobenius' reciprocity implies 
$\dim(\SF(V)) \leq \dim(U) \cdot (G:H) \leq 4371 \cdot (G:H)$, and so $d \leq 139,300$, again a 
contradiction.  

7) Finally we treat the case $\ell = 47$. The claim proved in 5) implies that $V_{0}|_{H}$ 
contains a simple submodule $U_{0}$ of type $+$. Let $\varphi$, resp. $\alpha$, $\beta$, denote the Brauer 
character afforded by $V$, resp. $V_{0}|_{H}$, $V_{1}|_{H}$. According to \cite{Atlas}, $G$ has another
involution $t'$ such that $t,t',tt'$ are all in the class $2A$ of $G$ and $t'$ belongs to class $2A$ in $H$;
in fact $C_{G}(\la t,t' \ra) = 2^{2} \cdot \ta E_{6}(2)$. Observe that $\beta(t') = 0$. Indeed, 
$t|_{V_{0}} = 1_{V_{0}}$ and $t|_{V_{1}} = -1_{V_{1}}$, whence $\alpha(tt') = \alpha(t')$ and 
$\beta(tt') = -\beta(t')$. Now
$$\alpha(t') + \beta(t') = \varphi(t') = \varphi(tt') = \alpha(tt') + \beta(tt') = \alpha(t') - \beta(t'),$$
and so $\beta(t') = 0$.

By Lemma \ref{index}(i) we get $d \leq 221,336$ if $k = 4$ 
and $d \leq 330,975$ if $k = 5$. It follows that all irreducible constituents of $\alpha$ have degree $1$, 
$4371$, or $96,254$, and we label them as $\psi_{1}$, $\psi_{3}$, and $\psi_{2}$, respectively. As in 6), 
we see that $V_{0}|_{H}$ cannot have a simple submodule or a simple quotient of dimension $\leq 4371$ if 
$k = 5$, and it cannot have two distinct simple submodules if $k = 4$. Observe that $\psi_{1}$ and 
$\psi_{2}$ belong to the principal block and $\psi_{3}$ has defect $0$. Claim that all composition factors 
of $V_{0}|_{H}$ belong to the principal block. Assume the contrary. Then $k \neq 5$ as we noted. Now if 
these composition factors involve two different blocks then $V_{0}|_{H}$ has simple submodules $U_{0}$ and 
$U_{1}$ (from different blocks), a contradiction. So all composition factors of $V_{0}|_{H}$ have Brauer 
character $\psi_{3}$ (of degree $4371$) and are isomorphic to $U_{0}$. Since $\psi_{3}$ has defect $0$, 
$\Ext^{1}_{H}(U_{0},U_{0}) = 0$. Thus $V_{0}|_{H}$ is in fact a direct sum of say $a$ copies of $U_{0}$. 
As we noted, $a$ cannot be greater than $1$, so $\alpha = \psi_{3}$, in particular, $\alpha(t') = -493$.
Now
$$4371 - \dim(V_{1}) = \varphi(t) = \varphi(t') = \alpha(t') + \beta(t') = -493 + 0,$$
whence $\dim(V_{1}) = 4864$ and $d = 9235$, a contradiction.  

Thus we may write $\alpha = x\psi_{1} + y\psi_{2}$ for some non-negative integers $x,y$. Then 
$$x + y \cdot 96,254 - \dim(V_{1}) = \varphi(t) = \varphi(t') = \alpha(t') + \beta(t') = 
  x + y \cdot 4862 + 0,$$
and so $\dim(V_{1}) = y \cdot 91,392$. Since $\dim(V_{1}) \geq 96,256$, $y \geq 2$. It follows 
that $d \geq 2 \cdot 96,255 + 96,256 = 288,766$ and so $k = 5$ (as $d \leq 221,336$ if $k = 4$). Clearly, 
$\dim(U_{0}) \leq 96,254$ and $\SF(V)|_{H} \supset \SF(U_{0}) \supset U_{0}$. By Frobenius' reciprocity,
$\dim(\SF(V)) \leq (G:H) \cdot 96,254$, whence $d \leq 258,535$, contradicting the bound
$d \geq 288,766$.

To complete the proof, we notice that $\SF(V)$ is reducible when $\ell = 5$ by Lemma \ref{step1}.
\end{proof}

\section{Proofs of Main Results}

{\bf Proof of Theorem \ref{main}.} It follows immediately from Proposition \ref{red} and Theorems 
\ref{defi}, \ref{cross}, \ref{extra}, and \ref{spor}.
\hfill $\Box$

\medskip
{\bf Proof of Theorem \ref{lowdim}.} Since there is nothing to prove in the case $d = 1$, we may assume 
that $d \geq 2$. Now we can apply Proposition \ref{red} to $G$. 

First we consider the case $\GNC \neq 1$. If $d = 2$ then $\GC$ and $\GNC$ are both of type $A_{1}$ 
whence $\GNC = SL(V)$. If $d \geq 3$, then $\ell > k$ by Lemma \ref{step1} and we can apply Theorem 
\ref{defi} to $\GNC$.

Next we consider the case where the conclusion (iii) of Proposition \ref{reduction} holds. As in the 
proof of Theorem \ref{extra} we may assume $\ell = 0$ and $k = 4$; in particular $\dim(\SE(V))$ divides 
$|G|$. Now if $d = 2$ then $P$ is a $2$-group and $G/P \leq Sp_{2}(2)$ is a $\{2,3\}$-group, but 
$\dim(\SE(V)) = 5$. If $d = 3$ then $P$ is a $3$-group and $G/P \leq Sp_{2}(3)$ is a $\{2,3\}$-group, but 
$\dim(\SE(V)) = 15$. If $d = 4$ then $P$ is a $2$-group and $G/P \leq Sp_{4}(2)$ is a $\{2,3,5\}$-group, but 
$\dim(\SE(V)) = 35$. Thus all these possibilities cannot occur here.

So we may assume that the conclusion (ii) of Proposition \ref{reduction} holds. Consider $L := G^{(\infty)}$ 
and $S := L/Z(L)$. First suppose that $\ell > 0$ and $S \in Lie(\ell)$. If $d = 2$ then 
$L = SL_{2}(q)$ by \cite[Prop. 5.4.13]{KL}. If $d \geq 3$, then $\ell > k$ and again we can apply Theorem 
\ref{defi} to $\GN$. Now we may assume that $S \notin Lie(\ell)$ if $\ell > 0$. Existing lower bounds on the
dimension of (projective) irreducible representations of $S$, cf. \cite{KL} and \cite{LS}, and the condition 
$d \leq 4$ imply that $S \in \{\AAA_{5},\AAA_{6},\AAA_{7},PSL_{2}(7),PSL_{3}(4),PSU_{4}(2)\}$. 
We will analyze these
possibilities case by case and apply the obvious upper bound $\dim(\SE(V)) \leq \ml(G)$. 

Assume $S = \AAA_{5}$. Since $\ml(G) = 6$, $d = 2$, and we arrive at (iii) by inspecting \cite{Atlas} and 
\cite{JLPW}. Assume $S = \AAA_{6}$. Since $S \notin Lie(\ell)$, $\ell \neq 2,3$. It follows that $d = 3,4$ and
$V$ lifts to a complex module $\VC$. Now it is easy to check that $\SE(\VC)$ is reducible and so is 
$\SK(V)$. A similar argument applies to the cases $S = PSL_{2}(7)$, $PSU_{4}(2)$.  
Assume $S = PSL_{3}(4)$. Since $S \notin Lie(\ell)$, $\ell \neq 2$, whence $d \geq 3$ by \cite{JLPW}. This 
in turn implies by Lemma \ref{step1} that $\ell > k \geq 4$ and so $d \geq 6$ by \cite{JLPW}. Finally,
assume $S = \AAA_{7}$. Again, $d \geq 3$ by \cite{JLPW} and so $\ell \geq 5$ by Lemma \ref{step1}. If 
$d = 4$ then $V$ lifts to a complex module $\VC$ and $\SE(\VC)$ is reducible. So $d = 3$, $\ell = 5$ and we 
arrive at (iv).     
\hfill $\Box$

\medskip
{\bf Proof of Corollary \ref{sym3}.} Clearly we need to consider only the case where $\ell > 0$ and 
$d \geq 3$. Now the statement is a consequence of Theorems \ref{main} and \ref{lowdim}
(and a direct computation for the small groups listed in Theorem \ref{main}(iii) and 
Theorem \ref{lowdim}(iv)).  
\hfill $\Box$

\section{Larsen's conjecture}
Let $\FF$ be an algebraically closed field of characteristic $0$, $V = \FF^{d}$ with 
$d > 4$, and let $\GC = GL(V)$, $GO(V)$, or $Sp(V)$. Label the fundamental weights of 
$\GC$ such that $V = L(\om_{1})$, the irreducible $\GC$-module with highest weight $\om_{1}$, and 
$L(\om_{k})$ is a subquotient of the $\GC$-module $\WK(V)$ if $d \geq 2k$. Then $L(k\om_{1})$ is a 
subquotient of the $\GC$-module $\SK(V)$. 

\medskip
{\bf Proof of Theorem \ref{sym-alt}.}
Let $G$ be as in the theorem, $\GC = GO(V)$, and assume that $G$ is irreducible on 
$L(4\om_{1}) = \SE(V)/\SC(V)$. By Corollary \ref{sym2}(ii), $G$ is irreducible on 
$L(2\om_{1}) = \SB(V)/1_{G}$. By \cite[Cor. 1.7]{T2}, one of the following holds for $S := \soc(G/Z(G))$
and $L := G^{(\infty)}$.

a) $d = 7$ and $G = G_{2}(\FF)$. Here $G$ is reducible on $\WB(V)$.

b) $d = 2^{a} \geq 8$ and $G \leq N_{\GC}(E)$ for $E = 2^{1+2a}_{+}$. By Theorem \ref{extra}, $G$ is 
reducible on $\WE(V)$.

c) $d = (5^{n}+1)/2 \geq 13$, $S = PSp_{2n}(5)$, and $V|_{S}$ is a Weil representation.
We restrict $V$ to the subgroup $C := C_{S}(t)$, where $t$ is a long-root element, and apply 
Lemma \ref{index}(ii). It follows that $G$ is reducible on $ \WE(V)$.

d) $d = (3^{2n+1}+1)/4 \geq 7$, $S = SU_{2n+1}(3)$, and $V|_{S}$ affords the Weil character 
$\zeta^{2}_{2n+1,3}$, cf. \cite{TZ2}. Then for the subgroup $H = SU_{2n}(3)$, $V|_{H}$ affords
the character $\zeta^{0}_{2n,3} + \zeta^{1}_{2n,3} + \bar{\zeta}^{1}_{2n,3}$. When $n \geq 2$ we can apply 
Lemma \ref{index}(ii) to see that $G$ is reducible on $\WE(V)$. When $n = 1$, clearly
$G$ is reducible on $L(4\om_{1})$ (of dimension $182$).

e) $d = (2^{2n}+2)/3 \geq 6$, $S = SU_{2n}(2)$, and $V|_{S}$ affords the Weil character 
$\zeta^{0}_{2n,2}$, cf. \cite{TZ2}. Then for the subgroup $H = SU_{2n-1}(2)$, $V|_{H}$ affords
the character $\zeta^{1}_{2n-1,2} + \bar{\zeta}^{1}_{2n-1,2}$. When $n \geq 4$ we can apply 
Lemma \ref{index}(ii) to see that $G$ is reducible on $\WE(V)$. When $n = 2$, clearly
$G$ is reducible on $L(4\om_{1})$ (of dimension $105$).  
      
f) $(d,L) = (7,SL_{2}(8))$, $(18,Sp_{4}(4))$, $(7,Sp_{6}(2))$, $(8,\HA_{8})$, $(8,\HA_{9})$, 
$(8,2 \cdot Sp_{6}(2))$, $(8,\Om^{+}_{8}(2))$, $(14,G_{2}(3))$, $(22,McL)$, $(23,Co_{3})$, $(23,Co_{2})$, 
$(24,2 \cdot Co_{1})$, $(52,2 \cdot F_{4}(2))$, $(78,Fi_{22})$, $(133,HN)$, $(248,Th)$. In all cases but
$(24,2 \cdot Co_{1})$, $G$ is reducible on $L(4\om_{1})$. In the case of $(24,2 \cdot Co_{1})$, $G$ is 
reducible on $\WE(V)$ (but observe that $G$ is irreducible on $L(k\om_{1})$ for $k \leq 5$ !)
\hfill $\Box$

\medskip
{\bf Proof of Corollary \ref{larsen}.}
Notice that the reductivity of $G^{\circ}$ implies that the $G$-module $V^{\otimes 4}$ is semisimple. Also, 
$L(4\om_{1})$ and $L(\om_{4})$ are composition factors of the $\GC$-module $V^{\otimes 4}$, and $L(\om_{2})$ 
is a composition factors of the $\GC$-module $V^{\otimes 2}$. Hence the statement follows from Theorem 
\ref{sym-alt} in the case $\GC = GO(V)$ (notice that here $d > 4$ by the assumptions). Assume that 
$\GC = GL(V)$ or $Sp(V)$. Then $L(4\om_{1}) = \SE(V)$, and we can apply Theorems \ref{main} and 
\ref{lowdim}. First suppose that $\GC = GL(V)$. Then notice that $Sp(V)$ is reducible on the $\GC$-submodule 
$L(\om_{4}) = \WE(V)$ if $d \geq 8$ and on $L(\om_{2}) = \WB(V)$ if $d = 4,6$; furthermore, 
$G$ is reducible on $\WE(V)$ in the cases $(d,L) = (12,2G_{2}(4))$, $(12,6Suz)$. So we arrive at (ii). 
Assume $\GC = Sp(V)$. Then $G$ is reducible on $L(\om_{4})$ (of dimension $429$) in the case $(d,L) = (12,2G_{2}(4))$, so we again arrive at (ii).
\hfill $\Box$


\begin{thebibliography}{ABCD}

\bibitem[A]{A}
  M. Aschbacher, On the maximal subgroups of the finite classical groups,
{\it Invent. Math.} {\bf 76} $(1984)$, $469 - 514$.

\bibitem[BK]{BK}
  V. Balaji and J. Koll\'ar, Holonomy groups and stable vector bundles, {\it Publ. Math. RIMS}
(to appear), (preprint: arXiv:math.AG/0601120).

\bibitem[Atlas]{Atlas}
  J. H. Conway, R. T. Curtis, S. P. Norton, R. A. Parker, and R. A. Wilson,
`{\it An ATLAS of Finite Groups}', Clarendon Press, Oxford, $1985$.

\bibitem[DM]{DM}
  F. Digne and J. Michel, `{\it Representations of Finite Groups of Lie Type}', London
Mathematical Society Student Texts $21$, Cambridge University Press, $1991$.

\bibitem[DW]{DW}
  M. P. Do Carmo and N. R. Wallach, Minimal immersions of spheres into spheres,
{\it Annals of Math.} {\bf 93} $(1971)$, $43 - 62$.

\bibitem[Dyn]{Dyn}
  E. B. Dynkin, Maximal subgroups of the classical groups, 
{\it Amer. Math. Soc. Translations}, {\bf 6} $(1957)$, $245 - 378$. 

\bibitem[FH]{FH}
  W. Fulton and J. Harris, `{\it Representation Theory}', Springer-Verlag, New
York, $1991$.

\bibitem[Fo]{Fo}
  B. Ford, Overgroups of irreducible linear groups. I, {\it J. Algebra} {\bf 181} $(1996)$, $26 - 69$;
II, {\it Trans. Amer. Math. Soc.} {\bf 351} $(1999)$, $3869 - 3913$.
   
\bibitem[Ge]{Ge}
  M. Geck, Irreducible Brauer characters of the $3$-dimensional special
unitary groups in non-describing characteristic, {\it Comm. Algebra} {\bf 18}
$(1990)$, $563 - 584$.

\bibitem[GMST]{GMST}
  R. M. Guralnick, K. Magaard, J. Saxl, and Pham Huu Tiep, Cross characteristic
representations of symplectic groups and unitary groups, {\it J. Algebra} {\bf 257} $(2002)$, 
$291 - 347$.

\bibitem[GS]{GS}
  R. M. Guralnick and J. Saxl, Generation of finite almost simple groups by conjugates, 
{\it J. Algebra} {\bf 268} $(2003)$, $519 - 571$. 

\bibitem[GT1]{GT1}
  R. M. Guralnick and Pham Huu Tiep, Cross characteristic representations of even 
characteristic symplectic groups, {\it Trans. Amer. Math. Soc.} {\bf 356} $(2004)$, $4969 - 5023$.

\bibitem[GT2]{GT2}
  R. M. Guralnick and Pham Huu Tiep, Decompositions of small tensor powers and Larsen's 
conjecture, {\it Represent. Theory} {\bf 9} $(2005)$, $138 - 208$. 

\bibitem[GT3]{GT3}
  R. M. Guralnick and Pham Huu Tiep, The non-coprime $k(GV)$ problem, {\it J. Algebra} {\bf 293} 
$(2005)$, $185 - 242$.

\bibitem[HL]{HL}
  G. Hiss and K. Lux, `{\it Brauer Trees of Sporadic Groups}', Cambridge Univ. Press, $1989$.

\bibitem[HM]{HM}
  G. Hiss and G. Malle, Corrigenda: Low-dimensional representations of quasi-simple groups,
{\it LMS J. Comput. Math.} {\bf 5} $(2002)$, $95 - 126$.

\bibitem[Hof]{Hof}
  C. Hoffman, Projective representations for some exceptional finite groups of Lie type, in: 
`{\it Modular Representation Theory of Finite Groups}', M. J. Collins, B. J. Parshall, L. L. Scott, 
eds., Walter de Gruyter, Berlin et al, $2001$, $223 - 230$.

\bibitem[HH]{HH}
  P. N. Hoffman and J. F. Humphreys, `{\it Projective Representations of the
Symmetric Group}', Clarendon Press, Oxford, 1992.

\bibitem[Jan]{Jan}
  C. Jansen, The minimal degrees of faithful representations of the
sporadic simple groups and their covering groups, {\it LMS J. Comput. Math.}
{\bf 8} $(2005)$, $122 - 144$.

\bibitem[JLPW]{JLPW}
  C. Jansen, K. Lux, R. A. Parker, and R. A. Wilson, `{\it An ATLAS of
Brauer Characters}', Oxford University Press, Oxford, $1995$.

\bibitem[Ka1]{Ka1}
  N. Katz, Larsen's alternative, moments, and the monodromy of Lefschetz pencils, 
in: `{\it Contributions to Automorphic Forms, Geometry, and Number Theory}', Johns Hopkins University
Press, pp. $521 - 560$.

\bibitem[Ka2]{Ka2}
  N. Katz, `{\it Moments, Monodromy, and Perversity: a Diophantine Perspective}', 
Annals of Math. Study, {\bf 159} Princeton Univ. Press, $(2005)$.

\bibitem[KlL]{KlL}
  P. B. Kleidman and M. W. Liebeck, `{\it The Subgroup Structure of the
Finite Classical Groups}', London Math. Soc. Lecture Note Ser. no.
$129$, Cambridge University Press, $1990$.

\bibitem[KT]{KT}
  A. S. Kleshchev and Pham Huu Tiep, On restrictions of modular spin representations 
of symmetric and alternating groups, {\it Trans. Amer. Math. Soc.} {\bf 356} $(2004)$, $1971 - 1999$.

\bibitem[KL]{KL}
  J. Koll\'ar and M. Larsen, Symmetric powers, in: `{\it Algebra, Arithmetic and Geometry - a 
Tribute to Yuri Manin}', {\it Progress in Mathematics} vols.  {\bf 269}, {\bf 270}, Birkh\"auser.

\bibitem[LS]{LS}
  V. Landazuri and G. Seitz, On the minimal degrees of projective
representations of the finite Chevalley groups, {\it J. Algebra} {\bf 32}
$(1974)$, $418 - 443$.

\bibitem[Lu]{Lu}
  F. L\"ubeck, Smallest degrees of representations of exceptional groups of 
Lie type, {\it Comm. Algebra} {\bf 29} $(2001)$, $2147 - 2169$.

\bibitem[MMT]{MMT}
  K. Magaard, G. Malle and Pham Huu Tiep, Irreducibility of tensor squares, symmetric squares,
and alternating squares, {\it Pacific J. Math.} {\bf 202} $(2002)$, $379 - 427$.

\bibitem[MT1]{MT1}
  K. Magaard and Pham Huu Tiep, Irreducible tensor products of representations
of finite quasi-simple groups of Lie type, in: `{\it Modular Representation Theory of
Finite Groups}', M. J. Collins, B. J. Parshall, L. L. Scott, eds., Walter de Gruyter,
Berlin et al, $2001$, pp. $239 - 262$.

\bibitem[MT2]{MT2}
  K. Magaard and Pham Huu Tiep, Quasisimple subgroups of classes ${\mathcal C}_{6}$ and 
${\mathcal C}_{7}$ of finite classical groups, (in preparation).

\bibitem[McN]{McN}
  G. McNinch, Semisimplicity of exterior powers of semisimple representations of groups, {\it J. Algebra} 
{\bf 225} $(2000)$, $646 - 666$.

\bibitem[ModAt]{ModAt}
  Decomposition matrices, available online at\\  
$\text{http://www.math.rwth-aachen.de/homes/MOC/decomposition/}$

\bibitem[Mu1]{Mu1}
  J. M\"uller, `{\it Zerlegungszahlen f\"ur generische Iwahori-Hecke-Algebren von exzeptionellem Typ}', 
Dissertation, RWTH Aachen, 1995.

\bibitem[Mu2]{Mu2}
  J. M\"uller, (private communication).

\bibitem[NT]{NT}
  G. Navarro and Pham Huu Tiep, Rational irreducible characters and
rational conjugacy classes in finite groups, {\it Trans. Amer. Math. Soc.}  
{\bf 360} $(2008)$, $2443 - 2465$.

\bibitem[Noz]{Noz}
  S. Nozawa, Characters of the finite general unitary group $U(5,q^{2})$,
{\it J. Fac. Sci. Univ. Tokyo Sect. IA} {\bf 23} $(1976)$, $23 - 74$.

\bibitem[Ra1]{Ra1}
  C. S. Rajan, Unique decomposition of tensor products of irreducible representations of simple algebraic 
groups, {\it Ann. of Math.} {\bf 160} $(2004)$, $683 - 704$. 

\bibitem[Ra2]{Ra2}
  C. S. Rajan, Recovering modular forms and representations from tensor and symmetric powers, in: 
`{\it Algebra and Number Theory}', Hindustan Book Agency, Delhi, $(2005)$, pp. $281 - 298$.

\bibitem[Se1]{Se1}
  G. M. Seitz, `{\it The Maximal Subgroups of Classical Algebraic Groups}, 
{\it Mem. Amer. Math. Soc.}, {\bf 67} $(1987)$, no. 365. 

\bibitem[Se2]{Se2}
  G. M. Seitz, Cross-characteristic embeddings of finite groups of Lie type,
{\it Proc. London Math. Soc.} {\bf 60} $(1990)$, $166 - 200$.

\bibitem[S]{S}
  J.-P. Serre, Sur la semi-simplicit\'e des produits tensoriels de repr\'esentations de groupes,
{\it Invent. Math.} {\bf 116} $(1994)$, $513 - 530$. 

\bibitem[ST]{ST}
  P. Sin and Pham Huu Tiep, Rank $3$ permutation modules for finite classical groups, 
{\it J. Algebra} {\bf 291} $(2005)$, $551 - 606$. 

\bibitem[Su]{Su}
  I. Suprunenko, Conditions on the irreducibility of restrictions of irreducible representations of 
the group $SL(n,K)$ to connected algebraic subgroups, Preprint $\#13$, $(222)$, Inst. Mat. Akad. Nauk
BSSR, $1985$ (in Russian).

\bibitem[Tes]{Tes}
  D. Testerman, Irreducible subgroups of exceptional algebraic groups, {\it  Mem. Amer. Math. Soc.}
{\bf 75} $(1988)$,  no. 390.

\bibitem[Th]{Th}
  J. G. Thompson, Bilinear forms in characteristic $p$ and the Frobenius-Schur indicator, in: {\it Lecture
Notes Math.} {\bf 1185} $(1986)$, pp. $221 - 230$.
   
\bibitem[T1]{T1}
  Pham Huu Tiep, Basic spin representations of $2\SSS_{n}$ and $2\AAA_{n}$
as globally irreducible representations, {\it Archiv Math.} {\bf 64} $(1995)$, $103 - 112$.

\bibitem[T2]{T2}
  Pham Huu Tiep, Finite groups admitting grassmannian $4$-designs, {\it J. Algebra} 
{\bf 306} $(2006)$, $227 - 243$.

\bibitem[TZ1]{TZ1}
  Pham Huu Tiep and A. E. Zalesskii, Minimal characters of the finite classical groups, 
{\it Comm. Algebra} {\bf 24} $(1996)$, $2093 - 2167$.

\bibitem[TZ2]{TZ2}
  Pham Huu Tiep and A. E. Zalesskii, Some characterizations of the Weil
representations of the symplectic and unitary groups, {\it J. Algebra} {\bf 192}
$(1997)$, $130 - 165$.

\bibitem[Wa]{Wa}
  D. B. Wales, Some projective representations of $S_{n}$, {\it J. Algebra} {\bf 61} $(1979)$, $37 - 57$.

\bibitem[Zs]{Zs}
  K. Zsigmondy, Zur Theorie der Potenzreste, {\it Monath. Math. Phys.} {\bf 3} $(1892)$,
$265 - 284$.

\end{thebibliography}
\end{document}